\documentclass[12pt,a4paper]{amsart}
\usepackage{ucs}
\usepackage[utf8x]{inputenc}
\usepackage[all]{xy}
\usepackage[T1]{fontenc}
\usepackage{lmodern}
\usepackage{amsmath, amssymb,amsthm}
\usepackage{fullpage}
\usepackage{color}

\numberwithin{equation}{section}

\newcommand{\C}{\mathbb{C}}

\newcommand{\rig}{\mathrm{rig}}
\newcommand{\red}{\mathrm{red}}
\newcommand{\val}{\mathrm{val}}

\newcommand{\End}{\mathrm{End}}
\newcommand{\Fil}{\mathrm{Fil}}

\newcommand{\GL}{\mathrm{GL}}
\newcommand{\Hom}{\mathrm{Hom}}
\newcommand{\id}{\mathrm{id}}

\newcommand{\Q}{\mathbb{Q}}
\newcommand{\Qbar}{\overline{\mathbb{Q}}}
\newcommand{\rbar}{\overline{r}}
\newcommand{\rec}{\mathrm{rec}}
\newcommand{\Res}{\mathrm{Res}}
\newcommand{\rhobar}{\overline{\rho}}
\newcommand{\Spm}{\mathrm{Sp}}
\newcommand{\Spec}{\mathrm{Spec}}
\newcommand{\diag}{\mathrm{diag}}
\newcommand{\Spf}{\mathrm{Spf}}
\newcommand{\Z}{\mathbb{Z}}
\newcommand{\Gbb}{\mathbb{G}}
\newcommand{\Abb}{\mathbb{A}}
\newcommand{\Ubb}{\mathbb{U}}
\newcommand{\boldB}{\mathbb{B}}
\newcommand{\Xfrak}{\mathfrak{X}}
\newcommand{\Dcal}{\mathcal{D}}
\newcommand{\Ocal}{\mathcal{O}}
\newcommand{\Scal}{\mathcal{S}}
\newcommand{\Tcal}{\mathcal{T}}
\newcommand{\Wcal}{\mathcal{W}}
\newcommand{\Fcal}{\mathcal{F}}
\newcommand{\Rcal}{\mathcal{R}}
\newcommand{\Gcal}{\mathcal{G}}

\newcommand{\dbl}{{\mathchoice{\mbox{\rm [\hspace{-0.15em}[}}
                              {\mbox{\rm [\hspace{-0.15em}[}}
                              {\mbox{\scriptsize\rm [\hspace{-0.15em}[}}
                              {\mbox{\tiny\rm [\hspace{-0.15em}[}}}}
\newcommand{\dbr}{{\mathchoice{\mbox{\rm ]\hspace{-0.15em}]}}
                              {\mbox{\rm ]\hspace{-0.15em}]}}
                              {\mbox{\scriptsize\rm ]\hspace{-0.15em}]}}
                              {\mbox{\tiny\rm ]\hspace{-0.15em}]}}}}

\newtheorem{theo}{Theorem}[section]
\newtheorem{conj}[theo]{Conjecture}
\newtheorem{coro}[theo]{Corollary}
\newtheorem{defi}[theo]{Definition}
\newtheorem{lemm}[theo]{Lemma}
\newtheorem{prop}[theo]{Proposition}
\newtheorem{rema}[theo]{Remark}
\newtheorem{ex}[theo]{Example}

\author{Christophe Breuil, Eugen Hellmann and Benjamin Schraen}
\address{Christophe Breuil\\
C.N.R.S.\\
Laboratoire de Math\'ematiques d'Orsay\\
Universit\'e Paris-Sud\\
Universit\'e Paris-Saclay\\
F-91405 Orsay\\
France\\
Christophe.Breuil@math.u-psud.fr}

\address{Eugen Hellmann\\
Mathematisches Institut\\
Universit\"at Bonn\\
Endenicher Allee 60\\
D-53115 Bonn \\
Germany\\
hellmann@math.uni-bonn.de}

\address{Benjamin Schraen\\
C.N.R.S.\\ 
Laboratoire de Math\'ematiques\\
Universit\'e de Versailles St.~Quen\-tin\\
45 Av. des \'Etats-Unis\\
F-78035 Versailles\\
France\\
benjamin.schraen@uvsq.fr}

\title{Smoothness and classicality on Eigenvarieties}

\thanks{We thank John Bergdall, Laurent Berger, Peter Scholze and Jack Thorne for their answers to our questions. C.~B. and B.~S. are supported by the C.N.R.S. and by the A.N.R. project Th\'eHopaD (ANR-2011-BS01-005) and B.~S. thanks I.H.\'E.S. for its hospitality. E.~H. is supported by SFB-TR 45 of the D.F.G.}

\begin{document}

\begin{abstract}
Let $p$ be a prime number and $f$ an overconvergent $p$-adic automorphic form on a definite unitary group which is split at $p$. Assume that $f$ is of ``classical weight'' and that its Galois representation is crystalline at $p$, then $f$ is conjectured to be a classical automorphic form. We prove new cases of this conjecture in arbitrary dimensions by making crucial use of the patched eigenvariety constructed in \cite{BHS}.
\end{abstract}

\maketitle
\tableofcontents

\section{Introduction}

Let $p$ be a prime number. In this paper we are concerned with classicality of $p$-adic automorphic forms on some unitary groups, i.e.~we are looking for criteria that decide whether a given $p$-adic automorphic form is classical or not. More precisely we work with $p$-adic forms of finite slope, that is, in the context of \emph{eigenvarieties}.

Let $F^+$ be a totally real number field and $F$ be an imaginary quadratic extension of $F^+$. We fix a unitary group $G$ in $n$ variables over $F^+$ which splits over $F$ and over all $p$-adic places of $F^+$, and which is compact at all infinite places of $F^+$. Associated to such a group $G$ (and the choice of a tame level, i.e.~a compact open subgroup of $G(\Abb_{F^+}^{p\infty})$) there is a nice Hecke eigenvariety which is an equidimensional rigid analytic space of dimension $n[F^+:\Q]$, see e.g. \cite{Cheeig}, \cite{BelChe} or \cite{Emerton}. One may view a  {\it $p$-adic overconvergent eigenform of finite slope}, or simply overconvergent form, as a point $x$ of such an eigenvariety and one can associate to each overconvergent form a continuous semi-simple representation $\rho_x:{\rm Gal}(\overline F/F)\rightarrow \GL_n(\overline \Q_p)$ which is unramified outside a finite set of places of $F$ and which is trianguline in the sense of \cite{Co} at all places of $F$ dividing $p$ (\cite{KPX}).

A natural expectation deduced from the Langlands and Fontaine-Mazur conjectures is that, if $\rho_x$ is de Rham (in the sense of Fontaine) at places of $F$ dividing $p$, then $x$ is a {\it classical} automorphic form (see Definition \ref{defclass} and Proposition \ref{compaclas} for the precise definition). However, the naive version of this statement fails for two reasons: (1) a classical automorphic form for $G(\Abb_{F^+})$ can only give Galois representations which have distinct Hodge-Tate weights (in each direction $F\hookrightarrow \overline\Q_p$) and (2) the phenomenon of {\it companion} forms shows that there can exist classical and non-classical forms giving the same Galois representation. However, we can resolve (1) by requiring $\rho_x$ to have distinct Hodge-Tate weights and (2) by requiring $x$ to be of ``classical'' (or dominant) weight. In fact, since the Hodge-Tate weights of $\rho_x$ are related to the weight of $x$, requiring the latter automatically implies the former, once $\rho_x$ is assumed to be de Rham. As a conclusion, it seems reasonable to expect that any overconvergent form $x$ of classical weight such that $\rho_x$ is de Rham at places of $F$ dividing $p$ is a classical automorphic form (see Conjecture $\ref{classiconj}$ and Remark \ref{extra}).

Such a classicality theorem is due to Kisin (\cite{Kisinoverconvergent}) in the context of Coleman-Mazur's eigencurve, i.e.~in the slightly different setting of ${\rm GL}_2/\Q$. Note that, at the time of \cite{Kisinoverconvergent}, the notion of a trianguline representation was not available, and in fact \cite{Kisinoverconvergent} inspired Colmez to define trianguline representations (\cite{Co}). 

In the present paper we prove new cases of this classicality conjecture (in the above unitary setting). In particular we are able to deal with cases where the overconvergent form $x$ is {\it critical}. Throughout, we assume that $\rho_x$ is crystalline at $p$-adic places. Essentially the same proof should work if $\rho_x$ is only assumed crystabelline, but the crystalline assumption significantly simplifies the notation.

To state our main results, we fix an overconvergent form $x$ of classical weight such that $\rho_x$ is crystalline at all places dividing $p$. Such an overconvergent form can be described by a pair $(\rho_x,\delta_x)$, where $\rho_x$ is as above and $\delta_x=(\delta_{x,v})_{v\in S_p}$ is a locally $\Q_p$-analytic character of the diagonal torus of $G(F^+\otimes_{\Q}\Q_p)\cong \prod_{v\in S_p}\GL_n(F^+_v).$  Here $S_p$ denotes the set of places of $F^+$ dividing $p$. There are nontrivial relations between $\rho_{x,v}:=\rho\vert_{{\rm Gal}(\overline {F^+_v}/F^+_v)}$ and $\delta_{x,v}$, in particular the character $\delta_{x,v}$ defines an ordering of the eigenvalues of the crystalline Frobenius on $D_{\rm cris}(\rho_{x,v})$. If we assume that these Frobenius eigenvalues are pairwise distinct, then this ordering defines a Frobenius stable flag in $D_{\rm cris}(\rho_{x,v})$. We can therefore associate to $x$ for each $v\in S_p$ a permutation $w_{x,v}$ that gives the relative position of this flag with respect to the Hodge filtration on $D_{\rm cris}(\rho_{x,v})$, see \S\ref{weyl} (where we rather use another equivalent definition of $w_{x,v}$ in terms of triangulations).  Following \cite[\S2.4.3]{BelChe} we say that $x$ is {\it noncritical} if, for each $v$, the permutation $w_{x,v}$ is trivial. The invariant $(w_{x,v})_{v\in S_p}$ can thus be seen as ``measuring'' the criticality of $x$. 

In the statement of our main theorem, we need to assume a certain number of Working Hypotheses (basically the combined hypotheses of all the papers we use). We denote by $\rhobar_x$ the mod $p$ semi-simplification of $\rho_x$. These Working Hypotheses are:
\begin{enumerate}
\item[(i)] The field $F$ is unramified over $F^+$ and $G$ is quasi-split at all finite places of $F^+$;
\item[(ii)]the tame level of $x$ is hyperspecial at all finite places of $F^+$ inert in $F$;
\item[(iii)]$\rhobar_x({\rm Gal}(\overline F/F(\zeta_p))$ is adequate (\cite{Thorne}) and $\zeta_p\notin \overline F^{\ker({\rm ad}\rhobar_x)}$;
\item[(iv)]the eigenvalues of $\varphi$ on $D_{\rm cris}(\rho_{x,v})$ are sufficiently generic for any $v\in S_p$ (Definition \ref{veryreg}).
\end{enumerate}

Our main theorem is:

\begin{theo}[Cor. \ref{mainclassic}]\label{mainintro}
Let $p>2$ and assume that the group $G$  and the tame level satisfy (i) and (ii).
Let $x$ be an overconvergent form of classical weight such that $\rho_x$ is crystalline and satisfies (iii) and (iv). If $w_{x,v}$ is a product of distinct simple reflections for all places $v$ of $F^+$ dividing $p$, then $x$ is classical.
\end{theo}

Note that the assumption on the $w_{x,v}$ in Theorem \ref{mainintro} is empty when $n=2$, and already this $n=2$ case was not previously known (to the knowledge of the authors). The noncritical case of Theorem \ref{mainintro}, i.e. the special case where all the $w_{x,v}$ are trivial, is already known and due to Chenevier (\cite[Prop.4.2]{Chenevierfern}). Thus the main novelty, and difficulty, in Theorem \ref{mainintro} is that it deals with possibly critical (though not too critical) points.

In fact we give a more general classicality criterion and prove that it is satisfied under the assumptions of Theorem $\ref{mainintro}$. This criterion is formulated in terms of the rigid analytic space of trianguline representations  $X_{\rm tri}^\square(\rhobar_{x,v})$ defined in \cite{HellmannFS} and \cite[\S2.2]{BHS}. For every $v\in S_p$ there is a canonical morphism from the eigenvariety to $X_{\rm tri}^\square(\rhobar_{x,v})$.

\begin{theo}[Cor. \ref{classicalitycrit}, Rem. \ref{uniqueirredcompo}]\label{mainintro2}
Let $p>2$ and assume that the group $G$  and the tame level satisfy (i) and (ii).
Let $x$ be an overconvergent form of classical weight such that $\rho_x$ is crystalline and satisfies (iii) and (iv). If for any $v\in S_p$ the image $x_v$ of $x$ in $X_{\rm tri}^\square(\rhobar_{x,v})$ is contained in a unique irreducible component of $X_{\rm tri}^\square(\rhobar_{x,v})$, then $x$ is classical. 
\end{theo}

According to this theorem we need to understand the local geometry of the space $X_{\rm tri}^\square(\rhobar_{x,v})$ at $x_v$. It turns out that much of this local geometry is controlled by the Weyl group element $w_{x,v}$ associated to $x$ which only depends on the image $x_v$ of $x$ in $X_{\rm tri}^\square(\rhobar_{x,v})$. For $v\in S_p$ denote by $\lg(w_{x,v})$ the length of the permutation $w_{x,v}$ and by $d_{x,v}$ the rank of the $\Z$-module generated by $w_{x,v}(\alpha)-\alpha$, as $\alpha$ ranges over the roots of $(\Res_{F_v^+/\Q_p}\GL_n)\times_{\Q_p} \overline \Q_p\cong \prod_{\tau:\, F_v^+\hookrightarrow \overline\Q_p}\GL_n$. Then $d_{x,v}\leq \lg(w_{x,v})$, with equality if and only if $w_{x,v}$ is a product of distinct simple reflections (Lemma \ref{coxeter}).

\begin{theo}[Th. \ref{upperbound}, Cor. \ref{Xtrismooth}]\label{mainlocal}
Let $v\in S_p$ and let $X\subseteq X_{\rm tri}^\square(\overline \rho_{x,v})$ be a union of irreducible components that contain $x_v$ and satisfy the {\it accumulation property} of Definition \ref{accu} at $x_v$. Then
$$\dim T_{X,x_v}\leq \dim X+ \lg(w_{x,v})-d_{x,v}=\dim X_{\rm tri}^\square(\overline \rho_{x,v}) + \lg(w_{x,v})-d_{x,v},$$
where $T_{X,x_v}$ is the tangent space to $X$ at $x_v$. In particular $X$ is smooth at $x_v$ when $w_{x,v}$ is a product of distinct simple reflections.
\end{theo}

The accumulation condition in Theorem \ref{mainlocal} actually prevents us from directly applying it to $X=X_{\rm tri}^\square(\overline \rho_{x,v})$ and thus directly deducing Theorem \ref{mainintro} from Theorem \ref{mainintro2}. Hence we have to sharpen Theorem \ref{mainintro2}, see Theorem \ref{classicalitycrit}.

Assuming the classical modularity lifting conjectures for $\rhobar_x$ (in all weights with trivial inertial type), there is a certain union $\widetilde X_{\rm tri}^\square(\rhobar_{x,v})$ of irreducible components of $X_{\rm tri}^\square(\rhobar_{x,v})$ such that $\prod_{v\in S_p}\widetilde X_{\rm tri}^\square(\rhobar_{x,v})$ is (essentially) described by the patched eigenvariety $X_p(\rhobar_x)$ defined in \cite{BHS} (see Remark \ref{conjvariant}). In the last section of the paper (\S\ref{modularity}), we prove (assuming modularity lifting conjectures) that the inequality in Theorem \ref{mainlocal} for $X=\widetilde X_{\rm tri}^\square(\rhobar_{x,v})$ is an {\it equality} for all $v\in S_p$,
\begin{equation}\label{mainequality}
\dim T_{\widetilde X_{\rm tri}^\square(\rhobar_{x,v}),x_v}=\dim X_{\rm tri}^\square(\overline \rho_{x,v}) + \lg(w_{x,v})-d_{x,v}\ \ \ \ \ \ {\rm (assuming\ modularity),}
\end{equation}
see Corollary \ref{bonnedim}. The precise computation (\ref{mainequality}) of the dimension of the tangent space is intimately related to (and uses in its proof) the existence of many \emph{companion points} on the patched eigenvariety $X_p(\rhobar_x)$. These companion points are provided by the following unconditional theorem, which is of independent interest.

\begin{theo}[Th. \ref{intercompanion}]\label{companionptsintro}
Let $y=((\rho_v)_{v\in S_p},\epsilon)$ be a point on $X_p(\rhobar_x)$. Let $T$ be the diagonal torus in $\GL_n$ and let $\delta$ be a locally $\Q_p$-analytic character of $T(F^+\otimes_{\Q}\Q_p)$ such that $\epsilon\delta^{-1}$ is an algebraic character of $T(F^+\otimes_{\Q}\Q_p)$ and such that $\epsilon$ is strongly linked to $\delta$ in the sense of \cite[\S5.1]{HumBGG} (as modules over the Lie algebra of $T(F^+\otimes_{\Q}\Q_p)$). Then $((\rho_v)_{v\in S_p},\delta)$ is also a point on $X_p(\rhobar_x)$.
\end{theo}

We also prove that the equality (\ref{mainequality}) for all $v\in S_p$ (and thus the modularity lifting conjectures) imply that the initial Hecke eigenvariety is itself {\it singular} at $x$ as soon as the Weyl element $w_{x,v}$ is {\it not} a product of distinct simple reflections for some $v\in S_p$, see Corollary \ref{singhecke}.

Let us now outline the strategy of the proofs of Theorems \ref{mainintro2} and \ref{mainlocal}.

The proof of Theorem \ref{mainlocal} crucially uses results of Bergdall (\cite{Bergdall}), together with a fine analysis of the various conditions on the infinitesimal deformations of $\rho_{x,v}$ carried by vectors in $T_{X,x_v}$, see \S\ref{wedgesection} and \S\ref{endofproof}. Very recently, Bergdall proved an analogous bound for the dimension of the tangent space of the initial Hecke eigenvariety at $x$ assuming standard vanishing conjectures on certain Selmer groups (\cite{Bergdraft}).
 
The proof of Theorem \ref{mainintro2} makes use of the patched eigenvariety $X_p(\rhobar_x)$ constructed in \cite{BHS} by applying Emerton's construction of eigenvarieties \cite{Emerton} to the locally analytic vectors of the patched Banach $G(F^+\otimes_{\Q}\Q_p)$-representation $\Pi_\infty$ of \cite{CEGGPS}. As usual with the patching philosophy, the space $X_p(\rhobar_x)$ can be related to another geometric object which has a much more local flavour, namely the space $X_{\rm tri}^\square(\overline \rho_{x,p}):=\prod_{v\in S_p}X_{\rm tri}^\square(\overline \rho_{x,v})$ of trianguline representations. More precisely, by \cite[Th.3.20]{BHS} there is a Zariski closed embedding:
\begin{equation}\label{maininclusion}
X_p(\overline\rho_x)\hookrightarrow \Xfrak_{\rhobar_x^p}\times \Ubb^g \times X_{\rm tri}^\square(\overline \rho_{x,p}),
\end{equation}
identifying the source with a union of irreducible components of the target.
Here $\Ubb^g$ is an open polydisc (related to the patching variables) and $\Xfrak_{\rhobar_x^p}$ is the rigid analytic generic fiber of the framed deformation space of $\rhobar_x$ at all the ``bad'' places prime to $p$. Moreover the Hecke eigenvariety containing $x$ can be embedded into the patched eigenvariety $X_p(\overline\rho_x)$ (see \cite[Th.4.2]{BHS}). As previously, we denote by $x_v$ the image of $x$ in $X_{\rm tri}^\square(\overline \rho_{x,v})$ via (\ref{maininclusion}).

For $v\in S_p$ let us write ${\bf k}_{v}$ for the set of labelled Hodge-Tate weights of $\rho_{x,v}$, and $R_{\overline \rho_{x,v}}^{\square,{\bf k}_v{\rm -cr}}$ for the quotient defined in \cite{Kisindef} of the framed deformation ring of $\overline \rho_{x,v}$ parametrizing crystalline deformations of $\overline \rho_{x,v}$ of Hodge-Tate weight ${\bf k}_v$, and $\Xfrak_{\overline \rho_{x,v}}^{\square,{\bf k}_v{\rm -cr}}$ for the rigid space $(\Spf\, R_{\overline \rho_{x,v}}^{\square,{\bf k}_v{\rm -cr}})^{\rm rig}$. We relate $\Xfrak_{\overline \rho_{x,v}}^{\square,{\bf k}_v{\rm -cr}}$ to $X_{\rm tri}^\square(\overline \rho_{x,v})$ by introducing a third rigid analytic space $\widetilde \Xfrak_{\overline \rho_{x,v}}^{\square,{\bf k}_v{\rm -cr}}$ finite over $\Xfrak_{\overline \rho_{x,v}}^{\square,{\bf k}_v{\rm -cr}}$ parametrizing crystalline deformations $\rho_v$ of $\overline \rho_{x,v}$ of Hodge-Tate weights ${\bf k}_v$ {\it together with an ordering} of the Frobenius eigenvalues on $D_{\rm cris}(\rho_{v})$, see \S\ref{variant} for a precise definition. The space $\widetilde \Xfrak_{\overline \rho_{x,v}}^{\square,{\bf k}_v{\rm -cr}}$ naturally embeds into $X_{\rm tri}^\square(\overline \rho_{x,v})$ and contains the point $x_v$ (and is smooth at $x_v$). We prove that there is a unique irreducible component $Z_{{\rm tri}}(x_v)$ of $X_{\rm tri}^\square(\overline \rho_{x,v})$ containing the unique irreducible component of $\widetilde \Xfrak_{\overline \rho_{x,v}}^{\square,{\bf k}_v{\rm -cr}}$ passing through $x_v$ (Corollary \ref{defnZtri(x)}). Let $Z_{\rm tri}(x):=\prod_{v\in S_p}Z_{\rm tri}(x_v)$, which is thus an irreducible component of $X_{\rm tri}^\square(\overline \rho_{x,p})$ containing $x$. Then Theorem \ref{mainintro2} easily follows from the following theorem (see (i) of Remark \ref{uniqueirredcompo}):

\begin{theo}[Th. \ref{classicalitycrit}]\label{mainclasscrit}
Assume that $\Xfrak_{\rhobar_x^p}\times\Ubb^g\times Z_{\rm tri}(x)\subseteq X_p(\overline\rho_x)$ via (\ref{maininclusion}). Then the point $x$ is classical.
\end{theo}

Let us finally sketch the proof of Theorem \ref{mainclasscrit} (in fact, for the same reason as above, we have to sharpen Theorem \ref{mainclasscrit}, see Theorem \ref{classicalitycrit}). Let $R_\infty$ be the usual patched deformation ring of $\rhobar_x$, there is a canonical morphism of rigid spaces $X_p(\overline\rho_x)\longrightarrow {\mathfrak X}_\infty:=(\Spf\, R_\infty)^{\rm rig}$. Let $L(\lambda)$ be the finite dimensional algebraic representation of $G(F^+\otimes_{\Q}\Q_p)$ associated (via the usual shift) to the Hodge-Tate weights $({\bf k}_v)_{v\in S_p}$. Proving classicality of $x$ turns out to be equivalent to proving that the image of $x$ in ${\mathfrak X}_\infty$ is in the support of the $R_\infty$-module $\Pi_\infty(\lambda)'$ which is the continuous dual of:
\begin{equation*}
\Pi_\infty(\lambda):=\Hom_{{\prod}_{v\in S_p}\GL_n(\mathcal{O}_{F_{\tilde{v}}})}\big(L(\lambda),\Pi_\infty\big).
\end{equation*}
By \cite[Lem.4.17]{CEGGPS}, the $R_\infty$-module $\Pi_\infty(\lambda)'$ is essentially a Taylor-Wiles-Kisin ``usual'' patched module for the trivial inertial type and the Hodge-Tate weights $({\bf k}_v)_{v\in S_p}$. Forgetting the factors $\Xfrak_{\rhobar_x^p}$ and $\Ubb^g$ which appear in ${\mathfrak X}_\infty$, its support is a union of irreducible components of the smooth rigid space $\prod_{v\in S_p}\Xfrak_{\overline \rho_{x,v}}^{\square,{\bf k}_v{\rm -cr}}$. It is thus enough to prove that the unique irreducible component $Z_{\rm cris}(\rho_x)$ of $\prod_{v\in S_p}\Xfrak_{\overline \rho_{x,v}}^{\square,{\bf k}_v{\rm -cr}}$ passing through $(\rho_{x,v})_{v\in S_p}$ contains a point which is in the support of $\Pi_\infty(\lambda)'$. But it is easy to find a point $y$ in $Z_{\rm tri}(x)$ sufficiently close to $x$ such that $(\rho_{y,v})_{v\in S_p}\in Z_{\rm cris}(\rho_x)$ (in particular $\rho_{y,v}$ is crystalline of the same Hodge-Tate weights as $\rho_{x,v}$) {\it and} moreover $\rho_{y,v}$ is {\it generic} in the sense of \cite[Def.2.8]{BHS} for all $v\in S_p$. The assumption in Theorem \ref{mainclasscrit} implies $y\in X_p(\overline\rho_x)$ and it is now not difficult to prove that such a generic crystalline point of $X_p(\overline\rho_x)$ is always classical, i.e. is in the support of $\Pi_\infty(\lambda)'$.

We end this introduction with the main notation of the paper.

If $K$ is a finite extension of $\Q_p$ we denote by $\mathcal{G}_K$ the absolute Galois group $\mathrm{Gal}(\overline{K}/K)$ and by $\Gamma_K$ the Galois group $\mathrm{Gal}(K(\zeta_{p^n},n\geq 1)/K)$ where $(\zeta_{p^n})_{n\geq 1}$ is a compatible system of primitive $p^n$-th roots of $1$ in $\overline K$. We normalize the reciprocity map $\rec_K:\, K^\times\rightarrow \mathcal{G}_K^{\rm ab}$ of local class field theory so that the image of a uniformizer of $K$ is a geometric Frobenius element. We denote by $\varepsilon$ the $p$-adic cyclotomic character and recall that its Hodge-Tate weight is $1$.

For $a\in L^\times$ (where $L$ is any finite extension of $K$) we denote by ${\rm unr}(a)$ the unramified character of $\mathcal{G}_K$, or equivalently of $\mathcal{G}_K^{\rm ab}$ or $K^\times$, sending to $a$ (the image by $\rec_K$ of) a uniformizer of $K$. For $z\in L$, we let $\vert z\vert_K:=p^{-[K:\Q_p]{\val}(z)}$ where ${\val}(p)=1$. We let $K_0\subseteq K$ be the maximal unramified subfield (we thus have $(\vert \ \vert_K)\vert_{K^\times}={\rm unr}(p^{-[K_0:\Q_p]})={\rm unr}(q^{-1})$ where $q$ is the cardinality of the residue field of $K$). 

If $X=\Spm\,A$ is an affinoid space, we write $\mathcal{R}_{A,K}$ for the Robba ring associated to $K$ with $A$-coefficients (see \cite[Def.6.2.1]{KPX} though our notation is slightly different). Given a continuous character $\delta: K^\times \rightarrow A^\times$ we write $\Rcal_{A,K}(\delta)$ for the rank one $(\varphi,\Gamma_K)$-module on $\Spm\,A$ defined by $\delta$, see \cite[Construction 6.2.4]{KPX}. If $X$ is a rigid analytic space over $L$ (a finite extension of $\Q_p$) and $x$ is a point on $X$, we denote by $k(x)$ the residue field of $x$ (a finite extension of $L$), so that we have $x\in X(k(x))$. If $X$ and $Y$ are two rigid analytic spaces over $L$, we often write $X\times Y$ instead of $X\times_{\Spm\,L}Y$.

If $X$ is a ``geometric object over $\Q_p$'' (i.e. a rigid space, a scheme, an algebraic group, etc.), we denote by $X_K$ its base change to $K$ (for instance if $X$ is the algebraic group $\GL_n$ we write $\GL_{n,K}$). If $H$ is an abelian $p$-adic Lie group, we let $\widehat{H}$ be the rigid analytic space over $\Q_p$ which represents the functor mapping an affinoid space $X=\Spm\,A$ to the group $\Hom_{\rm cont}(H,A^\times)$ of continous group homomorphisms (or equivalently locally $\Q_p$-analytic group homomorphisms) $H\rightarrow A^\times$. Finally, if $M$ is an $R$-module and $I\subseteq R$ an ideal, we denote by $M[I]\subseteq M$ the submodule of elements killed by $I$, and if $S$ is any finite set, we denote by $\vert S\vert$ its cardinality.

\section{Crystalline points on the trianguline variety}\label{localpart1}

We give several important definitions and results, including the key local statement bounding the dimension of some tangent spaces on the trianguline variety (Theorem \ref{upperbound}).

\subsection{Recollections}\label{begin}

We review some notation and definitions related to the trianguline variety.

We fix two finite extensions $K$ and $L$ of $\Q_p$ such that:
$$\vert\Hom(K,L)\vert=[K:\Q_p]$$
and denote by $\mathcal{O}_K$, ${\mathcal O}_L$ their respective rings of integers. We fix a uniformizer $\varpi_K\in \mathcal{O}_K$ and denote by $k_L$ the residue field of ${\mathcal O}_L$. We let $\mathcal{T}:=\widehat{K^\times}$ and $\mathcal{W}:=\widehat{\mathcal{O}_K^\times}$. The restriction of characters to $\mathcal{O}_K^\times$ induces projections $\mathcal{T}\twoheadrightarrow \mathcal{W}$ and $\mathcal{T}_L\twoheadrightarrow \mathcal{W}_L$. 
If ${\bf k}:=(k_\tau)_{\tau:\, K\hookrightarrow L}\in\Z^{\Hom(K,L)}$, we denote by $z^{\bf k}\in \mathcal{T}(L)$ the character:
\begin{eqnarray}\label{charc}
z\longmapsto\prod_{\tau\in\Hom(K,L)}\tau(z)^{k_\tau}
\end{eqnarray}
where $z \in K^\times$. For ${\bf k}=(k_{\tau,i})_{1\leq i\leq n,\tau:\, K\hookrightarrow L}\in(\mathbb{Z}^n)^{\Hom(K,L)}$, we denote by $\delta_{\bf k}\in \mathcal{T}^n(L)$ the character:
$$(z_1,\dots,z_n)\longmapsto\prod_{\stackrel{1\leq i\leq n}{\tau:\, K\hookrightarrow L}}\tau(z_i)^{k_{\tau,i}}$$
where $(z_1,\dots,z_n)\in (K^\times)^n$. We also denote by $\delta_{\bf k}$ its image in $\mathcal{W}^n(L)$ (i.e. its restriction to $({\mathcal O}_K^\times)^n$). We say that a point $\delta\in \mathcal{W}^n_L$ is {\it algebraic} if $\delta=\delta_{\bf k}$ for some ${\bf k}=(k_{\tau,i})_{1\leq i\leq n,\tau:\, K\hookrightarrow L}\in(\mathbb{Z}^n)^{\Hom(K,L)}$. We say that an algebraic $\delta=\delta_{\bf k}$ is {\it dominant} (resp. {\it strictly dominant}) if moreover $k_{\tau,i}\geq k_{\tau,i+1}$ (resp. $k_{\tau,i}> k_{\tau,i+1}$) for $i\in \{1,\dots,n-1\}$ and $\tau\in \Hom(K,L)$.

We write $\Tcal_{\rm reg}\subset \Tcal_L$ for the Zariski-open complement of the $L$-valued points $z^{-\bf k}, \vert z\vert_K z^{{\bf k}+{\bf 1}}$, with ${\bf k}=(k_\tau)_{\tau:K\hookrightarrow L}\in \Z_{\geq 0}^{\Hom(K,L)}$. We write $\Tcal_{\rm reg}^n$ for the Zariski-open subset of characters $(\delta_1,\dots,\delta_n)$ such that $\delta_i/\delta_j\in \Tcal_{\rm reg}$ for $i\neq j$. 

We fix a continuous representation $\rbar:\mathcal{G}_K\rightarrow \GL_n(k_L)$ and let $R_{\rbar}^{\square}$ be the  framed local deformation ring of $\rbar$ (a local complete noetherian $\mathcal{O}_L$-algebra of residue field $k_L$). We write $\mathfrak{X}_{\rbar}^\square:=(\Spf\, R_{\rbar}^\square)^{\rig}$ for the rigid analytic space over $L$ associated to the formal scheme $\Spf\, R_{\rbar}^\square$. Recall that a representation $r$ of $\Gcal_K$ on a finite dimension $L$-vector space is called \emph{trianguline of parameter} $\delta=(\delta_1,\dots,\delta_n)$ if the $(\varphi,\Gamma_K)$-module $D_{\rig}(r)$ over $\Rcal_{L,K}$ associated to $r$ admits an increasing filtration $\Fil_\bullet$ by sub-$(\varphi,\Gamma_K)$-modules over $\Rcal_{L,K}$ such that the graded piece $\Fil_i/\Fil_{i-1}$ is isomorphic to $\Rcal_{L,K}(\delta_i)$. 
We let $X_{\rm tri}^\square(\rbar)$ be the associated framed trianguline variety, see \cite[\S2.2]{BHS} and \cite{HellmannFS}. Recall that $X_{\rm tri}^\square(\rbar)$ is the reduced rigid analytic space over $L$ which is the Zariski closure in $\mathfrak{X}_{\rbar}^\square\times \mathcal{T}^n_L$ of:
\begin{equation}\label{ucris}
U_{\rm tri}^\square(\rbar):=\{{\rm points}\ (r,\delta)\ {\rm in}\ \Xfrak^\square_{\rbar}\times\Tcal^n_{\rm reg}\ {\rm such\ that}\ r\ \text{is trianguline of parameter}\ \delta\}
\end{equation}
(the space $U_{\rm tri}^\square(\rbar)$ is denoted $U_{\rm tri}^\square(\rbar)^{\rm reg}$ in \cite[\S2.2]{BHS}). The rigid space $X_{\rm tri}^\square(\rbar)$ is reduced equidimensional of dimension $n^2+[K:\Q_p]\frac{n(n+1)}{2}$ and its subset $U_{\rm tri}^\square(\rbar)\subset X_{\rm tri}^\square(\rbar)$ turns out to be Zariski-open, see \cite[Th.2.6]{BHS}. Moreover by {\it loc.~cit.}~the rigid variety $U_{\rm tri}^\square(\rbar)$ is smooth over $L$ and equidimensional, hence there is a bijection between the set of connected components of $U_{\rm tri}^\square(\rbar)$ and the set of irreducible components of $X_{\rm tri}^\square(\rbar)$.

We denote by $\omega:X_{\rm tri}^\square(\rbar)\rightarrow \mathcal{W}^n_L$ the composition $X_{\rm tri}^\square(\rbar)\hookrightarrow \mathfrak{X}_{\rbar}^\square\times \mathcal{T}^n_L\twoheadrightarrow \mathcal{T}^n_L \twoheadrightarrow \mathcal{W}^n_L$. If $x$ is a point of $X_{\rm tri}^\square(\rbar)$, we write $x=(r,\delta)$ where $r\in \mathfrak{X}_{\rbar}^\square$ and $\delta=(\delta_1,\dots,\delta_n)\in \mathcal{T}_L^n$. We say that a point $x=(r,\delta)\in X_{\rm tri}^\square(\rbar)$ is \emph{crystalline} if $r$ is a crystalline representation of $\mathcal{G}_K$. 

\begin{lemm}\label{paramofcrystpt}
Let $x=(r,\delta)\in X_{\rm tri}^\square(\rbar)$ be a crystalline point. Then for $i\in \{1,\dots,n\}$ there exist ${\bf k}_i=(k_{\tau,i})_{\tau:\, K\hookrightarrow L}\in\mathbb{Z}^{\Hom(K,L)}$ and $\varphi_i\in k(x)^\times$ such that:
$$\delta_i=z^{{\bf k}_i}{\rm unr}(\varphi_i).$$
Moreover the $(k_{\tau,i})_{i,\tau}$ are the labelled Hodge-Tate weights of $r$ and the $\varphi_i$ are the eigenvalues of the geometric Frobenius on the (unramified) Weil-Deligne representation ${\rm WD}(r)$ associated to $r$ (cf. \cite{Fo}). 
\end{lemm}
\begin{proof}
The fact that the $(k_{\tau,i})_{i,\tau}$ are the Hodge-Tate weights of $r$ follows for instance from \cite[Prop.2.9]{BHS}.
By \cite[Th.6.3.13]{KPX} there exist for each $i$ a continuous character $\delta'_i:K^\times \rightarrow k(x)^\times$ such that $r$ is trianguline of parameter $\delta':=(\delta'_1,\dots,\delta'_n)$ and such that $\delta_i/\delta'_i$ is an algebraic character of $K^\times$ (i.e. of the form $z^{{\bf k}}$ for some ${\bf k}\in\Z^{\Hom(K,L)}$). It thus suffices to prove that each $\delta'_i$ is of the form $z^{{\bf k}'_i}{\rm unr}(\varphi_i)$ for some ${\bf k}'_i\in\Z^{\Hom(K,L)}$ where the $\varphi_i\in k(x)^\times$ are the eigenvalues of the geometric Frobenius on ${\rm WD}(r)$, or equivalently (using the definition of ${\rm WD}(r)$) are the eigenvalues of the linearized Frobenius $\varphi^{[K_0:\Q_p]}$ on the $K_0\otimes_{\Q_p}k(x)$-module $D_{\rm cris}(r):=(B_{\rm cris}\otimes_{\Q_p}r)^{\mathcal{G}_K}$. By \cite[Th.3.6]{Berger} there is an isomorphism (recall $t$ is ``Fontaine's $2i\pi$''):
\begin{eqnarray}\label{berger}
D_{\rm cris}(r)\cong D_{\rig}(r)[\tfrac{1}{t}]^{\Gamma_K},
\end{eqnarray}
and a triangulation  $\Fil_\bullet$ of $D_{\rig}(r)$ with graded pieces giving the parameter $\delta'$ induces a complete $\varphi$-stable filtration $\Fcal_\bullet$ on $D_{\rm cris}(r)$ such that $\Fcal_i/\Fcal_{i-1}$ is the filtered $\varphi$-module associated to $\Rcal_{L,K}(\delta'_i)=\Fil_i/\Fil_{i-1}$ by the same recipee as (\ref{berger}). It follows from this and from \cite[Example 6.2.6(3)]{KPX} that $\delta'_i$ is of the form $z^{{\bf k}'_i}{\rm unr}(a)$ where $a\in k(x)^\times$ is the unique element such that $\varphi^{[K_0:\Q_p]}$ acts on the underlying $\varphi$-module of $\Fcal_i/\Fcal_{i-1}$ by multiplication by $1\otimes a\in K_0\otimes_{\Q_p} k(x)$. This finishes the proof.
\end{proof}

Note that Lemma \ref{paramofcrystpt} implies that if $x=(r,\delta)\in X_{\rm tri}^\square(\rbar)$ is a crystalline point, then $\omega(x)$ is algebraic ($=\delta_{\bf k}$ for ${\bf k}:=(k_{\tau,i})_{1\leq i\leq n,\tau:\, K\hookrightarrow L}$ where the $k_{\tau,i}$ are as in Lemma \ref{paramofcrystpt}). We say that a point $x=(r,\delta)\in X_{\rm tri}^\square(\rbar)$ such that $\omega(x)$ is algebraic is {\it dominant} (resp. {\it strictly dominant}) if $\omega(x)$ is dominant (resp. strictly dominant).

\subsection{A variant of the crystalline deformation space}\label{variant}

We define a certain irreducible component $Z_{{\rm tri},U}(x)$ of a sufficiently small open neighbouhood $U\subseteq X_{\rm tri}^\square(\rbar)$ containing a crystalline strictly dominant point $x$ (Corollary \ref{defnZtri(x)}).

We fix ${\bf k}=(k_{\tau,i})_{1\leq i\leq n,\tau:\, K\hookrightarrow L}\in (\Z^n)^{\Hom(K,L)}$ such that $k_{\tau,i}> k_{\tau,i+1}$ for all $i,\tau$ and write $R_{\overline r}^{\square,{\bf k}{\rm -cr}}$ for the crystalline deformation ring of $\overline r$ with Hodge-Tate weights ${\bf k}$, i.e.~the reduced and $\Z_p$-flat quotient of $R^\square_{\overline r}$ such that, for any finite extension $L'$ of $L$, a morphism $x:\Spec\, L'\rightarrow \Spec\, R_{\rbar}^\square$ factors through $\Spec\, R_{\overline r}^{\square, {\bf k}{\rm -cr}}$ if and only if the representation $\Gcal_K\rightarrow \GL_n(L')$ defined by $x$ is crystalline with labelled Hodge-Tate weights $(k_{\tau,i})_{1\leq i\leq n,\tau:\, K\hookrightarrow L}$. That this ring exists is the main result of \cite{Kisindef}.
We write $\Xfrak_{\overline r}^{\square,{\bf k}{\rm -cr}}$ for the rigid analytic space associated to $\Spf\, R_{\overline r}^{\square,{\bf k}{\rm -cr}}$. By \cite{Kisindef}, it is smooth over $L$.

Let $\tilde r:{\mathcal G}_K\rightarrow \GL_n(R_{\overline r}^{\square,{\bf k}{\rm -cr}})$ be the corresponding universal deformation. By \cite[Th.2.5.5]{Kisindef} or \cite[Cor.6.3.3]{BergerColmez} there is a coherent $K_0\otimes_{\Q_p}\Ocal_{\!\Xfrak_{\overline r}^{\square,{\bf k}{\rm -cr}}}$-module $\Dcal$ that is locally on $\Xfrak_{\overline r}^{\square,{\bf k}{\rm -cr}}$ free over $K_0\otimes_{\Q_p}\Ocal_{\!\Xfrak_{\overline r}^{\square,{\bf k}{\rm -cr}}}$ together with a $\varphi\otimes \id$-linear automorphism $\Phi_{\rm cris}$ such that:
$$(\Dcal,\Phi_{\rm cris})\otimes_{\Ocal_{\!\Xfrak_{\overline r}^{\square,{\bf k}{\rm -cr}}}} k(x)\cong D_{\rm cris}\big(\tilde r\otimes_{R_{\overline r}^{\square,{\bf k}{\rm -cr}}} k(x)\big)$$
for all $x\in \Xfrak_{\overline r}^{\square,{\bf k}{\rm -cr}}$. Fixing an embedding $\tau_0:K_0\hookrightarrow L$ we can define the associated family of Weil-Deligne representations:
$$({\rm WD}(\tilde r),\Phi):=\big(\Dcal\otimes_{K_0\otimes_{\Q_p}\Ocal_{\!\Xfrak_{\overline r}^{\square,{\bf k}{\rm -cr}}},\tau_0\otimes\id}\Ocal_{\!\Xfrak_{\overline r}^{\square,{\bf k}{\rm -cr}}},\Phi_{\rm cris}^{[K_0:\Q_p]}\otimes\id\big)$$
on $\Xfrak_{\overline r}^{\square,{\bf k}{\rm -cr}}$ whose isomorphism class does not depend on the choice of the embedding $\tau_0$. 

Let $T^\rig\cong ({\mathbb G}_{{\rm m}}^\rig)^{n}$ be the rigid analytic space over $\Q_p$ associated to the diagonal torus $T\subset \GL_{n}$ and let $\Scal_n$ be the Weyl group of $(\GL_{n},T)$ acting on $T$, and thus on $T^\rig$, in the usual way. Recall that the map:
$$\diag(\varphi_1,\varphi_2,\dots,\varphi_n)\mapsto {\rm coefficients\ of}\ (X-\varphi_1)(X-\varphi_2)\dots (X-\varphi_n)$$
induces an isomorphism of schemes over $\Q_p$:
$$T/\Scal_n\buildrel\sim\over\longrightarrow {\mathbb G}_{{\rm a}}^{n-1}\times_{\Spec\,\Q_p} {\mathbb G}_{{\rm m}}$$
and also of the associated rigid analytic spaces.
We deduce that the coefficients of the characteristic polynomial of the Frobenius $\Phi$ on ${\rm WD}(\tilde r)$ determine a morphism of rigid analytic spaces over $L$:
$$\Xfrak_{\overline r}^{\square,{\bf k}{\rm -cr}}\longrightarrow T_L^{\rig}/\Scal_n.$$
Let us define:
$$\widetilde\Xfrak_{\overline r}^{\square,{\bf k}\rm -cr}:=\Xfrak_{\overline r}^{\square,{\bf k}\rm -cr}\times_{T_L^\rig/\Scal_n}T_L^\rig.$$
Concretely $\widetilde\Xfrak_{\rbar}^{\square,{\bf k}\rm-cr}$ parametrizes crystalline framed $\Gcal_K$-deformations $r$ of $\overline r$ of labelled Hodge-Tate weights ${\bf k}$ together with an ordering $(\varphi_1,\dots,\varphi_n)$ of the eigenvalues of the geometric Frobenius on ${\rm WD}(r)$.

\begin{lemm}\label{reduced}
The rigid analytic space $\widetilde\Xfrak_{\overline r}^{\square,{\bf k}\rm -cr}$ is reduced.
\end{lemm}
\begin{proof}
It is sufficient to prove this result locally. Let $\mathrm{Sp}\,C$ be an admissible irreducible affinoid open subspace of $\mathfrak{X}_{\overline{r}}^{\square,{\rm {\bf k}-cr}}$ whose image in $T_L^{\rm rig}/\mathcal{S}_n$ is contained in an admissible affinoid open irreducible subspace $\mathrm{Sp}\,A$ of $T_L^{\rm rig}/\mathcal{S}_n$. As both $\mathfrak{X}_{\overline{r}}^{\square,{\rm {\bf k}-cr}}$ and $T_L^{\rm rig}/\mathcal{S}_n$ are smooth over $L$ we can find an admissible open affinoid covering of $\mathfrak{X}_{\overline{r}}^{\square,{\rm {\bf k}-cr}}$ by such $\mathrm{Sp}\,C$. The map $T_L^{\rm rig}\rightarrow T_L^{\rm rig}/\mathcal{S}_n$ is finite flat being the rigidification of a map of affine schemes $T_L\rightarrow T_L/\mathcal{S}_n$ which is finite flat. Consequently the inverse image of $\mathrm{Sp}\,A$ in $T_L^{\rm rig}$ is an admissible affinoid open subspace $\mathrm{Sp}\,B$ with $B$ an affinoid algebra which is finite flat over $A$. As $B$ is a finite $A$-algebra, we have an isomorphism $C\widehat{\otimes}_AB\simeq C\otimes_AB$. It follows, by definition of the fiber product of rigid analytic spaces, that the rigid analytic spaces of the form $\mathrm{Sp}(C\otimes_AB)$ form an admissible open covering of $\widetilde\Xfrak_{\overline r}^{\square,{\bf k}\rm -cr}=\mathfrak{X}_{\overline{r}}^{\square,{\rm {\bf k}-cr}}\times_{T_L^{\rm rig}/\mathcal{S}_n}T_L^{\rm rig}$. It is sufficient to prove that rings $C\otimes_AB$ as above are reduced. From Lemma \ref{red} below it is sufficient to prove that $C\otimes_AB$ is a finite flat generically \'etale $C$-algebra. As $B$ is finite flat over $A$, the $C$-algebra $C\otimes_AB$ is clearly finite flat. It is sufficient to prove that it is a generically \'etale $C$-algebra. As $B$ is generically \'etale over $A$, it is sufficient to prove that the map $\mathrm{Spec}\,C\rightarrow\mathrm{Spec}\,A$ is dominant. It is thus sufficient to prove that the map of rigid analytic spaces $\mathfrak{X}_{\overline{r}}^{\square,{\rm {\bf k}-cr}}\rightarrow T_L^{\rm rig}/\mathcal{S}_n$ is open. This follows from the fact that it has, locally on $\mathfrak{X}_{\overline{r}}^{\square,{\rm {\bf k}-cr}}$, a factorization:
$$\mathfrak{X}_{\overline{r}}^{\square,{\rm {\bf k}-cr}}\longrightarrow (\Res_{K_0/\Q_p}\GL_{n,K_0} \times_{\Spm\,\Q_p} {\rm Flag})^\rig\times_{\Spm\,\Q_p}\Spm\,L\longrightarrow T_L^{\rm rig}/\mathcal{S}_n$$
where the first map is the smooth map in the proof of Lemma \ref{genericopensinXcris} below, and the second is the projection on $(\mathrm{Res}_{K_0/\mathbb{Q}_p}\mathrm{GL}_{n,K_0})_L^{\rm rig}$ followed by the base change to $L$ of the rigidification of the morphism $\mathrm{Res}_{K_0/\mathbb{Q}_p}\mathrm{GL}_{n,K_0}\rightarrow T/\mathcal{S}_n$ defined in \cite[(9.1)]{HartlHellmann}. The first map being smooth is flat and thus open by \cite[Cor.9.4.2]{Bo}, and the last two are easily seen to be open.
\end{proof}

The following (well-known) lemma was used in the proof of Lemma \ref{reduced}.

\begin{lemm}\label{red}
Let $A$ a commutative noetherian domain and $B$ a finite flat $A$-algebra. Then the ring $B$ has no embedded component, i.e. all its associated ideals are minimal prime ideals. Moreover if $B$ is generically \'etale over $A$, i.e. $\mathrm{Frac}(A)\otimes_A B$ is a finite \'etale $\mathrm{Frac}(A)$-algebra, then the ring $B$ is reduced.
\end{lemm}
\begin{proof}
As $B$ is flat over $A$, the map $\mathrm{Spec}\,B\rightarrow\mathrm{Spec}\,A$ has an open image, and $A$ being a domain it contains the unique generic point of $\mathrm{Spec}\,A$, which implies that the natural map $A\rightarrow B$ is injective. Moreover $B$ being finite over $A$, the image of $\mathrm{Spec}\,B\rightarrow\mathrm{Spec}\,A$ is closed, hence it is $\mathrm{Spec}\,A$ since $\mathrm{Spec}\,A$ is connected. In particular $B$ is a faithfully flat $A$-algebra. As $B$ is a flat $A$-module, it follows from \cite[\S IV.2.6 Lem.1]{BourbakiAC} applied with $E=A$ and $F=B$ that $\mathfrak{p}\in\mathrm{Ass}(B)$ implies $\mathfrak{p}\cap A=0$ ($A$ is a domain, so $\mathrm{Ass}(A)=\{0\}$). It then follows from \cite[\S V.2.1 Cor.1]{BourbakiAC} that if $\mathfrak{p}\in\mathrm{Ass}(B)$, then $\mathfrak{p}$ is a minimal prime of $B$. Indeed, $A$ being noetherian and $B$ a finite $A$-module, $B$ is an integral extension of $A$. We can apply {\it loc. cit.} to the inclusion $\mathfrak{q}\subseteq\mathfrak{p}$ where $\mathfrak{q}$ is a minimal prime ideal of $B$ (both ideals $\mathfrak{q}$ and $\mathfrak{p}$ being above the prime ideal $(0)$ of $A$ since $\mathfrak{p}\cap A=\mathfrak{q}\cap A=0$).

Let $\mathfrak{d}_{B/A}$ be the discriminant of $B/A$ (its existence comes from the fact that $B$ is a finite faithfully flat $A$-algebra, hence a finite projective $A$-module). As the extension is generically \'etale, we can find $f\in\mathfrak{d}_{B/A}$ such that $B_f$ is \'etale over $A_f$. As $A_f$ is a domain, $B_f$ is then reduced. Thus the nilradical $\mathfrak{n}$ of $B_f$ is killed by some power of $f$. Replacing $f$ by this power, we can assume that the vanishing ideal of $\mathfrak{n}$ contains $f$. Assume that $\mathfrak{n}$ is nonzero and let $\mathfrak{p}$ be a prime ideal of $B$ minimal among prime ideals containing $\mathrm{Ann}_B(\mathfrak{n})$. It follows from \cite[\S IV.1.3 Cor.1]{BourbakiAC} that $\mathfrak{p}$ is an associated prime of the $B$-module $\mathfrak{n}$ and consequently of $B$. But we have $f\in\mathfrak{p}$ which contradicts the fact that $\mathfrak{p}\cap A=0$.
\end{proof}

We now embed this ``refined'' crystalline deformation space $\widetilde\Xfrak_{\overline r}^{\square,{\bf k}\rm -cr}$ into the space $X_{\rm tri}^\square(\rbar)$ as follows. We define a morphism of rigid spaces over $L$:
\begin{eqnarray}\label{cristotriang}
\Xfrak_{\overline r}^{\square,{\bf k}\rm -cr}\times_{\Spm\,L} T_L^\rig&\longrightarrow &\Xfrak_{\overline r}^\square\times_{\Spm\,L}\Tcal^n_L\\
\nonumber (r,\varphi_1,\dots,\varphi_n)&\longmapsto & (r, z^{{\bf k}_1}{\rm unr}(\varphi_1),\dots, z^{{\bf k}_n}{\rm unr}(\varphi_n)).
\end{eqnarray}
This morphism is a closed embedding of reduced rigid spaces as both maps $r\mapsto r$ and $(\varphi_1,\dots,\varphi_n)\mapsto (z^{{\bf k}_1}{\rm unr}(\varphi_1),\dots, z^{{\bf k}_n}{\rm unr}(\varphi_n))$ respectively define closed embeddings $\Xfrak_{\overline r}^{\square,{\bf k}\rm-cr}\!\!\hookrightarrow \Xfrak_{\overline r}^\square$ and $T_L^\rig\hookrightarrow \Tcal^n_L$. We claim that the restriction of the morphism $(\ref{cristotriang})$ to:
\begin{equation}\label{cristotriangres}
\widetilde\Xfrak_{\overline r}^{\square,{\bf k}\rm -cr}\hookrightarrow\Xfrak_{\overline r}^{\square,{\bf k}\rm -cr}\times_{\Spm\,L}T_L^\rig
\end{equation}
factors through $X_{\rm tri}^\square(\overline r)\subset \Xfrak_{\overline r}^\square\times_{\Spm\,L}\Tcal^n_L$. 
As the source of this restriction is reduced by Lemma \ref{reduced}, it is enough to check it on a Zariski-dense set of points of $\widetilde\Xfrak_{\overline r}^{\square,{\bf k}\rm -cr}$. 

Let $r$ be an $n$-dimensional crystalline representation of ${\mathcal G}_K$ over a finite extension $L'$ of $L$ of Hodge-Tate weights ${\bf k}$ and let $\varphi_1,\dots,\varphi_n$ be an ordering of the eigenvalues of a geometric Frobenius on ${\rm WD}(r)$, equivalently of the eigenvalues of $\varphi^{[K_0:\Q_p]}$ on $D_{\rm cris}(r)$ (that are assumed to be in $L'^\times$). Assuming moreover that the $\varphi_i$ are pairwise distinct, this datum gives rise to a unique complete $\varphi$-stable flag of free $K_0\otimes_{\Q_p}L'$-modules:
$$0=\Fcal_0\subset \Fcal_1\subset \dots\subset \Fcal_n=D_{\rm cris}(r)$$
on $D_{\rm cris}(r)$ such that $\varphi^{[K_0:\Q_p]}$ acts on $\Fcal_i/\Fcal_{i-1}$ by multiplication by $\varphi_i$ (this is a \emph{refinement} in the sense of \cite[Def.2.4.1]{BelChe}). By the same argument as in the proof of Lemma \ref{paramofcrystpt} using Berger's dictionary between crystalline $(\varphi,\Gamma_K)$-modules and filtered $\varphi$-modules (see e.g. (\ref{berger})), the filtration $\Fcal_\bullet$ induces a triangulation $\Fil_{\bullet}$ on $D_{\rig}(r)$. If we assume that $\Fcal_\bullet$ is {\it noncritical} in the sense of \cite[Def.2.4.5]{BelChe}, i.e. the filtration $\Fcal_\bullet$ is in general position with respect to the Hodge filtration $\Fil^\bullet D_{\rm dR}(r)$ on $D_{\rm dR}(r)$, that is, for all embeddings $\tau:K\hookrightarrow L$ and all $i=1,\dots,n-1$ we have:
\begin{multline}\label{generalpos}
\big(\Fcal_{i}\otimes_{K_0\otimes_{\Q_p}L',\tau\otimes \id}L'\big)\oplus \big(\Fil^{-k_{\tau,i+1}}D_{\rm dR}(r)\otimes_{K\otimes_{\Q_p}L',\tau\otimes \id}L'\big)=D_{\rm cris}(r)\otimes_{K_0\otimes_{\Q_p}L,\tau\otimes \id}L'\\
=D_{\rm dR}(r)\otimes_{K\otimes_{\Q_p}L',\tau\otimes \id}L',
\end{multline}
then $\Fcal_{i}/\Fcal_{i-1}$ is a filtered $\varphi$-module of Hodge-Tate weights ${\bf k}_i$, or equivalently $\Fil_i/\Fil_{i-1}\cong \Rcal_{L',K}(\delta_i)$ with $\delta_i=z^{{\bf k}_i}{\rm unr}(\varphi_i)$. 

\begin{lemm}\label{genericopensinXcris}
There are smooth (over $L$) Zariski-open and Zariski-dense subsets in $\widetilde\Xfrak_{\overline r}^{\square,{\bf k}\rm -cr}$:
$$\widetilde V_{\overline r}^{\square,{\bf k}\rm -cr} \subset \widetilde U_{\overline r}^{\square,{\bf k}\rm -cr}\subset\widetilde\Xfrak_{\overline r}^{\square,{\bf k}\rm -cr}$$
such that:
\begin{enumerate}
\item[(i)] a point $(r,\varphi_1,\dots,\varphi_n)\in \widetilde \Xfrak_{\overline r}^{\square,{\bf k}\rm -cr}$ lies in $\widetilde U_{\overline r}^{\square,{\bf k}\rm -cr}$ if and only if the $\varphi_i$ are pairwise distinct;\\
\item[(ii)] a point  $(r,\varphi_1,\dots,\varphi_n)\in \widetilde U_{\overline r}^{\square,{\bf k}\rm -cr}$ lies in $\widetilde V_{\overline r}^{\square,{\bf k}\rm -cr}$ if and only if it satisfies assumption $(\ref{generalpos})$ above and $z^{{\bf k}_i-{\bf k}_j}{\rm unr}(\varphi_i\varphi_j^{-1})\in\Tcal_{\rm reg}$ for $i\ne j$. 
\end{enumerate}
Moreover the image of $\widetilde V_{\overline r}^{\square,{\bf k}\rm -cr}$ via (\ref{cristotriang}) composed with (\ref{cristotriangres}) lies in $U_{\rm tri}^\square(\overline r)$.
\end{lemm}
\begin{proof}
The idea of the proof is the same as that of \cite[Lem.4.4]{Chenevier}. It is enough to show that all the statements are true locally on $\widetilde\Xfrak_{\overline r}^{\square,{\bf k}\rm -cr}$. Let us (locally) fix a basis of the coherent locally free $K_0\otimes_{\Q_p}\Ocal_{\!\Xfrak_{\overline r}^{\square,{\bf k}{\rm -cr}}}$-module $\Dcal$ on $\Xfrak_{\rbar}^{\square,{\bf k}\rm -cr}$. By the choice of such a basis, the matrix of the crystalline Frobenius $\Phi_{\rm cris}$ and the Hodge filtration define (locally) a morphism:
$$\Xfrak_{\rbar}^{\square,{\bf k}\rm -cr}\longrightarrow (\Res_{K_0/\Q_p}\GL_{n,K_0} \times_{\Spm\,\Q_p} {\rm Flag})^\rig\times_{\Spm\,\Q_p}\Spm\,L$$
where ${\rm Flag}:=(\Res_{K/\Q_p}\GL_{n,K})/(\Res_{K/\Q_p} B)$ (compare \cite[\S8]{HartlHellmann}). By \cite[Prop.8.12]{HartlHellmann} and the discussion preceding it, it follows that this morphism is smooth, hence so is the morphism:
\begin{equation}\label{morphismbasechange}
\widetilde \Xfrak_{\rbar}^{\square,{\bf k}\rm -cr}\longrightarrow \big((\Res_{K_0/\Q_p}\GL_{n,K_0})^\rig_L\times_{T_L^\rig/\Scal_n}T_L^\rig\big)\times_{\Spm\,L} {\rm Flag}_L^\rig
\end{equation}
where $\Res_{K_0/\Q_p}\GL_{n,K_0}\rightarrow T/\Scal_n$ is the morphism defined in \cite[(9.1)]{HartlHellmann}. On the other hand, using that the morphism $T\rightarrow T/\Scal_n$ is obviously smooth in the neighbourhood of a point $(\varphi_1,\dots,\varphi_n)\in T$ where the $\varphi_i$ are pairwise distinct, we see that the conditions of (i), resp.~(ii), in the statement cut out smooth (over $L$) Zariski-open and Zariski-dense subspaces of:
\begin{equation}\label{RHS}
\big((\Res_{K_0/\Q_p}\GL_{n,K_0})^\rig_L\times_{T_L^\rig/\Scal_n}T_L^\rig\big)\times_{\Spm\,L} {\rm Flag}_L^\rig.
\end{equation}
Their inverse images in $\widetilde \Xfrak_{\overline r}^{\square,{\bf k}\rm -cr}$ via (\ref{morphismbasechange}) are thus smooth over $L$ and Zariski-open in $\widetilde \Xfrak_{\overline r}^{\square,{\bf k}\rm -cr}$. Let us prove that these inverse images are also Zariski-dense in $\widetilde\Xfrak_{\overline r}^{\square,{\bf k}\rm -cr}$. It is enough to prove that they intersect nontrivially every irreducible component of $\widetilde \Xfrak_{\overline r}^{\square,{\bf k}\rm -cr}$. Let $\Spm\,A$ be any affinoid open subset of $\widetilde \Xfrak_{\overline r}^{\square,{\bf k}\rm -cr}$, it follows from \cite[Cor.9.4.2]{Bo} that the image of $\Spm\,A$ by the smooth, hence flat, morphism (\ref{morphismbasechange}) is admissible open in (\ref{RHS}). In particular its intersection with one of the above Zariski-open and Zariski-dense subspaces of (\ref{RHS}) can't be empty, which proves the statement. The final claim of the lemma follows from (ii), the discussion preceding Lemma \ref{genericopensinXcris} and the definition (\ref{ucris}) of $U_{\rm tri}^\square(\overline r)$ . 
\end{proof}

Note that $\widetilde\Xfrak_{\overline r}^{\square,{\bf k}\rm -cr}$ is equidimensional of the same dimension as $\Xfrak_{\overline r}^{\square,{\bf k}\rm -cr}$. Indeed, by Lemma \ref{genericopensinXcris} it is enough to prove the same statement for $\widetilde U_{\overline r}^{\square,{\bf k}\rm -cr}$. But this is clear since the map $\widetilde U_{\overline r}^{\square,{\bf k}\rm -cr}\rightarrow \Xfrak_{\overline r}^{\square,{\bf k}\rm -cr}$ is smooth of relative dimension $0$, hence \'etale, and since $\Xfrak_{\overline r}^{\square,{\bf k}\rm -cr}$ is equidimensional (\cite{Kisindef}). Lemma \ref{genericopensinXcris} also implies that $(\ref{cristotriang})$ induces (as claimed above) a morphism:
\begin{equation}\label{embeddingcristri}
\iota_{\bf k}:\widetilde\Xfrak_{\overline r}^{\square,{\bf k}\rm -cr}\longrightarrow X_{\rm tri}^\square(\overline r)
\end{equation}
which is obviously a closed immersion as $(\ref{cristotriang})$ is. 

\begin{coro}\label{defnZtri(x)}
Let $x=(r,\delta)\in X_{\rm tri}^\square(\overline r)$ be a crystalline strictly dominant point such that $\omega(x)=\delta_{\bf k}$ and the Frobenius eigenvalues $(\varphi_1,\dots,\varphi_n)$ (cf. Lemma \ref{paramofcrystpt}) are pairwise distinct and let $U$ be an open subset of $X_{\rm tri}^\square(\overline r)$ containing $x$.\\
(i) The point $x$ belongs to $\iota_{\bf k}(\widetilde U_{\overline r}^{\square,{\bf k}{\rm -cr}})$ and there is a unique irreducible component $\widetilde Z_{\rm cris}(x)$ of $\widetilde\Xfrak_{\overline r}^{\square,{\bf k}{\rm -cr}}$ containing $\iota_{\bf k}^{-1}(x)$.\\
(ii) If $U$ is small enough there is a unique irreducible component $Z_{{\rm tri},U}(x)$ of $U$ containing $\iota_{\bf k}(\widetilde Z_{\rm cris}(x))\cap U$, and it is such that $Z_{{\rm tri},U}(x)\cap U'=Z_{{\rm tri},U'}(x)$ for any open $U'\subseteq U$ containing $x$.
\end{coro}
\begin{proof}
(i) The assumptions and Lemma \ref{paramofcrystpt} imply that $x$ is in the image of the map $\iota_{\bf k}$ in (\ref{embeddingcristri}) and the fact that the $\varphi_i$ are pairwise distinct implies that $x\in \iota_{\bf k}(\widetilde U_{\overline r}^{\square,{\bf k}{\rm -cr}})$. In particular $\widetilde \Xfrak_{\overline r}^{\square,{\bf k}{\rm -cr}}$ is smooth at $\iota_{\bf k}^{-1}(x)$ by Lemma \ref{genericopensinXcris} and thus $\iota_{\bf k}^{-1}(x)$ belongs to a unique irreducible component $\widetilde Z_{\rm cris}(x)$ of $\widetilde\Xfrak_{\overline r}^{\square,{\bf k}{\rm -cr}}$.

(ii) We have that $\iota_{\bf k}(\widetilde Z_{\rm cris}(x))\cap U$ is a Zariski-closed subset of $U$, and it is easy to see that it is still irreducible if $U$ is small enough since $\iota_{\bf k}(\widetilde Z_{\rm cris}(x))$ is smooth at $x$. Hence there exists at least one irreducible component of $U$ containing the irreducible Zariski-closed subset $\iota_{\bf k}(\widetilde Z_{\rm cris}(x))\cap U$. If there are two such irreducible components, then in particular {\it any} point of $\iota_{\bf k}(\widetilde Z_{\rm cris}(x))\cap U$ is singular in $U$, hence in $X_{\rm tri}^\square(\overline r)$. But Lemma \ref{genericopensinXcris} implies $\widetilde Z_{\rm cris}(x)\cap \widetilde V_{\overline r}^{\square,{\bf k}\rm -cr}\ne \emptyset$ is Zariski-open and Zariski-dense in $\widetilde Z_{\rm cris}(x)$, hence:
$$\big(\widetilde Z_{\rm cris}(x)\cap \widetilde V_{\overline r}^{\square,{\bf k}\rm -cr}\big)\cap \big(\widetilde Z_{\rm cris}(x)\cap \iota_{\bf k}^{-1}(U)\big) \ne \emptyset$$
from which we get $\iota_{\bf k}(\widetilde Z_{\rm cris}(x)\cap \widetilde V_{\overline r}^{\square,{\bf k}\rm -cr})\cap U\ne \emptyset$. The last statement of Lemma \ref{genericopensinXcris} also implies $\iota_{\bf k}(\widetilde Z_{\rm cris}(x)\cap \widetilde V_{\overline r}^{\square,{\bf k}\rm -cr})\cap U\subseteq U_{\rm tri}^\square(\overline r)$, which is then a contradiction since $U_{\rm tri}^\square(\overline r)$ is smooth over $L$.

Finally, shrinking $U$ again if necessary, we can assume that, for any open subset $U'\subseteq U$ containing $x$, the map $Z\mapsto Z\cap U'$ induces a bijection between the irreducible components of $U$ containing $x$ and the irreducible components of $U'$ containing $x$. It then follows from the definition of $Z_{{\rm tri},U}(x)$ that $Z_{{\rm tri},U}(x)\cap U'=Z_{{\rm tri},U'}(x)$.
\end{proof}

\begin{rema}\label{down}
{\rm (i) Since the map $\widetilde \Xfrak_{\overline r}^{\square,{\bf k}{\rm -cr}}\rightarrow \Xfrak_{\overline r}^{\square,{\bf k}{\rm -cr}}$ is finite, hence closed, and since $\widetilde \Xfrak_{\overline r}^{\square,{\bf k}{\rm -cr}}$, $\Xfrak_{\overline r}^{\square,{\bf k}{\rm -cr}}$ are both equidimensional (of the same dimension), the image of any irreducible component of $\widetilde \Xfrak_{\overline r}^{\square,{\bf k}{\rm -cr}}$ is an irreducible component of $\Xfrak_{\overline r}^{\square,{\bf k}{\rm -cr}}$. In particular the image of $\widetilde Z_{\rm cris}(x)$ in (i) of Corollary \ref{defnZtri(x)} is the unique irreducible component of $\Xfrak_{\overline r}^{\square,{\bf k}{\rm -cr}}$ containing $r$.\\
(ii) Either by the same proof as that for (ii) of Corollary \ref{defnZtri(x)} or as a consequence of (ii) of Corollary \ref{defnZtri(x)}, we see that there is also a unique irreducible component $Z_{{\rm tri}}(x)$ of the whole $X_{\rm tri}^\square(\overline r)$ which contains the irreducible closed subset $\iota_{\bf k}(\widetilde Z_{\rm cris}(x))$.}
\end{rema}

\subsection{The Weyl group element associated to a crystalline point}\label{weyl}

We review the definition of the Weyl group element associated to certain crystalline points on $X_{\rm tri}^\square(\overline r)$ (measuring their ``criticality'') and state our main local results (Theorem \ref{upperbound}, Corollary \ref{Xtrismooth}). 

We keep the notation of \S\ref{variant}. We let $W\cong \prod_{\tau:\, K\hookrightarrow L}{\mathcal S}_n$ be the Weyl group of the algebraic group:
$$(\Res_{K/\Q_p}\GL_{n,K})\times_{\Spec\,\Q_p}\Spec\,L\cong \prod_{\tau:\, K\hookrightarrow L}{\GL_{n,L}}$$
and $X^\ast((\Res_{K/\Q_p}T_K)\times_{\Spec\,\Q_p}\Spec\,L)\cong \prod_{\tau:\, K\hookrightarrow L}X^\ast(T_L)$ be the $\Z$-module of algebraic characters of $(\Res_{K/\Q_p}T_K)\times_{\Spec\,\Q_p}\Spec\,L$ (recall $T$ is the diagonal torus in $\GL_n$ and $T_K$, $T_L$ its base change to $K$, $L$). We write ${\rm lg}(w)$ for the length of $w$ in the Coxeter group $W$ (for the set of simple reflections associated to the simple roots of the upper triangular matrices).

Let $x=(r,\delta)=(r,\delta_1,\dots,\delta_n)$ be a crystalline strictly dominant point on $X_{\rm tri}^\square(\rbar)$. Then by Lemma $\ref{paramofcrystpt}$ the characters $\delta_i$ are of the form $\delta_i=z^{{\bf k}_i}{\rm unr} (\varphi_i)$ where ${\bf k}_i=(k_{\tau,i})_{\tau:\, K\hookrightarrow L}$ and the $\varphi_i\in k(x)^\times$ are the eigenvalues of the geometric Frobenius on ${\rm WD}(r)$. Assume that the $\varphi_i$ are pairwise distinct, then as in \S\ref{variant} the ordering $(\varphi_1,\dots,\varphi_n)$ defines a complete $\varphi$-stable flag of free $K_0\otimes_{\Q_p}k(x)$-modules $0=\Fcal_0\subset \Fcal_1\subset \dots\subset \Fcal_n=D_{\rm cris}(r)$ on $D_{\rm cris}(r)$ such that $\varphi^{[K_0:\Q_p]}$ acts on $\Fcal_i/\Fcal_{i-1}$ by multiplication by $\varphi_i$. We view $\Fcal_i$ as a filtered $\varphi$-module with the induced Hodge filtration. If we write $(k'_{\tau,i})_{\tau:K\hookrightarrow L}$ for the Hodge-Tate weights of $\Fcal_i/\Fcal_{i-1}$, we find that there is a unique $w_x=(w_{x,\tau})_{\tau:\, K\hookrightarrow L}\in W=\prod_{\tau: K\hookrightarrow L}\Scal_n$  such that:
\begin{equation}\label{wxdef}
k'_{\tau, i}=k_{\tau, w_{x,\tau}^{-1}(i)}
\end{equation}
for all $i\in \{1,\dots, n\}$ and each $\tau: K\hookrightarrow L$. We call $w_x$ the \emph{Weyl group element associated to} $x$. Note that $\Fcal_\bullet$ is noncritical (see \S\ref{variant}) if and only if $w_{x,\tau}=1$ for all $\tau:K\hookrightarrow L$, in which case we say that the crystalline strictly dominant point $x=(r,\delta)$ is {\it noncritical}. 

For $w\in W$ we denote by $d_w\in \Z_{\geq 0}$ the rank of the $\Z$-submodule of $X^\ast((\Res_{K/\Q_p}T_K)\times_{\Spec\,\Q_p}\Spec\,L)$ generated by the $w(\alpha)-\alpha$ where $\alpha$ runs among the roots of $(\Res_{K/\Q_p}\GL_{n,K})\times_{\Spec\,\Q_p}\Spec\,L$. 

\begin{lemm}\label{coxeter}
With the above notations we have:
$$d_w\leq \lg(w)\leq [K:\Q_p]\frac{n(n-1)}{2}$$
and $\lg(w)= d_w$ if and only if $w$ is a product of distinct simple reflections.
\end{lemm}
\begin{proof}
Note first that the right hand side inequality is obvious. Let us write in this proof $X:=X^\ast(({\rm Res}_{K/\Q_p}T_K)\times_{\Spec\,\Q_p}\Spec\,L)$, $X_{\Q}:=X\otimes_{\Z}\Q$, and let us denote by $S$ the subset of simple reflections in $W$ (thus $\dim_{\Q}(X_{\Q})=[K:\Q_p]n$ and $\vert S\vert=[K:\Q_p](n-1)$). The rank of the subgroup of $X$ generated by the $w(\alpha)-\alpha$ for $\alpha$ as above, or equivalently by the $w(\alpha)-\alpha$ for $\alpha\in X$, is equal to the dimension of the $\Q$-vector space $(w-\id)X_{\Q}$ which, by the rank formula, is equal to $\dim_{\Q}(X_{\Q})-\dim_{\Q}(\ker(w-\id))$. Let $I$ be the set of simple reflections appearing in $w$, we have $|I|\leq \lg(w)$ and $|I|= \lg(w)$ if and only if $w$ is a product of distinct simple reflections. It is thus enough to prove $\dim_{\Q}(\ker(w-\id))\geq \dim_{\Q}(X_{\Q})-|I|$ with equality when $w$ is a product of distinct simple reflections. Note that $\ker(w-\id)$ obviously contains the $\Q$-subvector space of $X_\Q$ of fixed points by the subgroup $W_I$ of $W$ generated by the elements of $I$, and it follows from \cite[Th.1.12(c)]{Hum} that, when $w$ is a product of distinct simple reflections, then $\ker(w-\id)$ is exactly this $\Q$-subvector space. It is thus enough to prove that this $\Q$-subvector space of $X_\Q$, which is just the intersection of the hyperplanes $\ker(s-\id)$ for $s\in I$, has dimension $\dim_{\Q}(X_\Q)-|I|$. However we know that for any $\Q$-subvector space $V\subset X_{\Q}$ and any reflection $s$ of $X_{\Q}$, we have $\dim_{\Q}(V\cap\ker(s-\id))\geq \dim_{\Q}(V)-1$ and thus by induction:
$$\dim_{\Q}\big(V\bigcap\big(\bigcap_{s\in S}\ker(s-\id)\big)\big)\geq \dim_{\Q}(V)-|S|$$
with equality if and only if $\dim_{\Q}\big(V\bigcap\big(\bigcap_{s\in J}\ker(s-\id)\big)\big)= \dim_{\Q}(V)-|J|$ for all $J\subseteq S$. As the $\Q$-subvector space $X_{\Q}^W$ of fixed points by $W$ has dimension $[K:\Q_p]$ (it is generated by the characters $\tau\circ\det$ for $\tau:K\hookrightarrow L$), we have:
$$ \dim_{\Q}(X_{\Q}^W)=\bigcap_{s\in S}\ker(s-\id)=[K:\Q_p]=\dim_{\Q}(X_{\Q})-|S|.$$
Consequently we deduce (taking $V=X_\Q$):
$$\dim_{\Q}\big(\bigcap_{s\in I}\ker(s-\id)\big)= \dim_{\Q}(X_\Q)-|I|$$
which is the desired formula.
\end{proof}

Recall that, if $X$ is a rigid analytic variety over $L$ and $x\in X$, the tangent space to $X$ at $x$ is the $k(x)$-vector space:
\begin{eqnarray}\label{tangent}
T_{X,x}:=\Hom_{k(x)}\big({\mathfrak m}_{X,x}/{\mathfrak m}^2_{X,x},k(x)\big)=\Hom_{k(x)-{\rm alg}}\big({\mathcal O}_{X,x},k(x)[\varepsilon]/(\varepsilon^2)\big)
\end{eqnarray}
where ${\mathfrak m}_{X,x}$ is the maximal ideal of the local ring ${\mathcal O}_{X,x}$ at $x$ to $X$. If $X$ is equidimensional, recall also that $\dim_{k(x)}T_{X,x}\geq \dim X$ and that $X$ is smooth at $x$ if and only if $\dim_{k(x)}T_{X,x}= \dim X$.

We let $\widetilde X_{\rm tri}^\square(\rbar)\subseteq X_{\rm tri}^\square(\rbar)$ be the union of the irreducible components $C$ of $X_{\rm tri}^\square(\rbar)$ such that $C\cap U_{\rm tri}^\square(\rbar)$ contains a crystalline point. For instance it follows from Lemma \ref{genericopensinXcris} that all the closed embeddings (\ref{embeddingcristri}) factor as closed embeddings $\iota_{\bf k}:\widetilde\Xfrak_{\overline r}^{\square,{\bf k}\rm -cr}\hookrightarrow \widetilde X_{\rm tri}^\square(\overline r)\subseteq X_{\rm tri}^\square(\overline r)$. In particular any point $x\in X_{\rm tri}^\square(\rbar)$ which is crystalline strictly dominant is in $\widetilde X_{\rm tri}^\square(\rbar)$.

The following statement is our main local conjecture.

\begin{conj}\label{mainconj}
Let $x\in X_{\rm tri}^\square(\rbar)$ be a crystalline strictly dominant point such that the Frobenius eigenvalues $(\varphi_1,\dots,\varphi_n)$ (cf. Lemma \ref{paramofcrystpt}) are pairwise distinct, let $w_x$ be the Weyl group element associated to $x$ (cf. (\ref{wxdef})) and let $d_x:=d_{w_x}$. Then:
$$\dim_{k(x)}T_{\widetilde X_{\rm tri}^\square(\rbar),x}=\lg(w_x)-d_{x}+\dim X_{\rm tri}^\square(\rbar)=\lg(w_x)-d_{x}+n^2+[K:\Q_p]\frac{n(n+1)}{2}.$$
\end{conj}

In particular, since $\dim \widetilde X_{\rm tri}^\square(\rbar)=\dim X_{\rm tri}^\square(\rbar)=n^2+[K:\Q_p]\frac{n(n+1)}{2}$, we see by Lemma \ref{coxeter} that $\widetilde X_{\rm tri}^\square(\rbar)$ should be smooth at $x$ if and only if $w_x$ is a product of distinct simple reflections.

\begin{rema}
{\rm The reader can wonder why we don't state Conjecture \ref{mainconj} with $X_{\rm tri}^\square(\rbar)$ instead of $\widetilde X_{\rm tri}^\square(\rbar)$. The reason is that Conjecture \ref{mainconj} with $\widetilde X_{\rm tri}^\square(\rbar)$ is actually implied by other conjectures, see \S\ref{modularity}, and we don't know if this is the case with $X_{\rm tri}^\square(\rbar)$.}
\end{rema}

In order to state our main result in the direction of (a weaker version of) Conjecture \ref{mainconj}, we need the following two definitions. 

\begin{defi}\label{veryreg}
A crystalline strictly dominant point $x=(r,\delta)\in X_{\rm tri}^\square(\rbar)$ is very regular if it satisfies the following conditions (where the $(\varphi_i)_{1\leq i\leq n}$ are the geometric Frobenius eigenvalues on ${\rm WD}(r)$):
\begin{enumerate}
\item[(i)]$\varphi_i\varphi_j^{-1}\notin \{1,q\}$ for $1\leq i\ne j\leq n$;\\
\item[(ii)]$\varphi_1\varphi_2\dots\varphi_i$ is a simple eigenvalue of the geometric Frobenius acting on $\bigwedge_{k(x)}^i\!\!{\rm WD}(r)$ for $1\leq i\leq n$.
\end{enumerate}
\end{defi}

\begin{rema}\label{varphi}
{\rm If $x=(r,\delta)$ is crystalline strictly dominant, it easily follows from the dominance property that (i) of Definition \ref{veryreg} is equivalent to $\delta_i\delta_j^{-1}\notin \{z^{-\bf h},\vert z\vert_K z^{\bf h},\ \vert z\vert_K^{-1} z^{\bf h},\ {\bf h}\in \Z_{\geq 0}^{\Hom(K,L)}\}$ for $1\leq i\ne j\leq n$. In particular it implies $\delta\in \Tcal_{\rm reg}^n$, whence the terminology (compare also \cite[\S6.1]{Bergdall}).}
\end{rema}

\begin{defi}\label{accu}
Let $X$ be a union of irreducible components of an open subset of $X_{\rm tri}^\square(\rbar)$ (over $L$) and let $x\in X_{\rm tri}^\square(\rbar)$ such that $\omega(x)$ is algebraic. Then $X$ satisfies the {\rm accumulation property at} $x$ if $x\in X$ and if, for any positive real number $C>0$, the set of crystalline strictly dominant points $x'=(r',\delta')$ such that:
\begin{enumerate}
\item[(i)]the eigenvalues of $\varphi^{[K_0:\Q_p]}$ on $D_{\rm cris}(r')$ are pairwise distinct;\\
\item[(ii)]$x'$ is noncritical;\\
\item[(iii)]$\omega(x')=\delta'\vert_{({\mathcal O}_K^\times)^n}=\delta_{\bf k'}$ with $k'_{\tau,i}-k'_{\tau,i+1}>C$ for $i\in \{1,\dots,n-1\}$, $\tau\in \Hom(K,L);$
\end{enumerate}
accumulate at $x$ in $X$ in the sense of \cite[\S3.3.1]{BelChe}.
\end{defi}

It easily follows from Definition \ref{accu} that $X$ satisfies the accumulation property at $x$ if and only if each irreducible component of $X$ containing $x$ satisfies the accumulation property at $x$. In particular, if $x$ belongs to each irreducible component of $X$, we see that for every $C>0$ the set of points $x'$ in the statement of Definition \ref{accu} is also Zariski-dense in $X$. Since $U_{\rm tri}^\square(\rbar)\cap X$ is Zariski-open and Zariski-dense in $X$, we also see that each irreducible component of $X$ containing $x$ also contains points $x'$ as in Definition \ref{accu} which are in $U_{\rm tri}^\square(\rbar)$, hence each irreducible component of $X$ containing $x$ is in $\widetilde X_{\rm tri}^\square(\rbar)$.

In \S\ref{phiGammacohomology} below we will prove the theorem that follows. 

\begin{theo}\label{upperbound}
Let $x\in X_{\rm tri}^\square(\rbar)$ be a crystalline strictly dominant very regular point and let $X\subseteq X_{\rm tri}^\square(\rbar)$ be a union of irreducible components of an open subset of $X_{\rm tri}^\square(\rbar)$ such that $X$ satisfies the accumulation property at $x$. Then we have:
$$\dim_{k(x)}T_{X,x}\leq \lg(w_x)-d_x+\dim X_{\rm tri}^\square(\rbar)=\lg(w_x)-d_x+n^2+[K:\Q_p]\frac{n(n+1)}{2}.$$
\end{theo}

By Lemma \ref{coxeter} we thus deduce the following important corollary.

\begin{coro}\label{Xtrismooth}
Let $x\in X_{\rm tri}^\square(\rbar)$ be a crystalline strictly dominant very regular point and let $X\subseteq X_{\rm tri}^\square(\rbar)$ be a union of irreducible components of an open subset of $X_{\rm tri}^\square(\rbar)$ such that $X$ satisfies the accumulation property at $x$. Assume that $w_x$ is a product of distinct simple reflections. Then $X$ is smooth at $x$.
\end{coro}

\begin{rema}
{\rm Note that for $X$, $x$ as above we only have $\dim_{k(x)}T_{X,x}\leq \dim_{k(x)}T_{\widetilde X_{\rm tri}^\square(\rbar),x}$, thus Theorem \ref{upperbound} doesn't give an upper bound on $\dim_{k(x)}T_{\widetilde X_{\rm tri}^\square(\rbar),x}$ (but Conjecture \ref{mainconj} implies Theorem \ref{upperbound}). However Theorem \ref{upperbound} and Corollary \ref{Xtrismooth} will be enough for our purpose.}
\end{rema}

\section{Crystalline points on the patched eigenvariety}\label{globalpart}

We state the classicality conjecture (Conjecture \ref{classiconj}) and prove new cases of it (Corollary \ref{mainclassic}).

\subsection{The classicality conjecture}\label{classic}

We review the definition of classicality (Definition \ref{class}, Proposition \ref{compaclas}) and state the classicality conjecture (Conjecture \ref{classiconj}).

We first recall the global setting, basically the same as \cite[\S2.4]{BHS}. We fix a totally real field $F^+$, we write $q_v$ for the cardinality of the residue field of $F^+$ at a finite place $v$ and we denote by $S_p$ the set of places of $F^+$ dividing $p$ . We fix a totally imaginary quadratic extension $F$ of $F^+$ that splits at all places of $S_p$ and let $\Gcal_F:={\rm Gal}(\overline F/F)$. We fix a unitary group $G$ in $n$ variables over $F^+$ (with $n\geq 2$) such that $G\times_{F^+}F\cong \GL_{n,F}$ and $G(F^+\otimes_\Q\mathbb{R})$ is compact. We fix an isomorphism $i:G\times_{F^+}F\buildrel \sim\over\rightarrow \GL_{n,F}$ and, for each $v\in S_p$, a place $\tilde v$ of $F$ dividing $v$. The isomorphisms $F_v^+\buildrel\sim\over\rightarrow F_{\tilde v}$ and $i$ induce an isomorphism $i_{\tilde v}:\,G(F_v^+)\xrightarrow{\sim}\GL_n(F_{\tilde v})$ for $v\in S_p$. We let $G_v:=G(F^+_v)\cong \GL_n(F_{\tilde v})$ and $G_p:=\prod_{v\in S_p}G(F^+_v)\cong \prod_{v\in S_p}\GL_n(F_{\tilde v})$. We denote by $K_v$ (resp. $B_v$, resp. $\overline B_v$, resp. $T_v$) the inverse image of $\GL_n(\mathcal{O}_{F_{\tilde{v}}})$ (resp. of the subgroup of upper triangular matrices of $\GL_n(F_{\tilde v})$, resp. of the subgroup of lower triangular matrices of $\GL_n(F_{\tilde v})$, resp. of the subgroup of diagonal matrices of $\GL_n(F_{\tilde v})$) in $G_v$ under $i_{\tilde v}$ and we let $K_p:=\prod_{v\in S_p}K_v$ (resp. $B_p:=\prod_{v\in S_p}B_v$, resp. $\overline B_p:=\prod_{v\in S_p}\overline B_v$, resp. $T_p:=\prod_{v\in S_p}T_v$). We let $T_p^0:=T_p\cap K_p=\prod_{v\in S_p}(T_v\cap K_v)$. 

We fix a finite extension $L$ of $\Q_p$ that is assumed to be large enough so that $|\Hom(F^+_v,L)|=[F^+_v:\Q_p]$ for $v\in S_p$. We let $\widehat{T}_{p,\rm reg}\subset \widehat{T}_{p,L}$ the open subspace of characters $\delta=(\delta_v)_{v\in S_p}=(\delta_{v,1},\dots,\delta_{v,n})_{v\in S_p}$ such that $\delta_{v,i}/\delta_{v,j}\in {\mathcal T}_{v,\rm reg}$ for all $v\in S_p$ and all $i\ne j$, where ${\mathcal T}_{v,\rm reg}$ is defined as ${\mathcal T}_{\rm reg}$ of \S\ref{begin} but with $F^+_v=F_{\tilde v}$ instead of $K$.

We fix a tame level $U^p=\prod_{v}U_v\subset G(\Abb_{F^+}^{p\infty})$ where $U_v$ is a compact open subgroup of $G(F_v^+)$ and we denote by $\widehat S(U^p,L)$ the associated space of $p$-adic automorphic forms on $G(\Abb_{F^+})$ of tame level $U^p$ with coefficients in $L$, that is, the $L$-vector space of continuous functions $f:G(F^+)\backslash G(\Abb_{F^+}^{\infty})/U^p\longrightarrow L$. Since $G(F^+)\backslash G(\Abb_{F^+}^{\infty})/U^p$ is compact, it is a $p$-adic Banach space (for the sup norm) endowed with the linear continuous unitary action of $G_p$ by right translation on functions. In particular a unit ball is given by the $\Ocal_{L}$-submodule $\widehat S(U^p,\Ocal_{L})$ of continuous functions $f:G(F^+)\backslash G(\Abb_{F^+}^{\infty})/U^p\longrightarrow \Ocal_{L}$ and the corresponding residual representation is the $k_{L}$-vector space $S(U^p,k_{L})$ of locally constant functions $f:G(F^+)\backslash G(\Abb_{F^+}^{\infty})/U^p\longrightarrow k_{L}$ (a smooth admissible representation of $G_p$). Note that $S(U^p,k_L)=\varinjlim_{U_p}S(U^pU_p,k_L)$ where the inductive limit is taken over compact open subgroups $U_p$ of $G_p$ and where $S(U^pU_p,k_L)$ is the finite dimensional $k_L$-vector space of functions $f:G(F^+)\backslash G(\Abb_{F^+}^{\infty})/U^pU_p\longrightarrow k_L$. We also denote by $\widehat S(U^p,L)^{\rm an}\subset \widehat S(U^p,L)$ the $L$-subvector space of locally $\Q_p$-analytic vectors for the action of $G_p$ (\cite[\S7]{STdist}). This is a strongly admissible locally $\Q_p$-analytic representation of $G_p$.

We fix $S$ a finite set of finite places of $F^+$ that split in $F$ containing $S_p$ and the set of finite places $v\nmid p$ (that split in $F$) such that $U_v$ is not maximal. We consider the commutative spherical Hecke algebra:
$$\mathbb{T}^S:=\varinjlim_I\big(\bigotimes_{v\in I}\Ocal_L[U_v\backslash G(F^+_v)/U_v]\big),$$
the inductive limit being taken over finite sets $I$ of finite places of $F^+$ that split in $F$ such that $I\cap S=\emptyset$. This Hecke algebra naturally acts on the spaces $\widehat S(U^p,L)$, $\widehat S(U^p,L)^{\rm an}$, $\widehat S(U^p,\Ocal_L)$, $S(U^p,k_L)$ and $S(U^pU_p,k_L)$ (for any compact open subgroup $U_p$). If $C$ is a field, $\theta:\mathbb{T}^S\rightarrow C$ a ring homomorphism and $\rho: \Gcal_F\rightarrow \GL_n(C)$ a group homomorphism which is unramified at each finite place of $F$ above a place of $F^+$ which splits in $F$ and is not in $S$, we refer to \cite[\S2.4]{BHS} for what it means for $\rho$ to be {\it associated to} $\theta$.

Though we could state a more general classicality conjecture, it is convenient for us to assume right now the following two extra hypothesis: $p>2$ and $G$ quasi-split at each finite place of $F^+$ (these assumptions will be needed anyway for our partial results, note however that they imply that $4$ divides $n[F^+:\Q]$ which rules out the case $n=2$, $F^+=\Q$). We fix $\mathfrak{m}^S$ a maximal ideal of $\mathbb{T}^S$ of residue field $k_L$ (increasing $L$ if necessary) such that $\widehat S(U^p,L)_{\mathfrak{m}^S}\neq 0$, or equivalently $\widehat S(U^p,\Ocal_L)_{\mathfrak{m}^S}\neq 0$, or $S(U^p,k_L)_{\mathfrak{m}^S}=\varinjlim_{U_p}S(U^pU_p,k_L)_{\mathfrak{m}^S}\neq 0$, or $S(U^pU_p,k_L)_{\mathfrak{m}^S}\ne 0$ for some $U_p$ (note that $\widehat S(U^p,L)_{\mathfrak{m}^S}$ is then a closed subspace of $\widehat S(U^p,L)$ preserved by $G_p$). We denote by $\rhobar=\rhobar_{\mathfrak{m}^S}:\Gcal_F\rightarrow \GL_n(k_L)$ the unique absolutely semi-simple Galois representation associated to $\mathfrak{m}^S$ (see \cite[Prop.6.6]{Thorne} and note that the running assumption $F/F^+$ unramified in {\it loc.~cit.} is useless at this point). We assume $\mathfrak{m}^S$ {\it non-Eisenstein}, that is, $\rhobar$ absolutely irreducible. Then it follows from \cite[Prop.6.7]{Thorne} (with the same remark as above) that the spaces $\widehat S(U^p,L)_{\mathfrak{m}^S}$, $\widehat S(U^p,\Ocal_L)_{\mathfrak{m}^S}$ and $S(U^p,k_L)_{\mathfrak{m}^S}$ become modules over $R_{\rhobar,S}$, the complete local noetherian $\Ocal_L$-algebra of residue field $k_L$ pro-representing the functor of deformations $\rho$ of $\rhobar$ that are unramified outside $S$ and such that $\rho'\circ c\cong \rho\otimes \varepsilon^{n-1}$ (where $\rho'$ is the dual of $\rho$ and $c\in {\rm Gal}(F/F^+)$ is the complex conjugation). 

The continuous dual $(\widehat S(U^p,L)_{\mathfrak{m}^S}^{\rm an})'$ of $\widehat S(U^p,L)_{\mathfrak{m}^S}^{\rm an}:=(\widehat S(U^p,L)_{\mathfrak{m}^S})^{\rm an}=(\widehat S(U^p,L)^{\rm an})_{\mathfrak{m}^S}$ becomes a module over the global sections $\Gamma(\Xfrak_{\rhobar,S},\Ocal_{\Xfrak_{\rhobar,S}})$ where $\Xfrak_{\rhobar,S}:=(\Spf\, R_{\rhobar,S})^{\rig}$ (see for instance \cite[\S3.1]{BHS}). We denote by $Y(U^p,\rhobar)$ the \emph{eigenvariety of tame level $U^p$} (over $L$) defined in \cite{Emerton} (see also \cite[\S4.1]{BHS}) associated to $\widehat S(U^p,L)_{\mathfrak{m}^S}^{\rm an}$, that is, the support of the coherent $\Ocal_{\Xfrak_{\rhobar,S}\times \widehat T_{p,L}}$-module $(J_{B_p}(\widehat S(U^p,L)^{\rm an}_{\mathfrak{m}^S}))'$ on $\Xfrak_{\rhobar,S}\times \widehat T_{p,L}$ where $J_{B_p}$ is Emerton's locally $\Q_p$-analytic Jacquet functor with respect to the Borel $B_p$ and $(\cdot)'$ means the continuous dual. This is a reduced closed analytic subset of $\Xfrak_{\rhobar,S}\times \widehat T_{p,L}$ of dimension $n[F^+:\Q]$ whose points are:
\begin{equation}\label{points}
\left\{x=(\rho,\delta)\in \Xfrak_{\rhobar,S}\times \widehat T_{p,L}\ {\rm such\ that}\ \Hom_{T_p}\big(\delta,J_{B_p}(\widehat S(U^p,L)^{\rm an}_{\mathfrak{m}^S}[\mathfrak{p}_\rho]\otimes_{k(\mathfrak{p}_\rho)}k(x))\big)\ne 0\right\}
\end{equation}
where $\mathfrak{p}_\rho\subset R_{\rhobar,S}$ denotes the prime ideal corresponding to the point $\rho\in \Xfrak_{\rhobar,S}$ under the identification of the sets underlying $\Xfrak_{\rhobar,S}=(\Spf\, R_{\rhobar,S})^{\rig}$ and ${\rm Spm}\,R_{\rhobar,S}[1/p]$ (\cite[Lem.7.1.9]{deJong}) and where $k(\mathfrak{p}_\rho)$ is its residue field. We denote by $\omega:Y(U^p,\rhobar)\rightarrow \widehat T^0_{p,L}$ the composition $Y(U^p,\rhobar)\hookrightarrow \Xfrak_{\rhobar,S}\times \widehat T_{p,L}\twoheadrightarrow \widehat T_{p,L} \twoheadrightarrow \widehat T^0_{p,L}$.

\begin{rema}\label{changeu}
{\rm If ${U'}^p\subseteq {U}^p$ (and $S$ contains $S_p$ and the set of finite places $v\nmid p$ that split in $F$ such that $U'_v$ is not maximal), then a point $x=(\rho,\delta)$ of $Y(U^p,\rhobar)$ is also in $Y({U'}^p,\rhobar)$ since $\widehat S(U^p,L)^{\rm an}_{\mathfrak{m}^S}[\mathfrak{p}_\rho]\subseteq \widehat S({U'}^p,L)^{\rm an}_{\mathfrak{m}^S}[\mathfrak{p}_\rho]$ and $J_{B_p}$ is left exact.}
\end{rema}

We let $X_{\rm tri}^\square(\rhobar_p)$ be the product rigid analytic variety $\prod_{v\in S_p}X_{\rm tri}^\square(\rhobar_{\tilde{v}})$ (over $L$) where $\rhobar_{\tilde v}$ is the restriction of $\rho$ to the decomposition subgroup of ${\mathcal G}_F$ at $\tilde v$ (that we identify with $\Gcal_{F_{\tilde v}}={\rm Gal}(\overline F_{\tilde v}/F_{\tilde v})$) and where $X_{\rm tri}^\square(\rhobar_{\tilde{v}})$ is as in \S\ref{begin}. This is a reduced closed analytic subvariety of $(\Spf\, R^\square_{\rhobar_p})^{\rig}\times\widehat{T}_{p,L}$ where $R^\square_{\rhobar_p}:=\widehat{\bigotimes}_{v\in S_p}R^\square_{\rhobar_{\tilde{v}}}$. Identifying $B_v$ (resp. $T_v$) with the upper triangular (resp. diagonal) matrices of $\GL_n(F_{\tilde v})$ via $i_{\tilde v}$, we let $\delta_{B_v}:=|\cdot|_{F_{\tilde v}}^{n-1}\otimes |\cdot|_{F_{\tilde v}}^{n-3}\otimes \cdots \otimes |\cdot|_{F_{\tilde v}}^{1-n}$ be the modulus character of $B_v$ and define as in \cite[\S2.3]{BHS} an automorphism $\imath_v:\widehat T_v\buildrel\sim\over\rightarrow \widehat T_v$ by:
\begin{equation*}
\imath_v(\delta_1,\dots,\delta_n):=\delta_{B_v}\cdot(\delta_1,\dots,\delta_i\cdot(\varepsilon\circ\rec_{F_{\tilde v}})^{i-1},\dots,\delta_n\cdot(\varepsilon\circ\rec_{F_{\tilde v}})^{n-1})
\end{equation*}
(the twist by $\delta_{B_v}$ here ultimately comes from the same twist appearing in the definition of $J_{B_v}$). It then follows from \cite[Th.4.2]{BHS} that the morphism of rigid spaces:
\begin{eqnarray}
(\Spf\,R_{\rhobar,S})^{\rig}\times \widehat T_{p,L}&\longrightarrow &(\Spf\, R^\square_{\rhobar_p})^{\rig}\times\widehat{T}_{p,L}\\
\nonumber \big(\rho,(\delta_v)_{v\in S_p}\big)=\big(\rho,(\delta_{v,1},\dots,\delta_{v,n})_{v\in S_p}\big)&\longmapsto & \big((\rho|_{\Gcal_{F_{\tilde v}}})_{v\in S_p},(\imath_v^{-1}(\delta_{v,1},\dots,\delta_{v,n}))_{v\in S_p}\big)
\end{eqnarray}
induces a morphism of reduced rigid spaces over $L$:
\begin{equation}\label{eigenvartotrianguline}
Y(U^p,\rhobar)\longrightarrow X_{\rm tri}^\square(\rhobar_p)=\prod_{v\in S_p}X_{\rm tri}^\square(\rhobar_{\tilde v})
\end{equation}
(note that (\ref{eigenvartotrianguline}) is thus {\it not} compatible with the weight maps $\omega$ on both sides). We say that a point $x=(\rho,\delta)=(\rho,(\delta_v)_{v\in S_p})\in Y(U^p,\rhobar)$ is crystalline (resp. dominant, resp. strictly dominant, resp.~crystalline strictly dominant very regular etc.) if for each $v\in S_p$ its image in $X_{\rm tri}^\square(\rhobar_{\tilde v})$ via (\ref{eigenvartotrianguline}) is (see \S\ref{begin} and Definition \ref{veryreg}). Due to the twist $\imath_v$ beware that $x=(\rho,\delta)\in Y(U^p,\rhobar)$ is strictly dominant if and only if $\delta\vert_{T_v\cap K_v}$ is (algebraic) dominant for each $v\in S_p$.

Let $\delta\in \widehat T_{p,L}$ be any locally algebraic character. Then we can write $\delta=\delta_{\lambda}\delta_{\rm sm}$ in $\widehat T_{p,L}$ where $\lambda=(\lambda_v)_{v\in S_p}\in \prod_{v\in S_p}(\Z^n)^{\Hom(F_{\tilde v},L)}$, $\delta_\lambda:=\prod_{v\in S_p}\delta_{\lambda_v}$ (see \S\ref{begin} for $\delta_{\lambda_v}\in \widehat T_{v,L}$) and $\delta_{\rm sm}$ is a smooth character of $T_{p}$ with values in $k(\delta)$ (the residue field of the point $\delta\in \widehat T_{p,L}$). Using the theory of Orlik and Strauch (\cite{OrlikStrauch}), we define as in \cite[(3.7)]{BHS} the following strongly admissible locally $\Q_p$-analytic representation of $G_p$ over $k(\delta)$:
$$\Fcal_{\overline B_p}^{G_p}(\delta):=\Fcal_{\overline B_p}^{G_p}\big((U(\mathfrak{g}_L)\otimes_{U({\overline {\mathfrak b}}_L)}(-\lambda))^\vee, \delta_{\rm sm}\delta_{B_p}^{-1}\big)$$
where $\delta_{B_p}:=\prod_{v\in S_p}\delta_{B_v}$ and where we refer to \cite[\S3.5]{BHS} for the details and notation. Recall that $\Fcal_{\overline B_p}^{G_p}(\delta)$ has the same constituents as the locally $\Q_p$-analytic principal series $({\rm Ind}_{\overline B_p}^{G_p}\delta_{\lambda}\delta_{\rm sm}\delta_{B_p}^{-1})^{\rm an}=({\rm Ind}_{\overline B_p}^{G_p}\delta\delta_{B_p}^{-1})^{\rm an}$ but in the ``reverse order'' (at least generically). If $\lambda$ is dominant (that is $\lambda_v$ is dominant for each $v$ in the sense of \S\ref{begin}), we denote by ${\rm LA}(\delta)$ the locally algebraic representation:
\begin{multline}\label{ladelta}
{\rm LA}(\delta):=\Fcal_{\overline B_p}^{G_p}(L(\lambda)',\delta_{\rm sm}\delta_{B_p}^{-1})=\Fcal_{G_p}^{G_p}\big(L(\lambda)',({\rm Ind}_{\overline B_p}^{G_p}\delta_{\rm sm}\delta_{B_p}^{-1})^{\rm sm}\big)\\
=L(\lambda)\otimes_L \big({\rm Ind}_{\overline B_p}^{G_p}\delta_{\rm sm}\delta_{B_p}^{-1}\big)^{\rm sm}
\end{multline}
where $L(\lambda)$ is the simple $U(\mathfrak{g}_L)$-module over $L$ of highest weight $\lambda$ relative to the Lie algebra of $B_p$ (which is finite dimensional over $L$ since $\lambda$ is dominant) that we see as an irreducible algebraic representation of $G_p$ over $L$, where $L(\lambda)'$ is its dual, and where $(-)^{\rm sm}$ denotes the smooth Borel induction over $k(\delta)$ (the second equality in (\ref{ladelta}) following from \cite[Prop.4.9(b)]{OrlikStrauch}). Arguing as in \cite[\S6]{OrlikStrauch} (note that $L(\lambda)'$ is the unique irreducible subobject of $(U(\mathfrak{g}_L)\otimes_{U({\overline {\mathfrak b}}_L)}(-\lambda))^\vee$), it easily follows from \cite[Th.5.8]{OrlikStrauch} (see also \cite[Th.2.3(iii)]{BreuilAnalytiqueI}) and \cite[\S5.1]{HumBGG} that ${\rm LA}(\delta)$ is identified with the maximal locally $\Q_p$-algebraic quotient of $\Fcal_{\overline B_p}^{G_p}(\delta)$ (or the maximal locally algebraic subobject of $({\rm Ind}_{\overline B_p}^{G_p}\delta\delta_{B_p}^{-1})^{\rm an}$). 

It follows from (\ref{points}) together with \cite[Th.4.3]{BreuilAnalytiqueII} that a point $x=(\rho,\delta)\in  \Xfrak_{\rhobar,S}\times\widehat T_{p,L}$ lies in $Y(U^p,\rhobar)$ if and only if:
\begin{equation}\label{adj}
\Hom_{T_p}\big(\delta,J_{B_p}(\widehat S(U^p,L)^{\rm an}_{\mathfrak{m}^S}[\mathfrak{p}_\rho]\otimes_{k(\mathfrak{p}_\rho)}k(x))\big)\cong\Hom_{G_p}\big(\Fcal_{\overline B_p}^{G_p}(\delta),\widehat S(U^p,L)_{\mathfrak{m}^S}^{\rm an}[\mathfrak{p}_{\rho}]\otimes_{k(\mathfrak{p}_\rho)}k(x)\big)\ne 0.
\end{equation}

\begin{defi}\label{defclass}
A point $x=(\rho,\delta)\in Y(U^p,\rhobar)$ is called \emph{classical} if there exists a nonzero continuous $G_p$-equivariant morphism:
$$\Fcal_{\overline B_p}^{G_p}(\delta)\longrightarrow \widehat S(U^p,L)_{\mathfrak{m}^S}^{\rm an}[\mathfrak{p}_{\rho}]\otimes_{k(\mathfrak{p}_\rho)}k(x)$$
that factors through the locally $\Q_p$-algebraic quotient ${\rm LA}(\delta)$ of $\Fcal_{\overline B_p}^{G_p}(\delta)$ (equivalently $(\rho,\delta)$ is classical if $\Hom_{G_p}({\rm LA}(\delta),\widehat S(U^p,L)_{\mathfrak{m}^S}[\mathfrak{p}_{\rho}]\otimes_{k(\mathfrak{p}_\rho)}k(x))\ne 0$).
\end{defi}

\begin{rema}
{\rm (i) This definition is \cite[Def.3.14]{BHS}.\\
(ii) It seems reasonnable to expect that if $x=(\rho,\delta)\in Y(U^p,\rhobar)$ is classical, then in fact {\it any} continuous $G_p$-equivariant morphism $\Fcal_{\overline B_p}^{G_p}(\delta)\longrightarrow \widehat S(U^p,L)_{\mathfrak{m}^S}^{\rm an}[\mathfrak{p}_{\rho}]\otimes_{k(\mathfrak{p}_\rho)}k(x)$ factors through ${\rm LA}(\delta)$. See the last statement of Corollary \ref{mainclassic} below for a partial result in that direction.}
\end{rema}

We fix an algebraic closure $\overline\Q_p$ of $L$ and embeddings $j_\infty:\overline\Q\hookrightarrow \C$, $j_p:\overline\Q\hookrightarrow \overline\Q_p$. Recall that, if $\pi=\pi_{\infty}\otimes_{\C}\pi_f$ is an automorphic representation of $G(\Abb_{F^+})$ over $\C$ where $\pi_{\infty}$ (resp. $\pi_f$) is a representation of $G(F^+\otimes_\Q\mathbb{R})$ (resp. of $G(\Abb_{F^+}^\infty)$), then due to the compactness of $G(F^+\otimes_\Q\mathbb{R})$, we have that $\pi_\infty$ is a finite dimensional irreducible representation that comes from an algebraic representation of $\Res_{F^+/\Q}G$ over $\C$ (argue as in \cite[\S\S6.2.3,6.7]{BelChe}). Moreover, arguing again as in {\it loc. cit.}, $\pi_\infty$ (resp. $\pi_f$) has a $\Qbar$-structure given by $j_\infty$ which is stable under the action of $(\Res_{F^+/\Q}G)(\Qbar)$ (resp. of $G(\Abb_{F^+}^\infty)$). Hence the scalar extension of the $\Qbar$-structure of $\pi_\infty$ to $\overline \Q_p$ via $j_p$ is endowed with an action of $(\Res_{F^+/\Q}G)(\overline \Q_p)$, thus in particular of $(\Res_{F^+/\Q}G)(\Q_p)=G(F^+\otimes_{\Q}\Q_p)=G_p$. This latter representation of $G_p$ is easily seen to be defined over $L$ and of the form $L(\lambda)$ for a dominant $\lambda$ as above. We say that $\pi_\infty$ {\it is of weight $\lambda$} if the resulting representation of $G_p$ is $L(\lambda)$. 

For the sake of completeness, we recall the following proposition showing that Definition \ref{defclass} coincides with the usual classicality definition.

\begin{prop}\label{compaclas}
A strictly dominant point $x=(\rho,\delta)\in Y(U^p,\rhobar)$, that is such that $\omega(x)=\delta_{\lambda}$ for some dominant $\lambda\in \prod_{v\in S_p}(\Z^n)^{\Hom(F_{\tilde v},L)}$, is classical if and only if there exists an automorphic representation $\pi=\pi_{\infty}\otimes_{\C}\pi_f^p\otimes_{\C}\pi_p$ of $G(\Abb_{F^+})$ over $\C$ such that the following conditions hold:
\begin{enumerate}
\item[(i)]the $G(F^+\otimes_\Q\mathbb{R})$-representation $\pi_\infty$ is of weight $\lambda$ in the above sense;\\
\item[(ii)]the $\Gcal_F$-representation $\rho$ is the Galois representation associated to $\pi$ (see proof below);\\
\item[(iii)]the invariant subspace $(\pi_f^p)^{U^p}$ is nonzero;\\
\item[(iv)] the $G_p$-representation $\pi_p$ is a quotient of $\big({\rm Ind}_{\overline B_p}^{G_p} \delta\delta_{\lambda}^{-1}\delta_{B_p}^{-1}\big)^{\rm sm}\otimes_{k(\delta)}\overline \Q_p$.
\end{enumerate}
If moreover $F$ is unramified over $F^+$ and $U_v$ is hyperspecial when $v$ is inert in $F$, then such a $\pi$ is unique and appears with multiplicity $1$ in $L^2(G(F^+)\backslash G(\Abb_{F^+}),\C)$.
\end{prop}
\begin{proof}
Let $W$ be any linear representation of $G_p$ over an $L$-vector space and $U$ any compact open subgroup of $G(\Abb_{F^+}^{\infty})$, we define $S(U,W)$ to be the $L$-vector space of functions $f:G(F^+)\backslash G(\Abb_{F^+}^{\infty})\longrightarrow W$ such that $f(gu)=u_p^{-1}(f(g))$ for $g\in G(\Abb_{F^+}^{\infty})$ and $u\in U$, where $u_p$ is the projection of $u$ in $G_p$. Fixing $U^p$ as previously, we define $S(U^p,W):=\varinjlim_{U_p}S(U^pU_p,W)$ (inductive limit taken over compact open subgroups $U_p$ of $G_p$) endowed with the linear left action of $G_p$ given by $(h_p\cdot f)(g):=h_p(f(gh_p))$ ($h_p\in G_p$, $g\in G(\Abb_{F^+}^{\infty})$) where the second $h_p$ is seen in $G(\Abb_{F^+}^{\infty})$ in the obvious way. Note that $\mathbb{T}^S$ also naturally acts on $S(U^p,W)$ (the representation $W$ here playing no role since this action is ``outside $p$''). Then it follows from \cite[\S7.1.4]{EGH} that there is an isomorphism of smooth representations of $G_p$ over $\overline\Q_p$:
\begin{equation}\label{class}
S(U^p,L(\lambda)')\otimes_{L}\overline \Q_p\cong \bigoplus_{\pi}\big[\big((\pi_f^p)^{U^p}\otimes_{\overline \Q}\pi_p\big)\otimes_{\overline\Q,j_p}\overline \Q_p\big]^{\oplus m(\pi)}
\end{equation}
where the direct sum is over the automorphic representations $\pi=\pi_\infty\otimes_{\C}\pi_p$ of $G(\Abb_{F^+})$ such that $\pi_\infty$ is of weight $\lambda$ and $(\pi_f^p)^{U^p}\ne 0$ (we take the $\Qbar$-structures) and where $m(\pi)$ is the multiplicity of $\pi$ in $L^2(G(F^+)\backslash G(\Abb_{F^+}),\C)$. We then say that a point $\rho\in \Xfrak_{\rhobar,S}$ is the Galois representation associated to $\pi$ (with $\pi_\infty$ of weight $\lambda$) if we have:
$$\big[\big((\pi_f^p)^{U^p}\otimes_{\overline \Q}\pi_p\big)\otimes_{\overline\Q,j_p}\overline \Q_p\big]^{\oplus m(\pi)}\subseteq S(U^p,L(\lambda)')_{\mathfrak{m}^S}[\mathfrak{p}_\rho]\otimes_{k(\mathfrak{p}_\rho)}\overline \Q_p$$
where $\mathfrak{p}_\rho$ is as in (\ref{points}) (and $R_{\rhobar,S}$ acts on $S(U^p,L(\lambda)')_{\mathfrak{m}^S}$ again using \cite[Prop.6.7]{Thorne}). Note that $S(U^p,L(\lambda)')_{\mathfrak{m}^S}[\mathfrak{p}_\rho]\otimes_{k(\mathfrak{p}_\rho)}\overline \Q_p\ne 0$ (equivalently $S(U^p,L(\lambda)')_{\mathfrak{m}^S}[\mathfrak{p}_\rho]\ne 0$) if and only if there exists an automorphic representation $\pi$ such that $\pi_\infty$ is of weight $\lambda$, $(\pi_f^p)^{U^p}\ne 0$ and $\rho$ is the Galois representation associated to $\pi$.

Let $\widehat S(U^p,L)^{\lambda-\rm la}\subset \widehat S(U^p,L)^{\rm an}$ be the closed $G_p$-subrepresentation of locally $L(\lambda)$-algebraic vectors, that is the $L$-subvector space of $\widehat S(U^p,L)^{\rm an}$ (or equivalently of $\widehat S(U^p,L)$) of vectors $v$ such that there exists a compact open subgroup $U_p$ of $G_p$ such that the $U_p$-subrepresentation generated by $v$ in $\widehat S(U^p,L)\vert_{U_p}$ is isomorphic to $(L(\lambda)\vert_{U_p})^{\oplus d}$ for some positive integer $d$. Note that the subspace $\widehat S(U^p,L)^{\lambda-\rm la}$ is preserved under the action of $\mathbb{T}^S$ (since the latter commutes with $G_p$). Then it follows from \cite[Prop.3.2.4]{Emerton} and its proof that there is an isomorphism of locally $\Q_p$-algebraic representations of $G_p$ over $L$ which is $\mathbb{T}^S$-equivariant (with the action of $\mathbb{T}^S$ on the right hand side given by its action on $S(U^p,L(\lambda)')$):
\begin{equation*}
\widehat S(U^p,L)^{\lambda-\rm la}\cong L(\lambda)\otimes_L S(U^p,L(\lambda)').
\end{equation*}
We then deduce a $G_p$-equivariant isomorphism of $R_{\rhobar,S}$-modules:
\begin{equation}\label{lambdaalgm}
\widehat S(U^p,L)^{\lambda-\rm la}_{\mathfrak{m}^S}\cong L(\lambda)\otimes_L S(U^p,L(\lambda)')_{\mathfrak{m}^S}
\end{equation}
where $\widehat S(U^p,L)^{\lambda-\rm la}_{\mathfrak{m}^S}:=(\widehat S(U^p,L)^{\lambda-\rm la})_{\mathfrak{m}^S}=(\widehat S(U^p,L)_{\mathfrak{m}^S})^{\lambda-\rm la}$.

Now let $x=(\rho,\delta)\in Y(U^p,\rhobar)$ with $\omega(x)=\delta_{\lambda}$ for $\lambda$ dominant and define $\mathfrak{p}_\rho$ as in (\ref{points}). From Definition \ref{defclass} and the definition of $\widehat S(U^p,L)^{\lambda-\rm la}$, we get that the point $x$ is classical if and only if there exists a nonzero $G_p$-equivariant morphism:
$$L(\lambda)\otimes_L \big({\rm Ind}_{\overline B_p}^{G_p} \delta\delta_{\lambda}^{-1}\delta_{B_p}^{-1}\big)^{\rm sm} \longrightarrow \widehat S(U^p,L)_{\mathfrak{m}^S}^{\lambda-\rm la}[\mathfrak{p}_\rho]\otimes_{k(\mathfrak{p}_\rho)}k(x)$$
if and only if by (\ref{lambdaalgm}) there exists a nonzero $G_p$-equivariant morphism:
$$\big({\rm Ind}_{\overline B_p}^{G_p} \delta\delta_{\lambda}^{-1}\delta_{B_p}^{-1}\big)^{\rm sm}\longrightarrow S(U^p,L(\lambda)')_{\mathfrak{m}^S}[\mathfrak{p}_\rho]\otimes_{k(\mathfrak{p}_\rho)}k(x)$$
if and only if by (\ref{class}) there exists an automorphic representation $\pi=\pi_\infty\otimes_{\C}\pi_f^p\otimes_{\C}\pi_p$ of $G(\Abb_{F^+})$ such that $\pi_\infty$ is of weight $\lambda$, $(\pi_f^p)^{U^p}\ne 0$, $\rho$ is the Galois representation associated to $\pi$ and $\pi_p$ is a quotient of $({\rm Ind}_{\overline B_p}^{G_p} \delta\delta_{\lambda}^{-1}\delta_{B_p}^{-1})^{\rm sm}\otimes_{k(\delta)}\overline \Q_p$.

Now let us prove the last assertion. According to \cite[Cor.5.3]{Labesse}, there exists an isobaric representation $\Pi=\Pi_1\boxplus\cdots\boxplus\Pi_r$ where $n=m_1+\cdots+m_r$ and $\Pi_i$ nonzero automorphic representations of $\mathrm{GL}_{m_i}(\mathbb{A}_F)$ occuring in the discrete spectrum such that $\Pi$ is a weak base change of $\pi$ in the sense of \cite[\S4.10]{Labesse}. Since $\overline{\rho}$, hence $\rho$, are absolutely irreducible, we have $r=1$ and $\Pi=\Pi_1$ cuspidal. The equality $m(\pi)=1$ then follows from \cite[Th.5.4]{Labesse} (which uses the extra assumption $F/F^+$ unramified). The uniqueness of $\pi$ is a consequence of the strong base change theorem \cite[Th.5.9]{Labesse} together with the fact that $\pi_v$ is unramified at finite places $v$ of $F^+$ which are inert in $F$ (which uses the extra assumption $U_v$ hyperspecial for $v$ inert) and the fact that the $L$-packets at finite places of $F^+$ which are split in $F$ are singletons.
\end{proof}

\begin{rema}
{\rm With the notation of Proposition \ref{compaclas}, write $\delta\delta_{\lambda}^{-1}=(\delta_{{\rm sm},v,1},\dots,\delta_{{\rm sm},v,n})_{v\in S_p}$, if moreover $\delta_{{\rm sm},v,i}/\delta_{{\rm sm},v,j}\notin \{1,\vert\cdot\vert_{F_{\tilde v}}^{-2}\}$ for $1\leq i\ne j\leq n$ and $v\in S_p$, then we see from (iv) of Proposition \ref{compaclas} that $\pi_p\cong ({\rm Ind}_{B_p}^{G_p} \delta\delta_{\lambda}^{-1})^{\rm sm}\otimes_{k(\delta)}\overline \Q_p$.}
\end{rema}

We then have the following conjecture, which by Proposition \ref{compaclas} is essentially a consequence of the Fontaine-Mazur conjecture and the Langlands philosophy, and which is the natural generalization in the context of definite unitary groups of the main result of \cite{Kisinoverconvergent} for $\GL_2/\Q$ (in the crystalline case).
  
\begin{conj}\label{classiconj}
Let $x=(\rho,\delta)\in Y(U^p,\rhobar)$ be a crystalline strictly dominant point. Then $x$ is classical.
\end{conj}

\begin{rema}\label{extra}
{\rm We didn't seek to state the most general classicality conjecture. Obviously, the assumptions that $p>2$ and $G$ is quasi-split at each finite place of $F^+$ shouldn't be crucial, and one could replace crystalline by de Rham.}
\end{rema}

\subsection{Proof of the main classicality result}\label{firstclassical}

We prove a criterion for classicality (Theorem \ref{classicalitycrit}) in terms of the patched eigenvariety of \cite{BHS}, which itself builds on the construction in \cite{CEGGPS} of a ``big'' patching module $M_\infty$. We use it to prove our main classicality result (Corollary \ref{mainclassic}).

We keep the notation of \S\ref{classic} and make the following extra assumptions (which are required for the construction of $M_\infty$): $F$ is unramified over $F^+$, $U_v$ is hyperspecial if $v$ is inert in $F$, $\rhobar(\mathcal{G}_{F(\zeta_p)})$ is adequate in the sense of \cite[Def.2.3]{Thorne} and $\zeta_p\notin \overline F^{\ker(\rhobar\otimes \rhobar')}$. For instance if $p>2n+1$ and $\rhobar\vert_{\mathcal{G}_{F(\zeta_p)}}$ is (still) absolutely irreducible, then $\rhobar(\mathcal{G}_{F(\zeta_p)})$ is automatically adequate (\cite[Th.9]{GHTT}). We first briefly recall some notation, definitions and statements and refer to \cite[\S3.2]{BHS} for more details on what follows. 

We let $R_{\rhobar_{\tilde v}}^{\overline\square}$ be the maximal reduced and $\Z_p$-flat quotient of the framed local deformation ring $R_{\rhobar_{\tilde v}}^\square$ and set:
$$R^{\rm loc}:=\widehat{\bigotimes}_{v\in S} R_{\rhobar_{\tilde v}}^{\overline\square},\ \ R_{\rhobar^p}:=\widehat{\bigotimes}_{v\in S\backslash S_p}R_{\rhobar_{\tilde v}}^{\overline\square},\ \ R_{\rhobar_p}:=\widehat{\bigotimes}_{v\in S_p}R_{\rhobar_{\tilde v}}^{\overline\square},\ \ R_\infty:=R^{\rm loc}\dbl x_1\dots,x_g\dbr$$
where $g\geq 1$ is some integer which will be fixed below. We let $\Xfrak_{\rhobar^p}:=(\Spf\, R_{\rhobar^p})^{\rig}$, $\Xfrak_{\rhobar_p}:=(\Spf R_{\rhobar_p})^{\rig}$ and $\Xfrak_\infty:=(\Spf\, R_\infty)^{\rig}$ so that:
\begin{equation}\label{product}
\Xfrak_\infty=\Xfrak_{\rhobar^p}\times \Xfrak_{\rhobar_p}\times \Ubb^g
\end{equation}
where $\Ubb:=(\Spf\,\Ocal_L\dbl y \dbr)^{\rig}$ is the open unit disc over $L$. We also define $S_\infty:=\Ocal_L\dbl y_1,\dots,y_t\dbr$ where $t:=g+[F^+:\Q]\tfrac{n(n-1)}{2}+|S|n^2$ and $\mathfrak{a}:=(y_1,\dots,y_t)$ (an ideal of $S_\infty$).

Thanks to Remark \ref{changeu} and Proposition \ref{compaclas} we can (and do) assume that the tame level $U^p$ is small enough so that we have:
\begin{equation}\label{net}
G(F)\cap (hU^pK_ph^{-1})=\{1\}\ \ {\rm for\ all\ }h\in G(\Abb^\infty_{F^+})
\end{equation}
(indeed, let $w\nmid p$ be a finite place of $F^+$ that splits in $F$ such that $U_w$ is maximal, replace $U^p$ by $U'^p:=U'_w\prod_{v\ne w}U_v$ where $U'_w$ is small enough so that $U'^p$ satisfies (\ref{net}), and use Proposition \ref{compaclas} and local-global compatibility at $w$ to deduce classicality in level $U^p$ from classicality in level $U'^p$). Then there is a quotient $R_{\rhobar,S}\twoheadrightarrow R_{\rhobar, \Scal}$ such that the action of $R_{\rhobar,S}$ on $\widehat S(U^p,L)_{\mathfrak{m}^S}$ factors through $R_{\rhobar,\Scal}$, an integer $g\geq 1$ and:
\begin{enumerate}
\item[(i)]a continuous $R_\infty$-admissible (see \cite[Def.3.1]{BHS}) unitary representation $\Pi_\infty$ of $G_p$ over $L$ together with a $G_p$-stable and $R_\infty$-stable unit ball $\Pi_\infty^\circ\subset \Pi_\infty$;\\
\item[(ii)]a morphism of local $\Ocal_L$-algebras $S_\infty\rightarrow R_\infty$ such that $M_\infty:=\Hom_{\Ocal_L}(\Pi_\infty^\circ,\Ocal_L)$ is finite projective as an $S_\infty\dbl K_p\dbr$-module;\\
\item[(iii)]compatible isomorphisms $R_\infty/\mathfrak{a}R_\infty\cong R_{\rhobar,\Scal}$ and $\Pi_\infty[\mathfrak{a}]\cong \widehat S(U^p,L)_{\mathfrak{m}^S}$ where the latter is $G_p$-equivariant.
\end{enumerate}
We then define the \emph{patched eigenvariety} $X_p(\rhobar)$ as the support of the coherent $\Ocal_{\Xfrak_\infty\times \widehat T_{p,L}}$-module $\mathcal{M}_\infty=(J_{B_p}(\Pi_\infty^{R_\infty-{\rm an}}))'$ on $\Xfrak_\infty\times \widehat T_{p,L}$ (see \cite[Def.3.2]{BHS} for $\Pi_\infty^{R_\infty-{\rm an}}$; strictly speaking $(J_{B_p}(\Pi_\infty^{R_\infty-{\rm an}}))'$ is the global sections of the sheaf $\mathcal{M}_\infty$). This is a reduced closed analytic subset of $\Xfrak_\infty\times \widehat T_{p,L}$ (\cite[Cor.3.19]{BHS}) whose points are (\cite[Prop3.7]{BHS}):
\begin{equation}\label{pointpatching}
\left\{x=(y,\delta)\in \Xfrak_\infty\times \widehat T_{p,L}\ {\rm such\ that}\ \Hom_{T_p}\big(\delta,J_{B_p}(\Pi_\infty^{R_\infty-\rm an}[\mathfrak{p}_y]\otimes_{k(\mathfrak{p}_y)}k(x))\big)\ne 0\right\}
\end{equation}
where $\mathfrak{p}_y\subset R_\infty$ denotes the prime ideal corresponding to the point $y\in \Xfrak_\infty$ (under the identification of the sets underlying $\Xfrak_\infty$ and ${\rm Spm}\,R_\infty[1/p]$) and ${k(\mathfrak{p}_y)}$ is the residue field of $\mathfrak{p}_y$. It follows from the proof of \cite[Th.4.2]{BHS} that we can recover the eigenvariety $Y(U^p,\rhobar)$ as the reduced Zariski-closed subspace of $X_p(\rhobar)$ underlying the vanishing locus of the ideal $\mathfrak{a}\Gamma(\Xfrak_\infty,\Ocal_{\Xfrak_\infty})$.

\begin{lemm}\label{CM}
The coherent sheaf $\mathcal{M}_\infty$ is Cohen-Macaulay over $X_p(\overline{\rho})$.
\end{lemm}
\begin{proof}
From the proof of \cite[Prop.3.10]{BHS} (to which we refer the reader for more details) we deduce that there exists an admissible affinoid covering $(U_i)_i$ of $X_p(\overline{\rho})$ such that $\Gamma(U_i,\mathcal{M}_\infty)$ is a finite projective module over a ring $\mathcal{O}_{\mathcal{W}_\infty}(W_i)$ whose action on $\Gamma(U_i,\mathcal{M}_\infty)$ factors through a ring homomorphism $\mathcal{O}_{\mathcal{W}_\infty}(W_i)\rightarrow\mathcal{O}_{X_p(\overline{\rho})}(U_i)$. Consequently we can deduce from \cite[Prop.16.5.3]{EGAIV1} that $\Gamma(U_i,\mathcal{M}_\infty)$ is a Cohen-Macaulay $\mathcal{O}_{X_p(\overline{\rho})}(U_i)$-module.
\end{proof}

It follows from \cite[Th.3.20]{BHS} that the isomorphism of rigid spaces:
\begin{eqnarray*}
\Xfrak_\infty\times \widehat T_{p,L}&\buildrel\sim\over\longrightarrow &\Xfrak_\infty\times\widehat{T}_{p,L}\\
\nonumber \big(x,(\delta_v)_{v\in S_p}\big)=\big(x,(\delta_{v,1},\dots,\delta_{v,n})_{v\in S_p}\big)&\longmapsto & \big(x,(\imath_v^{-1}(\delta_{v,1},\dots,\delta_{v,n}))_{v\in S_p}\big)
\end{eqnarray*}
induces via (\ref{product}) a morphism of reduced rigid spaces over $L$:
\begin{equation}\label{patchedeigenvartoXtri}
X_p(\rhobar)\longrightarrow \Xfrak_{\rhobar^p}\times X_{\rm tri}^\square(\rhobar_p)\times \Ubb^g
\end{equation}
which identifies the source with a union of irreducible components of the target. Note that the composition:
$$Y(U^p,\rhobar)\hookrightarrow X_p(\rhobar)\buildrel{(\ref{patchedeigenvartoXtri})}\over\longrightarrow\Xfrak_{\rhobar^p}\times X_{\rm tri}^\square(\rhobar_p)\times \Ubb^g\twoheadrightarrow X_{\rm tri}^\square(\rhobar_p)$$
is the map (\ref{eigenvartotrianguline}). An irreducible component of the right hand side of (\ref{patchedeigenvartoXtri}) is of the form $\Xfrak^p\times Z\times \Ubb^g$ where $\Xfrak^p$ (resp. $Z$) is an irreducible component of $\Xfrak^p$ (resp. $X_{\rm tri}^\square(\rhobar_p)$). Given an irreducible component $\Xfrak^p\subseteq \Xfrak_{\rhobar^p}$, we denote by $X_{\rm tri}^{\Xfrak^p \rm-aut}(\rhobar_p)\subseteq X_{\rm tri}^\square(\rhobar_p)$ the union (possibly empty) of those irreducible components $Z\subseteq X_{\rm tri}^\square(\rhobar_p)$ such that $\Xfrak^p\times Z\times \Ubb^g$ is an irreducible component of $X_p(\rhobar)$ via (\ref{patchedeigenvartoXtri}). The morphism (\ref{patchedeigenvartoXtri}) thus induces an isomorphism:
\begin{equation}\label{union}
X_p(\rhobar)\buildrel\sim\over\longrightarrow \bigcup_{\Xfrak^p} \big(\Xfrak^p\times X_{\rm tri}^{\Xfrak^p \rm-aut}(\rhobar_p)\times \Ubb^g\big)
\end{equation}
the union (inside $\Xfrak_{\rhobar^p}\times X_{\rm tri}^\square(\rhobar_p)\times \Ubb^g$) being over the irreducible components $\Xfrak^p$ of $\Xfrak_{\rhobar^p}$.

We now state and prove the main result of this section, which gives a criterion for classicality on $Y(U^p,\rhobar)$. Recall that, given a crystalline strictly dominant point $x_v=(r_v,\delta_v)\in X_{\rm tri}^\square(\rhobar_{\tilde v})$ such that the geometric Frobenius eigenvalues on ${\rm WD}(r_v)$ are pairwise distinct and $V_v\subseteq X_{\rm tri}^\square(\rhobar_{\tilde v})$ a sufficiently small open neighbourhood of $x_v$, we have constructed in Corollary $\ref{defnZtri(x)}$ an irreducible component $Z_{{\rm tri},V_v}(x_v)$ of $V_v$ containing $x_v$. 

\begin{theo}\label{classicalitycrit}
Let $x=(\rho,\delta)\in Y(U^p,\rhobar)$ be a crystalline strictly dominant point such that the eigenvalues $\varphi_{\tilde v,1},\dots,\varphi_{\tilde v,n}$ of the geometric Frobenius on the (unramified) Weil-Deligne representation ${\rm WD}(\rho\vert_{\Gcal_{F_{\tilde v}}})$ satisfy $\varphi_{\tilde v,i}\varphi_{\tilde v,j}^{-1}\notin \{1,q_v\}$ for all $i\neq j$ and all $v\in S_p$. Let $\Xfrak^p\subset \Xfrak_{\rhobar^p}$ be an irreducible component such that $x\in \Xfrak^p\times X_{\rm tri}^{\Xfrak^p \rm-aut}(\rhobar_p)\times \Ubb^g\subseteq X_p(\rhobar)$ via (\ref{union}), let $x_v\in X_{\rm tri}^\square(\rhobar_{\tilde v})$ (for $v\in S_p$) be the image of $x$ via:
$$\Xfrak^p\times X_{\rm tri}^{\Xfrak^p \rm-aut}(\rhobar_p)\times \Ubb^g\twoheadrightarrow X_{\rm tri}^{\Xfrak^p \rm-aut}(\rhobar_p)\hookrightarrow X_{\rm tri}^\square(\rhobar_p)\twoheadrightarrow X_{\rm tri}^\square(\rhobar_{\tilde v})$$
and let $V_v\subseteq X_{\rm tri}^\square(\rhobar_{\tilde v})$ (for $v\in S_p$) be a sufficiently small open neighbourhood of $x_v$ so that $Z_{{\rm tri},V_v}(x_v)\subseteq V_v$ is defined.
If we have:
$$\prod_{v\in S_p} Z_{{\rm tri},V_v}(x_v)\subseteq X_{\rm tri}^{\Xfrak^p\rm-aut}(\rhobar_p)$$
then the point $x$ is classical.
\end{theo}
\begin{proof}
Let us write $\mathfrak{p}_y\subset R_\infty$ for the prime ideal corresponding to the image $y$ of $x$ in $\Xfrak_\infty$ via $Y(U^p,\rhobar)\hookrightarrow X_p(\rhobar)\hookrightarrow \Xfrak_\infty\times \widehat{T}_{p,L}\twoheadrightarrow \Xfrak_\infty$ and $\mathfrak{p}_\rho\subset R_{\rhobar,S}$ for the prime ideal corresponding to the global representation $\rho$. Then it follows from property (iii) above that we have $\mathfrak{a}R_\infty\subseteq \mathfrak{p}_{\rho}$ and $\widehat S(U^p,L)_{\mathfrak{m}^S}[\mathfrak{p}_{\rho}]=\Pi_{\infty}[\mathfrak{p}_y]$. From Definition \ref{defclass} we thus need to show that $\Hom_{G_p}({\rm LA}(\delta),\Pi_{\infty}[\mathfrak{p}_y]\otimes_{k(\mathfrak{p}_y)}k(x))\ne 0$.

As in \S\ref{classic} let us write $\delta=\delta_{\lambda}\delta_{\rm sm}$ with $\lambda=(\lambda_v)_{v\in S_p}$ and:
$$\lambda_v:=(\lambda_{v,\tau,i})_{1\leq i\leq n,\tau\in\Hom(F_{\tilde v},L)}\in \Z^{\Hom(F_{\tilde v},L)}$$
(recall that each $\lambda_v$ is dominant with respect to $B_v$). Consider the usual induction with compact support ${\rm ind}_{K_p}^{G_p}(L(\lambda)\vert_{K_p})$ (resp. ${\rm ind}_{K_v}^{G_v}(L(\lambda_v)\vert_{K_v})$) where $L(\lambda_p)$ (resp. $L(\lambda_v)$) is the irreducible algebraic representation of $G_p$ (resp. $G_v$) over $L$ of highest weight $\lambda$ (resp. $\lambda_v$) with respect to $B_p$ (resp. $B_v$). Let $\mathcal{H}(\lambda):=\End_{G_p}({\rm ind}_{K_p}^{G_p}L(\lambda))$ and $\mathcal{H}(\lambda_v):=\End_{G_v}({\rm ind}_{K_v}^{G_v}L(\lambda_v))$ be the respective convolution algebras (which are commutative $L$-algebras), we have $\mathcal{H}(\lambda)\cong \prod_{v\in S_p}\mathcal{H}(\lambda_v)$. Moreover by Frobenius reciprocity:
$$\Pi_\infty(\lambda):=\Hom_{K_p}(L(\lambda),\Pi_\infty)\cong \Hom_{G_p}\big({\rm ind}_{K_p}^{G_p}L(\lambda),\Pi_\infty\big)$$
carries an action of $\mathcal{H}(\lambda)$. By a slight extension of \cite[Lem.4.16(1)]{CEGGPS} (see the proof of \cite[Prop.3.15]{BHS}), the action of $R_{\rhobar_{\tilde v}}^\square$ on $\Pi_\infty(\lambda)$ via $R_{\rhobar_{\tilde v}}^\square\rightarrow R^{\rm loc}\hookrightarrow R_\infty$ factors through its quotient $R_{\rhobar_{\tilde v}}^{\square,{\bf k}_v\rm -cr}$ where, for $v\in S_p$, ${\bf k}_v:=(k_{v,\tau,i})_{1\leq i\leq n,\tau\in\Hom(F_{\tilde v},L)}$ with $k_{v,\tau,i}:=\lambda_{v,\tau,i}-(i-1)$ (note that $\omega(x_v)=\delta_{{\bf k}_v}$ and that $R_{\rhobar_{\tilde v}}^{\square,{\bf k}_v\rm -cr}$ is also a quotient of $R_{\rhobar_{\tilde v}}^{\overline\square}$).

These two actions of $\mathcal{H}(\lambda_v)$ and $R_{\rhobar_{\tilde v}}^{\square,{\bf k}_v\rm -cr}$ on the $L$-vector space $\Pi_\infty(\lambda)$ are related. By \cite[Th.4.1]{CEGGPS} and a slight extension of \cite[Lem.4.16(2)]{CEGGPS} (see the proof of \cite[Prop.3.15]{BHS}), there is a unique $L$-algebra homomorphism $\eta_v:\mathcal{H}(\lambda_v)\rightarrow R_{\rhobar_{\tilde v}}^{\square,{\bf k}_v\rm -cr}[1/p]$ which interpolates the local Langlands correspondence (in a sense given in \cite[Th.4.1]{CEGGPS}) and such that the above action of $\mathcal{H}(\lambda_v)$ on $\Pi_\infty(\lambda)$ agrees with the action induced by that of $R_{\rhobar_{\tilde v}}^{\square,{\bf k}_v\rm -cr}[1/p]$ composed with the morphism $\eta_v$.

In order to show that ${\rm LA}(\delta)$ admits a nonzero $G_p$-equivariant morphism to $\Pi_{\infty}[\mathfrak{p}_y]\otimes_{k(\mathfrak{p}_y)}k(x)$, we claim it is enough to show that $\Pi_\infty(\lambda)[\mathfrak{p}_y]\cong \Hom_{G_p}({\rm ind}_{K_p}^{G_p}L(\lambda),\Pi_\infty[\mathfrak{p}_y])$ is nonzero. Indeed, by what we just saw, any nonzero $G_p$-equivariant morphism ${\rm ind}_{K_p}^{G_p}L(\lambda)\rightarrow \Pi_\infty[\mathfrak{p}_y]$ induces a nonzero $G_p$-equivariant morphism:
$${\rm ind}_{K_p}^{G_p}L(\lambda)\otimes_Lk(x)\longrightarrow \Pi_\infty[\mathfrak{p}_y]\otimes_{k(\mathfrak{p}_y)}k(x)$$
which factors through ${\rm ind}_{K_p}^{G_p}L(\lambda)\otimes_{\mathcal{H}(\lambda)}\theta_{\mathfrak{p}_y}$
where $\theta_{\mathfrak{p}_y}$ is the character:
$$\xymatrix{
\theta_{\mathfrak{p}_y}:\mathcal{H}(\lambda)\ar[r]^<<<<{\otimes_{v\in S_p}\eta_v}&\widehat\bigotimes_{v\in S_p} R_{\rhobar_{\tilde v}}^{\square,{\bf k}_v{\rm -cr}}[1/p]\ar[r] & k(\mathfrak{p}_y)\subseteq k(x),
}$$
the last morphism being the canonical projection to the residue field $k(\mathfrak{p}_y)$ at $\mathfrak{p}_y$ (the map $R_{\rhobar_p}\hookrightarrow R_\infty\twoheadrightarrow R_\infty/\mathfrak{p}_y$ factoring through $\widehat\bigotimes_{v\in S_p} R_{\rhobar_{\tilde v}}^{\square,{\bf k}_v{\rm -cr}}$ by the assumption on $\rho$). But by the compatibility with the local Langlands correspondence in \cite[Th.4.1]{CEGGPS} together with the assumption $\varphi_{\tilde v,i}/\varphi_{\tilde v,j}\neq q_v$ for $1\leq i,j\leq n$ and $v\in S_p$, we have ${\rm ind}_{K_p}^{G_p}L(\lambda)\otimes_{\mathcal{H}(\lambda)}\theta_{\mathfrak{p}_y}\cong {\rm LA}(\delta)\otimes_{k(\delta)}k(x)$.

By the same proof as that of \cite[Lem.4.17(2)]{CEGGPS}, the $R_\infty \otimes_{R_{\rhobar_p}}\widehat\bigotimes_{v\in S_p} R_{\rhobar_{\tilde v}}^{\square,{\bf k}_v{\rm -cr}}$\!-module $\Pi_{\infty}(\lambda)'$ is supported on a union of irreducible components of:
$$\Xfrak_{\rhobar^p}\times \prod_{v\in S_p} \Xfrak_{\rhobar_{\tilde v}}^{\square,{\bf k}_v\rm -cr}\times\Ubb^g$$
and we have to prove that $y$ is a point on one of these irreducible components. Since $y\in \Xfrak^p\times \prod_{v\in S_p} Z_{\rm cris}(\rho_{\tilde v})\times\Ubb^g$ where $Z_{\rm cris}(\rho_{\tilde v})$ is the unique irreducible component of $\Xfrak_{\rhobar_{\tilde v}}^{\square,{\bf k}_v\rm -cr}$ containing $\rho_{\tilde v}:=\rho\vert_{\Gcal_{F_{\tilde v}}}$ (recall $\Xfrak_{\rhobar_{\tilde v}}^{\square,{\bf k}_v\rm -cr}$ is smooth over $L$ by \cite{Kisindef}), it is enough to prove that $\Xfrak^p\times \prod_{v\in S_p} Z_{\rm cris}(\rho_{\tilde v})\times\Ubb^g$ is one of the irreducible components in the support of $\Pi_{\infty}(\lambda)'$, or equivalently that $\Xfrak^p\times \prod_{v\in S_p} Z_{\rm cris}(\rho_{\tilde v})\times\Ubb^g$ contains at least one point which is in the support of $\Pi_{\infty}(\lambda)'$.

For each $v\in S_p$ let $x'_v=(r'_v,\delta'_v)$ be any point in $\iota_{{\bf k}_v}(\widetilde Z_{\rm cris}(x_v))\cap V_v\subseteq V_v\subseteq X_{\rm tri}^\square(\rhobar_{\tilde v})$ where $\widetilde Z_{\rm cris}(x_v)$ is as in (i) of Corollary \ref{defnZtri(x)} (so in particular $x'_v$ is crystalline strictly dominant of Hodge-Tate weights ${\bf k}_v$ and $r'_v$ lies on $Z_{\rm cris}(\rho_{\tilde v})$ by (i) of Remark \ref{down}). Then we have $x'_v\in Z_{{\rm tri},V_v}(x_v)$ for $v\in S_p$ by (ii) of Corollary \ref{defnZtri(x)}. From the assumption:
$$\prod_{v\in S_p}Z_{{\rm tri},V_v}(x_v)\subset X_{\rm tri}^{\Xfrak^p\rm -aut}(\rhobar_p)$$
it then follows that there exists:
$$x'=(y',\epsilon')\in \Xfrak^p\times\prod_{v\in S_p}Z_{{\rm tri},V_v}(x_v)\times \Ubb^g\subseteq \Xfrak^p\times X_{\rm tri}^{\Xfrak^p \rm-aut}(\rhobar_p)\times \Ubb^g\buildrel (\ref{union})\over \subseteq X_p(\rhobar)\subset \Xfrak_\infty\times \widehat T_{p,L}$$ 
(with $y'\in \Xfrak_\infty$, $\epsilon'\in \widehat T_{p,L}$) mapping to $x'_v$ via $\Xfrak^p\times X_{\rm tri}^{\Xfrak^p \rm-aut}(\rhobar_p)\times \Ubb^g\twoheadrightarrow X_{\rm tri}^{\Xfrak^p \rm-aut}(\rhobar_p)\hookrightarrow X_{\rm tri}^\square(\rhobar_p)\twoheadrightarrow X_{\rm tri}^\square(\rhobar_{\tilde v})$ (so $\epsilon'_v=\imath_v^{-1}(\delta'_v)$) and where $y'$ still belongs to $\Xfrak^p\times \prod_{v\in S_p} Z_{\rm cris}(\rho_{\tilde v})\times\Ubb^g$. It is thus enough to prove that $y'$ is in the support of $\Pi_{\infty}(\lambda)'$, i.e. that $\Pi_\infty(\lambda)[\mathfrak{p}_{y'}]\cong \Hom_{K_p}(L(\lambda),\Pi_\infty[\mathfrak{p}_{y'}])$ is nonzero.

We conclude by a similar argument as in the proof of \cite[Prop.3.27]{BHS}. By (the proof of) \cite[Lem.4.4]{Chenevier} and the same argument as at the end of the proof of Lemma \ref{genericopensinXcris} (using the smoothness, hence flatness, of $\widetilde U_{\rhobar_{\tilde v}}^{\square,{\bf k}_v\rm -cr}\rightarrow \Xfrak_{\rhobar_{\tilde v}}^{\square,{\bf k}_v\rm -cr}$), we may choose $x_v'\in \iota_{{\bf k}_v}(\widetilde Z_{\rm cris}(x_v))\cap V_v$ such that the crystalline Galois representation $r'_v$ is generic in the sense of \cite[Def.2.8]{BHS}. Then we claim that the nonzero $G_p$-equivariant morphism $\Fcal_{\overline B_p}^{G_p}(\epsilon')\rightarrow \Pi_\infty^{R_\infty-\rm an}[\mathfrak{p}_{y'}]\otimes_{k(\mathfrak{p}_{y'})}k(x')$ corresponding by \cite[Th.4.3]{BreuilAnalytiqueII} to the nonzero $T_p$-equivariant morphism $\epsilon'\rightarrow J_{B_p}(\Pi_\infty^{R_\infty-\rm an}[\mathfrak{p}_{y'}]\otimes_{k(\mathfrak{p}_{y'})}k(x'))$ given by the point $x'$ factors through its locally $\Q_p$-algebraic quotient ${\rm LA}(\epsilon')$ (which provides a nonzero $K_p$-equivariant morphism $L(\lambda)\rightarrow \Pi_\infty[\mathfrak{p}_{y'}]$). Indeed, if it doesn't, then the computation of the Jordan-H\"older factors of $\Fcal_{\overline B_p}^{G_p}(\epsilon')$ (\cite[Cor.4.6]{BreuilAnalytiqueII}) together with \cite[Cor.3.4]{BreuilAnalytiqueI} show that there exits a point $x''=(y',\epsilon'')\in X_{p}(\rhobar)$ such that $\epsilon''$ is locally algebraic of {\it nondominant} weight. In particular there is some $v\in S_p$ such that the image of $x''$ in $X_{\rm tri}^\square(\rhobar_{\tilde v})$ is of the form $(r'_v, \imath_v^{-1}(\epsilon''_v))$ with $\imath_v^{-1}(\epsilon''_v)$ locally algebraic {\it not} strictly dominant. This contradicts \cite[Lem.2.11]{BHS}.
\end{proof}

Let $x=(\rho,\delta)\in Y(U^p,\rhobar)$ be a crystalline strictly dominant point such that for all $v\in S_p$ the eigenvalues of the geometric Frobenius on ${\rm WD}(\rho|_{\Gcal_{F_{\tilde v}}})$ are pairwise distinct. Recall that we have associated in \S\ref{weyl} a Weyl group element $w_{x_v}$ to the image $x_v$ of $x$ in $X_{\rm tri}^\square(\rhobar_{\tilde v})$ via (\ref{eigenvartotrianguline}). We write:
\begin{equation}\label{wx}
w_x:=(w_{x_v})_{v\in S_p}\in \prod_{v\in S_p}\Big(\prod_{F_{\tilde v}\hookrightarrow L}\Scal_n\Big)
\end{equation}
for the corresponding element of the Weyl group of $(\Res_{F^+/\Q}G)_L\buildrel\sim\over\rightarrow \prod_{v\in S_p} (\Res_{F_{\tilde v}/\Q_p}\GL_{n,F_{\tilde v}})_L$. We then obtain the following corollary, which is our main result.

\begin{coro}\label{mainclassic}
Let $x=(\rho,\delta)\in Y(U^p,\rhobar)$ be a crystalline strictly dominant very regular point. Assume that the Weyl group element $w_{x}$ in (\ref{wx}) is a product of pairwise distinct simple reflections. Then $x$ is classical. Moreover all eigenvectors associated to $x$ are classical, that is we have (see the proof of Proposition \ref{compaclas} for $\widehat S(U^p,L)_{\mathfrak{m}^S}^{\lambda-{\rm la}}$):
$$\mathrm{Hom}_{T_p}\big(\delta,J_{B_p}(S(U^p,L)^{\lambda-{\rm la}}_{\mathfrak{m}^S}[\mathfrak{p}_\rho]\otimes_{k(\mathfrak{p}_\rho)}k(x))\big)\buildrel\sim\over\longrightarrow\mathrm{Hom}_{T_p}\big(\delta,J_{B_p}(\widehat{S}(U^p,L)^{\rm an}_{\mathfrak{m}^S}[\mathfrak{p}_\rho]\otimes_{k(\mathfrak{p}_\rho)}k(x))\big).$$
\end{coro}
\begin{proof}
Keep the notation of Theorem \ref{classicalitycrit}. By Proposition \ref{acconZtri} below, for each $v\in S_p$ there is a sufficiently small open neighbourhood $V_v$ of $x_v$ in $X_{\rm tri}^\square(\rhobar_{\tilde v})$ such that the irreducible component $Z_{{\rm tri},V_v}(x_v)$ of $V_v$ in (ii) of Corollary \ref{defnZtri(x)} is defined and satisfies the accumulation property at $x_v$ (Definition \ref{accu}).

Seeing $x$ in $X_p(\rhobar)$ via the closed embedding $Y(U^p,\rhobar)\hookrightarrow X_p(\rhobar)$, by (\ref{union}) there exist irreducible components $\Xfrak^p$ of $\Xfrak_{\rhobar^p}$ and $Z=\prod_{v\in S_p}Z_v$ of $X_{\rm tri}^\square(\rhobar_p)=\prod_{v\in S_p}X_{\rm tri}^\square(\rhobar_{\tilde v})$ such that:
$$x\in \Xfrak^p\times Z\times \Ubb^g\subseteq \Xfrak^p\times X_{\rm tri}^{\Xfrak^p \rm-aut}(\rhobar_p)\times\Ubb^g\buildrel (\ref{union}) \over\subseteq X_p(\rhobar).$$
Then it follows from Proposition \ref{acconZaut} and Remark \ref{obvious} below that $Z_v$ satisfies the accumulation property at $x_v$ for all $v\in S_p$. Let $Y_v\subseteq Z_v\cap V_v$ be a nonempty union of irreducible components of $V_v$, then $X_v:=Y_v\cup Z_{{\rm tri},V_v}(x_v)$ satisfies the accumulation property at $x_v$ since both $Y_v$ and $Z_{{\rm tri},V_v}(x_v)$ do. But $X_v$ is smooth at $x_v$ by the assumption on $w_{x_v}$ and Corollary \ref{Xtrismooth} applied with $(X,x)=(X_v,x_v)$, hence there can only be one irreducible component of $X_v$ passing through $x_v$. We deduce in particular $Z_{{\rm tri},V_v}(x_v)\subseteq Y_v\subseteq Z_v$, hence $\prod_{v\in S_p}Z_{{\rm tri},V_v}(x_v)\subseteq \prod_{v\in S_p}Z_v\subseteq X_{\rm tri}^{\Xfrak^p \rm-aut}(\rhobar_p)$ and $x$ is classical by Theorem \ref{classicalitycrit}. We also deduce that the only possible $Z=\prod_{v\in S_p}Z_v$ passing through $(x_v)_{v\in S_p}$ is smooth at $(x_v)_{v\in S_p}$, hence that $X_{\rm tri}^{\Xfrak^p \rm-aut}(\rhobar_p)$ is smooth at $(x_v)_{v\in S_p}$.

Let us now prove the last statement. From the injection:
$$\mathrm{Hom}_{T_p}\big(\delta,J_{B_p}(S(U^p,L)^{\lambda-{\rm la}}_{\mathfrak{m}^S}[\mathfrak{p}_\rho]\otimes_{k(\mathfrak{p}_\rho)}k(x))\big)\hookrightarrow\mathrm{Hom}_{T_p}\big(\delta,J_{B_p}(\widehat{S}(U^p,L)^{\rm an}_{\mathfrak{m}^S}[\mathfrak{p}_\rho]\otimes_{k(\mathfrak{p}_\rho)}k(x))\big)$$ 
it is enough to prove that these two $k(x)$-vector spaces have the same (finite) dimension. Recall from \cite[\S3.2]{BHS} that for any $x'=(y',\delta')\in X_p(\rhobar)$ we have an isomorphism of $k(x')$-vector spaces:
\begin{equation}\label{iso1}
\mathrm{Hom}_{T_p}\big(\delta',J_{B_p}(\Pi_\infty^{R_\infty-{\rm an}}[\mathfrak{p}_{y'}]\otimes_{k(\mathfrak{p}_{y'})}k(x'))\big)\cong \mathcal{M}_\infty\otimes_{\mathcal{O}_{X_p(\overline{\rho})}} k(x').
\end{equation}
If moreover $x'=(\rho',\delta')\in Y(U^p,\rhobar)\hookrightarrow X_p(\rhobar)$ we have $\widehat S(U^p,L)_{\mathfrak{m}^S}[\mathfrak{p}_{\rho'}]=\Pi_{\infty}[\mathfrak{p}_{y'}]$, hence an isomorphism of $k(x')$-vector spaces:
\begin{equation}\label{iso2}
\mathrm{Hom}_{T_p}\big(\delta',J_{B_p}(\widehat{S}(U^p,L)^{\rm an}_{\mathfrak{m}^S}[\mathfrak{p}_{\rho'}]\otimes_{k(\mathfrak{p}_{\rho'})}k(x'))\big)\simeq\mathrm{Hom}_{T_p}\big(\delta',J_{B_p}(\Pi_\infty^{R_\infty-{\rm an}}[\mathfrak{p}_{y'}]\otimes_{k(\mathfrak{p}_{y'})}k(x'))\big).
\end{equation}
We first claim that $x$ is a smooth point of $X_p(\overline{\rho})$. Indeed, by what we proved above, it is enough to show that its component $y^p=(y_v)_{v\in S\backslash S_p}$ in $\mathfrak{X}_{\overline{\rho}^p}$ is a smooth point. As $x$ is classical, by Proposition \ref{compaclas} (in particular the end of the proof) it corresponds to an automorphic representation $\pi$ of $G(\mathbb{A}_{F^+})$ with cuspidal strong base change $\Pi$ to $\mathrm{GL}_n(\mathbb{A}_F)$. It then follows from \cite[Th.1.2]{Caraiani} that $\Pi$ is tempered, in particular generic, at all finite places of $F$. Then \cite[Lem.1.3.2(1)]{BLGGT} implies that $y_v$ for $v\in S\backslash S_p$ is a smooth point of $(\Spf\, R_{\rhobar_{\tilde v}}^{\overline\square})^{\rig}$. As $\mathcal{M}_\infty$ is Cohen-Macaulay (Lemma \ref{CM}) and $x$ is smooth on $X_p(\overline{\rho})$, we conclude from \cite[Cor.17.3.5(i)]{EGAIV1} that $\mathcal{M}_\infty$ is actually locally free at $x$. Consequently there exists an open affinoid neighbourhood of $x$ in $X_p(\overline{\rho})$ on which the dimension of the fibers of $\mathcal{M}_\infty$ is constant. Intersecting this neighbourhood with $Y(U^p,\overline{\rho})$ and using (\ref{iso1}) and (\ref{iso2}), we obtain an open affinoid neighbourhood $V_x$ of $x$ in $Y(U^p,\overline{\rho})$ on which $\dim_{k(x')}\mathrm{Hom}_{T_p}(\delta',J_{B_p}(\widehat{S}(U^p,L)^{\rm an}_{\mathfrak{m}^S}[\mathfrak{p}_{\rho'}]\otimes_{k(\mathfrak{p}_{\rho'})}k(x')))$ is constant for $x'=(\rho',\delta')\in V_x$.

Now let $x'\in V_x$ be a very classical point in the sense of \cite[Def.3.16]{BHS} and write $\omega(x')=\delta_{\lambda'}$ with dominant $\lambda'\in \prod_{v\in S_p}(\Z^n)^{\Hom(F_{\tilde v},L)}$. It follows from {\it loc. cit.} and \cite[Th.4.3]{BreuilAnalytiqueII} that we have:
\begin{eqnarray*}\label{adjla}
{\scriptstyle \Hom_{T_p}\big(\delta',\ J_{B_p}(\widehat S(U^p,L)^{\rm an}_{\mathfrak{m}^S}[\mathfrak{p}_{\rho'}]\otimes_{k(\mathfrak{p}_{\rho'})}k(x'))\big)}&\cong&{\scriptstyle \Hom_{G_p}\big({\rm LA}(\delta'),\ \widehat S(U^p,L)_{\mathfrak{m}^S}^{\rm an}[\mathfrak{p}_{\rho'}]\otimes_{k(\mathfrak{p}_{\rho'})}k(x')\big)}\\
&\cong& {\scriptstyle \Hom_{G_p}\big({\rm LA}(\delta'),\ \widehat S(U^p,L)_{\mathfrak{m}^S}^{\lambda'-{\rm la}}[\mathfrak{p}_{\rho'}]\otimes_{k(\mathfrak{p}_{\rho'})}k(x')\big)}\\
&\cong &{\scriptstyle \Hom_{T_p}\big(\delta',\ J_{B_p}(\widehat S(U^p,L)^{\lambda'-{\rm la}}_{\mathfrak{m}^S}[\mathfrak{p}_{\rho'}]\otimes_{k(\mathfrak{p}_{\rho'})}k(x'))\big)}.
\end{eqnarray*}
From what is proved above, it is thus enough to find a very classical point $x'$ in $V_x$ such that:
\begin{multline}\label{dimequal}
\dim_{k(x')}\Hom_{T_p}\big(\delta',J_{B_p}(\widehat S(U^p,L)^{\lambda'-{\rm la}}_{\mathfrak{m}^S}[\mathfrak{p}_{\rho'}]\otimes_{k(\mathfrak{p}_{\rho'})}k(x'))\big)=\\
\dim_{k(x)}\Hom_{T_p}\big(\delta,J_{B_p}(\widehat S(U^p,L)^{\lambda-{\rm la}}_{\mathfrak{m}^S}[\mathfrak{p}_{\rho}]\otimes_{k(\mathfrak{p}_{\rho})}k(x))\big).
\end{multline}
Let $x''=(\rho'',\delta'')\in Y(U^p,\rhobar)$ be any classical crystalline strictly dominant point and let $\omega(x'')=\delta_{\lambda''}$. By Proposition \ref{compaclas} it corresponds to a unique automorphic representation $\pi''$ which moreover has multiplicity $1$, hence we have (with the notation of the proof of Proposition \ref{compaclas}):
$$J_{B_p}\big(S(U^p,L)_{\mathfrak{m}^S}^{\lambda''-{\rm la}}[\mathfrak{p_{\rho''}}]\otimes_{k(\mathfrak{p_{\rho''}})}\overline\Q_p\big)\simeq J_{B_p}\big(L(\lambda'')\otimes_L\bigotimes_{v\in S_p}\pi''_{v}\big)\otimes_{\Qbar}({\pi_f''}^{p})^{U^p}\otimes_{\Qbar,j_p}\overline\Q_p.$$
From the definition of $S$ together with \cite[Prop.4.3.6]{EmertonJacquetI} and property (iv) in Proposition \ref{compaclas}, it then easily follows that:
\begin{equation}\label{calcdim}
\dim_{k(x'')}\Hom_{T_p}\big(\delta'',J_{B_p}(\widehat S(U^p,L)^{\lambda''-{\rm la}}_{\mathfrak{m}^S}[\mathfrak{p}_{\rho''}]\otimes_{k(\mathfrak{p}_{\rho''})}k(x''))\big)=\dim_{\Qbar}\Big(\bigotimes_{v\in S\backslash S_p} {\pi_v''}^{U_v}\Big).
\end{equation}
Let $Z$ be the union of $x$ and of the very classical points in $V_x$, by \cite[Thm.3.18]{BHS} this set $Z$ accumulates at $x$. By \cite[Th.1.2]{Caraiani}, we can apply \cite[Lem.4.5(ii)]{Chenevierfern} to the intersection of $Z$ with one irreducible component of $V_x$, and obtain that, for $v\nmid p$, the value $\dim_{\Qbar}{\pi''_{v}}^{U_v}$ is constant on this intersection. In particular $\dim_{\Qbar}(\bigotimes_{v\in S\backslash S_p} {\pi_v''}^{U_v})$ is also constant on this intersection, which finishes the proof by (\ref{calcdim}) and (\ref{dimequal}).
\end{proof}

\begin{rema}\label{uniqueirredcompo}
{\rm (i) Keeping the notation of Theorem \ref{classicalitycrit}, if there is a {\it unique} irreducible component $Z$ of $X_{\rm tri}^\square(\rhobar_p)$ passing through the image of $x$ in $X_{\rm tri}^\square(\rhobar_p)$, or equivalently if for each $v\in S_p$ there is a unique irreducible component of $X_{\rm tri}^\square(\rhobar_{\tilde v})$ passing through $x_v$, then $x$ is classical. Indeed, in that case there is an irreducible component $\Xfrak^p$ of $\Xfrak_{\rhobar^p}$ such that $x\in \Xfrak^p\times Z\times \Ubb^g=\Xfrak^p\times X_{\rm tri}^{\Xfrak^p \rm-aut}(\rhobar_p)\times\Ubb^g$. In particular, for a sufficiently small open neighbourhood $V_v$ of $x_v$ in $X_{\rm tri}^\square(\rhobar_{\tilde v})$, we have $\prod_{v\in S_p} V_v \subseteq Z=X_{\rm tri}^{\Xfrak^p\rm-aut}(\rhobar_p)$ and we see that the assumption in Theorem \ref{classicalitycrit} is {\it a fortiori} satisfied.\\
(ii) Let us recall the various global hypothesis underlying the statements of Theorem \ref{classicalitycrit} and Corollary \ref{mainclassic}: $p>2$, $G$ is quasi-split at all finite places of $F^+$, $F/F^+$ is unramified, $U_v$ is hyperspecial if $v$ is inert in $F$, $\rhobar(\mathcal{G}_{F(\zeta_p)})$ is adequate and $\zeta_p\notin \overline F^{\ker(\rhobar\otimes \rhobar')}$.}
\end{rema}

\subsection{Accumulation properties}

We prove some accumulation properties (as in Definition \ref{accu}) that are used in the proof of Corollary \ref{mainclassic} in order to apply Corollary \ref{Xtrismooth}.

We first go back to the purely local set up of \S\ref{localpart1}. We call a point $x=(r,\delta_1,\dots,\delta_n)\in X_{\rm tri}^\square(\overline r)$ \emph{saturated} if there exists a triangulation of the $(\varphi,\Gamma_K)$-module $D_{\rig}(r)$ with parameter $(\delta_1,\dots,\delta_n)$ (cf. \S\ref{begin}). Note that, if $x$ is crystalline strictly dominant with pairwise distinct Frobenius eigenvalues, then $x$ is saturated if and only if $x$ is noncritical (cf. \S\S\ref{variant}, \ref{weyl}). Recall also from \S\ref{begin} that if $x$ is saturated and if $(\delta_1,\dots,\delta_n)\in \Tcal_{\rm reg}^n$ then $x\in U_{\rm tri}^\square(\overline r)$.

\begin{lemm}\label{numnoncrit}
Let $x=(r,\delta_1,...,\delta_n)\in X_{\rm tri}^\square(\overline r)$ with $\omega(x)=\delta_{\bf k}$ for some ${\bf k}\!=\!(k_{\tau,i})_{1\leq i\leq n,\tau:\, K\hookrightarrow L}\!\in(\mathbb{Z}^n)^{\Hom(K,L)}$. Assume that:
\begin{equation}\label{eqnnumnoncrit}
k_{\tau,i}-k_{\tau,i+1}> [K:K_0] \val\big(\delta_1(\varpi_K)\cdots \delta_i(\varpi_K)\big)
\end{equation}
for $i\in \{1,\dots,n-1\}$, $\tau\in \Hom(K,L)$. Then $x$ is saturated and $r$ is semi-stable. If moreover $(\delta_1,\dots,\delta_n)\in \Tcal^n_{\rm reg}$, then $r$ is crystalline strictly dominant noncritical. 
\end{lemm}
\begin{proof}
By \cite[Th.6.3.13]{KPX} and \cite[Prop.2.9]{BHS} the representation $r$ is trianguline with parameter $(\delta'_1,\dots, \delta'_n)$ where $\delta'_i=\delta_i z^{{\bf k}_{w^{-1}(i)}-{\bf k}_i}$ for some $w=(w_\tau)_{\tau:\, K\hookrightarrow L}\in W=\prod_{\tau:K\hookrightarrow L} \Scal_n$.
As $D_{\rig}(r)$, and hence $\wedge_{\mathcal{R}_{k(x),K}}^i D_{\rig}(r)$, are $\varphi$-modules over $\mathcal{R}_{k(x),K}$ which are pure of slope zero (being \'etale $(\varphi,\Gamma_K)$-modules), it follows that for all $i$:
$$1\leq \big|\delta'_1(\varpi_K)\cdots\delta'_i(\varpi_K)\big|_K.$$
Since $\delta'_1(\varpi_K)\cdots\delta'_i(\varpi_K)=\delta_1(\varpi_K)\cdots\delta_i(\varpi_K)\cdot \prod_{j=1}^i\prod_\tau (\tau(\varpi_K)^{k_{\tau,w_{\tau}^{-1}(j)}-k_{\tau,j}})$ we obtain:
\begin{equation}\label{slope}
\val\big(\delta_1(\varpi_K)\cdots \delta_i(\varpi_K)\big) \geq \tfrac{1}{[K:K_0]}\sum_{j=1}^i\sum_{\tau}(k_{\tau,j}-k_{\tau,w_\tau^{-1}(j)}).
\end{equation}
We now prove by induction on $i$ that $w_\tau^{-1}(i)=i$ for all $\tau$. The inequality (\ref{slope}) for $i=1$ gives 
$\val(\delta_1(\varpi_K))\geq \tfrac{1}{[K:K_0]}\sum_\tau (k_{\tau,1}-k_{\tau,w_\tau^{-1}(1)})$. But assumption $(\ref{eqnnumnoncrit})$ with $i=1$ implies $\val(\delta_1(\varpi_K))< \tfrac{1}{[K:K_0]}\sum_\tau (k_{\tau,1}-k_{\tau,j})$ for $j\in \{2,\dots,n\}$ which forces $w_\tau^{-1}(1)=1$ for all $\tau$. 
Assume by induction that $w_\tau^{-1}(j)=j$ for all $j\leq i-1$ and all $\tau$. Then (\ref{slope}) gives:
$$\val(\delta_1(\varpi_K)\cdots\delta_i(\varpi_K))\geq \tfrac{1}{[K:K_0]}\sum_\tau (k_{\tau,i}-k_{\tau,w_\tau^{-1}(i)})$$
and again $(\ref{eqnnumnoncrit})$ implies $\val(\delta_1(\varpi_K)\cdots\delta_i(\varpi_K))< \tfrac{1}{[K:K_0]}\sum_\tau (k_{i,\tau}-k_{j,\tau})$ for $j\in \{i,\dots,n\}$ which forces $w_\tau^{-1}(i)=i$ for all $\tau$. We thus have $(\delta_1,\dots, \delta_n)=(\delta'_1,\dots, \delta'_n)$ which implies that the point $x=(r,\delta_1,\dots, \delta_n)$ is saturated. Since $\delta$ is strictly dominant, we obtain that $r$ is semi-stable by the argument in the proof of \cite[Th.3.14]{Chenevier} (see also the proof of \cite[Cor.2.7(i)]{HellmSchrDensity}). By a slight generalisation of the proof of Lemma \ref{paramofcrystpt} (that we leave to the reader), we have $\delta_i=z^{{\bf k}_i}{\rm unr}(\varphi_i)$ where the $\varphi_i$ are the eigenvalues of the linearized Frobenius $\varphi^{[K_0:\Q_p]}$ on the $K_0\otimes_{\Q_p}k(x)$-module $D_{\rm st}(r):=(B_{\rm st}\otimes_{\Q_p}r)^{\mathcal{G}_K}$. If in addition $(\delta_1,\dots,\delta_n)\in \Tcal^n_{\rm reg}$, then it follows from Remark \ref{varphi} that $\varphi_i\varphi_j^{-1}\ne p^{-[K_0:\Q_p]}$ for $1\leq i\leq j\leq n$ and the argument of \cite[Th.3.14]{Chenevier}, \cite[Cor.2.7(i)]{HellmSchrDensity} then shows that the monodromy operator $N$ on $D_{\rm st}(r)$ must be zero, i.e. that $r$ is crystalline. This finishes the proof.
\end{proof}

\begin{prop}\label{acconZtri}
Let $x=(r,\delta)\in X_{\rm tri}^\square(\overline r)$ be a crystalline strictly dominant point such that the eigenvalues of the geometric Frobenius on ${\rm WD}(r)$ are pairwise distinct. Then there exists a sufficiently small open neighbourhood $U$ of $x$ in $X_{\rm tri}^\square(\overline r)$ such that the irreducible component $Z_{{\rm tri},U}(x)$ of $U$ in (ii) of Corollary \ref{defnZtri(x)} is defined and satisfies the accumulation property at $x$. 
\end{prop}
\begin{proof}
Recall that we have to prove that, for any positive real number $C$, the set of points $x'=(r',\delta')\in Z_{{\rm tri},U}(x)$ such that $r'$ is crystalline with pairwise distinct geometric Frobenius eigenvalues on ${\rm WD}(r')$ and $x'$ is noncritical with $\omega(x')=\delta_{\bf k'}$ strictly dominant satisfying:
\begin{equation}\label{accumulation}
k'_{\tau,i}-k'_{\tau,i+1}>C
\end{equation}
for all $i=1,\dots,n-1$, $\tau:K\hookrightarrow L$ accumulates at $x$. 

Let $U$ be an open subset of $x$ in $X_{\rm tri}^\square(\overline r)$ as in (iii) of Corollary \ref{defnZtri(x)}, i.e. such that for any open $U'\subseteq U$ containing $x$ we have $Z_{{\rm tri},U}(x)\cap U'=Z_{{\rm tri},U'}(x)$. Let $\widetilde Z_{\rm cris}(x)$ as in (i) of Corollary \ref{defnZtri(x)}, by Lemma $\ref{genericopensinXcris}$, the space $V:=\iota_{\bf k}(\widetilde Z_{\rm cris}(x)\cap \widetilde V_{\overline r}^{\square,{\bf k}{\rm -cr}})$ is Zariski-open and Zariski-dense in $\iota_{\bf k}(\widetilde Z_{\rm cris}(x))$, hence accumulates in $\iota_{\bf k}(\widetilde Z_{\rm cris}(x))$ at any point of $\iota_{\bf k}(\widetilde Z_{\rm cris}(x))$, in particular at $x$. We claim that it is enough to prove that the points $x'\in U$ as above accumulate in $U$ at every point of $V\cap U$. Indeed, if $U'\subseteq U$ is an open neighbourhood containing $x$, then $U'$ also contains a point $v\in V$. By the accumulation statement at $v\in V\cap U$, the Zariski closure in $U'$ of the points $x'$ contains a small neighbourhood around $v$, hence contains an irreducible component of $U'$ containing $v$. But since $v$ is a smooth point of $U'$ (over $L$) as $v\in U_{\rm tri}^\square(\overline r)$ by the last statement of Lemma \ref{genericopensinXcris}, there is only one such irreducible component, and since $v\in \iota_{\bf k}(\widetilde Z_{\rm cris}(x))\cap U'\subseteq Z_{{\rm tri},U'}(x)$, we see that this irreducible component must be $Z_{{\rm tri},U'}(x)$. Thus the Zariski closure in $U'$ of the points $x'$ always contains $Z_{{\rm tri},U'}(x)$. This easily implies the proposition since $Z_{{\rm tri},U'}(x)=Z_{{\rm tri},U}(x)\cap U'$. 

Since $U_{\rm tri}^\square(\overline r)$ is open in $X_{\rm tri}^\square(\overline r)$, it is enough to prove that the crystalline points $x'$ in $U\cap U_{\rm tri}^\square(\overline r)$ satisfying the conditions in the first paragraph of this proof accumulate at any crystalline strictly dominant point $x$ of $U\cap U_{\rm tri}^\square(\overline r)$. The condition on their Frobenius eigenvalues is then in fact automatic by Remark \ref{varphi}. Shrinking $U$ further if necessary, we can take $U$ to be contained in some quasi-compact open neighbourhood of $x$ in $X_{\rm tri}^\square(\overline r)$, and thus we may assume that for $i\in \{1,\dots,n\}$ the functions $y=(r_y,(\delta_{y,1},\dots,\delta_{y,n}))\mapsto \delta_{y,i}(\varpi_K)$ are uniformly bounded on $U$. Hence by Lemma $\ref{numnoncrit}$ we may assume that $C$ is sufficiently large so that the points $x'\in U\cap U_{\rm tri}^\square(\overline r)$ with $\omega(x)=\delta_{\bf k'}$ algebraic satisfying $(\ref{accumulation})$ are in fact also automatically crystalline noncritical. Changing notation, we see that it is finally enough to prove that the points $x'\in U_{\rm tri}^\square(\overline r)$ satisfying (\ref{accumulation}) for $C$ big enough accumulate at any crystalline strictly dominant point $x$ of $U_{\rm tri}^\square(\overline r)$.

We now consider the rigid analytic spaces $\Scal_n$, $\Scal^\square(\overline r)$ appearing in the proof of \cite[Th.2.6]{BHS} (to which we refer the reader for more details; do not confuse here $\Scal_n$ with the permutation group!). In {\it loc. cit.} there is a diagram of rigid spaces over $\Tcal_L^n$:
$$\begin{xy}\xymatrix{&\Scal^\square(\overline r)\ar[dl]_{\pi_{\overline r}}\ar[dr]^g&\\ U_{\rm tri}^\square(\overline r) & & \Scal_n}
\end{xy}$$
where $\pi_{\overline r}$ is a $\Gbb_m^n$-torsor and $g$ is a composition $\Scal^\square(\overline r)\hookrightarrow \Scal_n^{\square,{\rm adm}}\rightarrow \Scal_n^{{\rm adm}}\hookrightarrow \Scal_n$ where the first and last maps are open embeddings and the middle one is a $\GL_n$-torsor.

Let us choose a point $\tilde x\in \pi_{\overline r}^{-1}(x)$. As $\pi_{\overline r}$ is a $\Gbb_m^n$-torsor, it is enough to prove that the points in $\Scal^\square(\overline r)$ satisfying $(\ref{accumulation})$ accumulate at $\tilde x$. 
The same argument shows that it is enough to prove that the points of $\Scal_n$ satisfying $(\ref{accumulation})$ accumulate at $g(\tilde x)$. But the morphism $\Scal_n\rightarrow \Tcal^n_L$ is a composition of open embeddings and structure morphisms of geometric vector bundles (compare the proof of \cite[Th.2.4]{HellmSchrDensity}). It follows that $g(\tilde x)$ has a basis of neighbourhoods $(U_i)_{i\in I}$ in $\Scal_n$ such that $V_i:=\omega(U_i)$ is a basis of neighbourhoods of $\omega( x)$ in $\Wcal_L^n$ and such that the rigid space $U_i$ is isomorphic to a product $V_i\times\boldB$ of rigid spaces over $L$ where $\boldB$ is some closed polydisc (compare \cite[Cor.3.5]{Chenevier} and \cite[Lem.2.18]{HellmSchrDensity}). Write $\omega( x)=\delta_{\bf k}$, it is thus enough to prove that the algebraic weights $\delta_{\bf k'}\in \Wcal_L^n$ satisfying $(\ref{accumulation})$ accumulate at $\delta_{\bf k}$ in $\Wcal_L^n$, which is obvious. 
\end{proof}

We are now back to the global setting of \S\ref{firstclassical}. Similarly to Definition \ref{accu}, we say that a union $X$ of irreducible components of an open subset of $X_{\rm tri}^\square(\rhobar_p)=\prod_{v\in S_p}X_{\rm tri}^\square(\rhobar_{\tilde v})$ satisfies the accumulation property at a point $x\in X$ if, for any positive real number $C>0$, $X$ contains crystalline strictly dominant points $x'=(x'_v)_{v\in S_p}$ with pairwise distinct Frobenius eigenvalues, which are noncritical, such that $\omega(x'_v)=\delta_{{\bf k}'_v}$ with $k'_{v,\tau,i}-k'_{v,\tau,i+1}>C$ for $v\in S_p$, $i\in \{1,\dots,n-1\}$, $\tau\in \Hom(F_{\tilde v},L)$ and that accumulate at $x$ in $X$. 

\begin{prop}\label{acconZaut}
Let $\Xfrak^p\subset \Xfrak_{\rhobar^p}$ be an irreducible component and $x\in X_{\rm tri}^{\Xfrak^p\rm- aut}(\rhobar_p)$ be a crystalline strictly dominant point. Then $X_{\rm tri}^{\Xfrak^p\rm-aut}(\rhobar_p)$ satisfies the accumulation property at $x$.
\end{prop}
\begin{proof}
It is enough to show that, for $C$ large enough, the points of $\Xfrak^p\times X_{\rm tri}^{\Xfrak^p \rm-aut}(\rhobar_p)\times \Ubb^g$ such that their projection to $X_{\rm tri}^{\Xfrak^p \rm-aut}(\rhobar_p)$ is a point $x'=(x'_v)_{v\in S_p}$ of the same form as above accumulate at any point of $\Xfrak^p\times X_{\rm tri}^{\Xfrak^p \rm-aut}(\rhobar_p)\times \Ubb^g$ mapping to $x$ in $X_{\rm tri}^{\Xfrak^p \rm-aut}(\rhobar_p)$. Using (\ref{union})  this claim is contained in the proof of \cite[Th.3.18]{BHS} (which itself is a consequence of \cite[Prop.3.10]{BHS}).
\end{proof}

\begin{rema}\label{obvious}
{\rm It is obvious from the definition that if a union $X$ of irreducible components of $X_{\rm tri}^\square(\rhobar_p)=\prod_{v\in S_p}X_{\rm tri}^\square(\rhobar_{\tilde v})$ satisfies the accumulation property at some point $x\in X$, then for each $v\in S_p$ the image of $X$ in $X_{\rm tri}^\square(\rhobar_{\tilde v})$ (which is a union of irreducible components of $X_{\rm tri}^\square(\rhobar_{\tilde v})$) satisfies the accumulation property at the image of $x$ in $X_{\rm tri}^\square(\rhobar_{\tilde v})$.}
\end{rema}

\section{On the local geometry of the trianguline variety}\label{phiGammacohomology}

This section is entirely local and devoted to the proof of Theorem \ref{upperbound} above giving an upper bound on some local tangent spaces. We use the notation of \S\ref{localpart1}.

\subsection{Tangent spaces}\label{start}

We start with easy preliminary lemmas on some tangent spaces. 

If $x=(r,\delta)\in X_{\rm tri}^\square(\rbar)$, we denote by ${\rm Ext}^1_{\mathcal{G}_K}(r,r)$ the usual $k(x)$-vector space of $\mathcal{G}_K$-extensions $0\rightarrow r\rightarrow *\rightarrow r\rightarrow  0$.

\begin{lemm}\label{ext1}
Let $x=(r,\delta)\in X_{\rm tri}^\square(\rbar)$ be any point, then there is an exact sequence of $k(x)$-vector spaces $0\rightarrow K(r)\rightarrow T_{\mathfrak{X}_{\rbar}^\square,r}\rightarrow {\rm Ext}^1_{\mathcal{G}_K}(r,r)\rightarrow 0$ where $K(r)$ is a $k(x)$-subvector space of $T_{\mathfrak{X}_{\rbar}^\square,r}$ of dimension $\dim_{k(x)}\End_{k(x)}(r)-\dim_{k(x)}\End_{\mathcal{G}_K}(r)=n^2-\dim_{k(x)}\End_{\mathcal{G}_K}(r)$.
\end{lemm}
\begin{proof}
It easily follows from \cite[Lem.2.3.3 \& Prop.2.3.5]{KisinModularity} that there is a topological isomorphism $\widehat{\mathcal{O}}_{\mathfrak{X}^\square_{\rbar},r}\cong R^\square_{r}$ where the former is the completed local ring at $r$ to the rigid analytic variety $\mathfrak{X}^\square_{\rbar}$ and the latter is the framed local deformation ring of $r$ in equal characteristic $0$. In particular from (\ref{tangent}) we have $T_{\mathfrak{X}_{\rbar}^\square,r}\cong \Hom_{k(x)}\big(R^\square_{r},k(x)[\varepsilon]/(\varepsilon^2)\big)$. Then the result follows by the same argument as in \cite[\S2.3.4]{KisinModularity}, seeing an element of ${\rm Ext}^1_{\mathcal{G}_K}(r,r)$ as a deformation of $r$ with values in $k(x)[\varepsilon]/(\varepsilon^2)$.
\end{proof}

\begin{lemm}\label{dim}
Let $x=(r,\delta)\in X_{\rm tri}^\square(\rbar)$ be a point such that $H^2(\mathcal{G}_K,r\otimes r')=0$ ($r'$ being the dual of $r$), then $\dim_{k(x)}{\rm Ext}^1_{\mathcal{G}_K}(r,r)=\dim_{k(x)}\End_{\mathcal{G}_K}(r)+n^2[K:\Q_p]$.
\end{lemm}
\begin{proof}
This follows by the usual argument computing $\dim_{k(x)}H^1(\mathcal{G}_K,r\otimes r')$ from the Euler characteristic formula of Galois cohomology using $\dim_{k(x)}H^0(\mathcal{G}_K,r\otimes r')=\dim_{k(x)}\End_{\mathcal{G}_K}(r)$ and $\dim_{k(x)}H^2(\mathcal{G}_K,r\otimes r')=0$.
\end{proof}

\begin{rema}\label{cascr}
{\rm Lemma \ref{dim} in particular holds if $x$ is crystalline and the Frobenius eigenvalues $(\varphi_i)_{1\leq i\leq n}$ (see Lemma \ref{paramofcrystpt}) satisfy $\varphi_i\varphi_j^{-1}\ne q$ for $1\leq i,j\leq n$. In particular it holds if $x$ is crystalline strictly dominant very regular (cf. Definition \ref{veryreg}).}
\end{rema}

We now fix a point $x=(r,\delta)\in X_{\rm tri}^\square(\rbar)$ which is crystalline strictly dominant very regular and a union $X$ of irreducible components of an open subset of $X_{\rm tri}^\square(\rbar)$ such that $X$ satisfies the accumulation property at $x$ (Definition \ref{accu}). It obviously doesn't change the tangent space $T_{X,x}$ of $X$ at $x$ if we replace $X$ by the union of its irreducible components that contain $x$, hence we may (and do) assume that $x$ belongs to each irreducible component of $X$.

\begin{lemm}\label{inj}
There is an injection of $k(x)$-vector spaces $T_{X,x}\hookrightarrow T_{\mathfrak{X}_{\rbar}^\square,r}$.
\end{lemm}
\begin{proof}
The embedding $X\hookrightarrow X_{\rm tri}^\square(\rbar)\hookrightarrow \mathfrak{X}_{\rbar}^\square\times \mathcal{T}^n_L$ induces an injection on tangent spaces (with obvious notation):
$$T_{X,x}\hookrightarrow T_{\mathfrak{X}_{\rbar}^\square,r}\oplus T_{\mathcal{T}^n_L,\delta}.$$
We thus have to show that the composition with the projection $T_{\mathfrak{X}_{\rbar}^\square,r}\oplus T_{\mathcal{T}^n_L,\delta}\twoheadrightarrow T_{\mathfrak{X}_{\rbar}^\square,r}$ remains injective. Let $\vec{v}\in T_{X,x}$ which maps to $0\in T_{\mathfrak{X}_{\rbar}^\square,r}$, and thus {\it a fortiori} to $0$ in ${\rm Ext}^1_{\mathcal{G}_K}(r,r)$ via the surjection in Lemma \ref{ext1}. We have to show that the image of $\vec{v}$ in $T_{\mathcal{T}^n_L,\delta}$ is also $0$. We know that the image of $\vec{v}$ in $T_{\mathcal{W}^n_L,\omega(x)}$ is zero since the Hodge-Tate weights don't vary (that is, the $d_{\tau,i,\vec{v}}$ below are all zero, see the beginning of \S\ref{wedgesection}). To conclude that the image in $T_{\mathcal{T}^n_L,\delta}$ is also $0$, we can for instance use Bergdall's Theorem \ref{bergsplit} below (which uses the accumulation property of $X$ at $x$) together with an obvious induction on $i$.
\end{proof}

\begin{lemm}\label{boundtx}
Assume that the $k(x)$-vector space image of $T_{X,x}$ in ${\rm Ext}^1_{\mathcal{G}_K}(r,r)$ has dimension smaller or equal than:
$$\dim_{k(x)}{\rm Ext}^1_{\mathcal{G}_K}(r,r)-d_x-\big([K:\Q_p]\frac{n(n-1)}{2}-\lg(w_x)\big).$$
Then Theorem \ref{upperbound} is true.
\end{lemm}
\begin{proof}
From Lemma \ref{inj} and Lemma \ref{ext1} we obtain a short exact sequence:
\begin{eqnarray}\label{exactseq}
0\longrightarrow K(r)\cap T_{X,x}\longrightarrow T_{X,x}\longrightarrow {\rm Ext}^1_{\mathcal{G}_K}(r,r).
\end{eqnarray}
Hence the assumption implies:
$$\dim_{k(x)}T_{X,x}\leq \dim_{k(x)}K(r)+\lg(w_x)-d_x+\dim_{k(x)}{\rm Ext}^1_{\mathcal{G}_K}(r,r)-[K:\Q_p]\frac{n(n-1)}{2}.$$
But from Lemma \ref{ext1}, Lemma \ref{dim} and Remark \ref{cascr} we have:
\begin{multline*}
\dim_{k(x}K(r)+\lg(w_x)-d_x+\dim_{k(x)}{\rm Ext}^1_{\mathcal{G}_K}(r,r)-[K:\Q_p]\frac{n(n-1)}{2}=\\
\lg(w_x)-d_x+n^2+[K:\Q_p]\frac{n(n+1)}{2}
\end{multline*}
which gives Theorem \ref{upperbound}.
\end{proof}

We will see below that $d_x$ correspond to the ``weight conditions'' and $[K:\Q_p]\frac{n(n-1)}{2}-\lg(w_x)$ correspond to the ``splitting conditions''.

\subsection{Tangent spaces and local triangulations}\label{wedgesection}

We recall some of the results of \cite{Bergdall} that we use to prove a technical statement on the image of $T_{X,x}$ in ${\rm Ext}^1_{\mathcal{G}_K}(r,r)$ (Corollary \ref{inclv}).

We keep the notation of \S\ref{start}, in particular $x=(r,\delta)=(r,\delta_1,\dots,\delta_n)\in X_{\rm tri}^\square(\rbar)$ is a crystalline strictly dominant very regular point and $X$ is a union of irreducible components of an open subset of $X_{\rm tri}^\square(\rbar)$, each component in $X$ satisfying the accumulation property at $x$. Taking a look at \cite[\S\S5.1,6.1]{Bergdall}, it is easy to see from the properties of $X_{\rm tri}^\square(\rbar)$ and from Definition \ref{accu} (together with the discussion that follows) that one can apply all the results of \cite[\S7]{Bergdall} at $X$ and the point $x$ (called the ``center'' and denoted by $x_0$ in {\it loc. cit.}). We let $w_x=(w_{x,\tau})_{\tau:\, K\hookrightarrow L}\in \prod_{\tau: K\hookrightarrow L}\Scal_n$ be the Weyl group element associated to $x$ (\S\ref{weyl}).

Recall that $D_{\rig}(r)$ is the \'etale $(\varphi,\Gamma_K)$-module over $\mathcal{R}_{k(x),K}=k(x)\otimes_{\Q_p}\mathcal{R}_{K}$ associated to $r$. Note that ${\rm Ext}^1_{\mathcal{G}_K}(r,r)\cong {\rm Ext}^1_{(\varphi,\Gamma_K)}(D_{\rig}(r),D_{\rig}(r))$ where the right hand side denotes the extension in the category of $(\varphi,\Gamma_K)$-module over $\mathcal{R}_{k(x),K}$ (see \cite[Prop.5.2.6]{BelChe} for $K=\Q_p$, the proof for any $K$ is analogous). We write $\omega(x)=\delta_{\bf k}$ for ${\bf k}=(k_{\tau,i})_{1\leq i\leq n,\tau:\, K\hookrightarrow L}\in(\mathbb{Z}^n)^{\Hom(K,L)}$. Let $\vec{v}\in T_{X,x}$, seeing the image of $\vec{v}$ in ${\rm Ext}^1_{\mathcal{G}_K}(r,r)$ as a $k(x)[\varepsilon]/(\varepsilon^2)$-valued representation of $\mathcal{G}_K$, we can write its Sen weights as $(k_{\tau,i}+\varepsilon d_{\tau,i,\vec{v}})_{1\leq i\leq n,\tau:\, K\hookrightarrow L}$ for some $d_{\tau,i,\vec{v}}\in k(x)$. The tangent space $T_{{\mathcal W}_L^n,\omega(x)}$ to ${\mathcal W}_L^n$ at $\omega(x)$ is isomorphic to $k(x)^{[K:\Q_p]n}$ and the $k(x)$-linear map of tangent spaces $d\omega : T_{X,x}\longrightarrow T_{{\mathcal W}_L,\omega(x)}$ induced by the weight map $\omega\vert_X$ sends $\vec{v}$ to the tuple $(d_{\tau,i,\vec{v}})_{1\leq i\leq n,\tau:\, K\hookrightarrow L}$. The following theorem is a direct application of \cite[Th.7.1]{Bergdall}.

\begin{theo}[\cite{Bergdall}]\label{bergweight}
For any $\vec{v}\in T_{X,x}$, we have $d_{\tau,i,\vec{v}}=d_{\tau,w_{x,\tau}^{-1}(i),\vec{v}}$ for $1\leq i\leq n$ and $\tau:\, K\hookrightarrow L$.
\end{theo}

Let $\vec{v}\in T_{X,x}$, we can see $\vec{v}$ as a $k(x)[\varepsilon]/(\varepsilon^2)$-valued point of $X$, and the composition:
$$\Spm\,k(x)[\varepsilon]/(\varepsilon^2)\buildrel \vec{v}\over \longrightarrow X\hookrightarrow X_{\rm tri}^\square(\rbar)\longrightarrow \mathcal{T}^n_L$$
gives rise to continuous characters $\delta_{i,\vec{v}}:K^\times \rightarrow (k(x)[\varepsilon]/(\varepsilon^2))^\times$ for $1\leq i\leq n$. The following theorem again follows from an examination of the proof of \cite[Th.7.1]{Bergdall}.

\begin{theo}[\cite{Bergdall}]\label{bergsplit}
For any $\vec{v}\in T_{X,x}$ and $1\leq i\leq n$ we have an injection of $(\varphi,\Gamma_K)$-modules over $\mathcal{R}_{k(x)[\varepsilon]/(\varepsilon^2),K}=k(x)[\varepsilon]/(\varepsilon^2)\otimes_{\Q_p}\mathcal{R}_{K}$:
$$\mathcal{R}_{k(x)[\varepsilon]/(\varepsilon^2),K}\big(\delta_{1,\vec{v}}\delta_{2,\vec{v}}\cdots \delta_{i,\vec{v}}\big)\hookrightarrow D_{\rig}(\wedge_{k(x)[\varepsilon]/(\varepsilon^2)}^ir_{\vec{v}})\cong \wedge_{\mathcal{R}_{k(x)[\varepsilon]/(\varepsilon^2),K}}^iD_{\rig}(r_{\vec{v}})$$
where the left hand side is the rank one $(\varphi,\Gamma_K)$-module defined by the character $\delta_{1,\vec{v}}\delta_{2,\vec{v}}\cdots \delta_{i,\vec{v}}$ (\cite[Cons.6.2.4]{KPX}) and where $r_{\vec{v}}$ is the $k(x)[\varepsilon]/(\varepsilon^2)$-valued representation of $\mathcal{G}_K$ associated to $\vec{v}$.
\end{theo}

From \cite[Prop.2.4.1]{BelChe} (which readily extends to $K\ne \Q_p$) or arguing as in \S\ref{variant}, the $(\varphi,\Gamma_K)$-module $D_{\rig}(r)$ has a triangulation $\Fil_{\bullet}$ for $\bullet\in \{1,\cdots,n\}$, the graded pieces being:
\begin{eqnarray}\label{triang}
\mathcal{R}_{k(x),K}\big(z^{{\bf k}_{w_x^{-1}(1)}}{\rm unr}(\varphi_1)\big),\dots,\mathcal{R}_{k(x),K}\big(z^{{\bf k}_{w_x^{-1}(n)}}{\rm unr}(\varphi_n)\big)
\end{eqnarray}
where ${\bf k}_{w_x^{-1}(i)}:=(k_{\tau,w_{x,\tau}^{-1}(i)})_{\tau:\, K\hookrightarrow L}$ (see (\ref{charc}) for $z^{{\bf k}_{j}}$). Note that we have:
\begin{equation}\label{dexat}
\delta_{i}(z)=z^{{\bf k}_{i}-{\bf k}_{w_x^{-1}(i)}}(z^{{\bf k}_{w_x^{-1}(i)}}{\rm unr}(\varphi_i)).
\end{equation}
For $1\leq i\leq n$ we let $D_{\rig}(r)^{\leq i}:=\Fil_{i}\subseteq D_{\rig}(r)$, and we set $D_{\rig}(r)^{\leq 0}:=0$. We thus have for $1\leq i\leq n$:
$${\rm gr}_iD_{\rig}(r):=D_{\rig}(r)^{\leq i}/D_{\rig}(r)^{\leq i-1}=\mathcal{R}_{k(x),K}\big(z^{{\bf k}_{w_x^{-1}(i)}}{\rm unr}(\varphi_i)\big).$$
For $\tau:\, K\hookrightarrow L$ we fix a Lubin-Tate element $t_\tau\in \mathcal{R}_{L,K}$ as in \cite[Not.6.2.7]{KPX} (recall that the ideal $t_\tau\mathcal{R}_{L,K}$ is uniquely determined). If ${\bf k}:=(k_\tau)_{\tau:\, K\hookrightarrow L}\in\Z_{\geq 0}^{\Hom(K,L)}$, we let $t^{\bf k}:=\prod_{\tau:\, K\hookrightarrow L}t_\tau^{k_{\tau}}$. We set for $1\leq i\leq n$:
$$\Sigma_i({\bf k},w_x):=\sum_{j=1}^i({\bf k}_{j}-{\bf k}_{w_x^{-1}(j)})\in \Z_{\geq 0}^{\Hom(K,L)}$$
(where nonnegativity comes from $k_{\tau,i}\geq k_{\tau,i+1}$ for every $i,\tau$) and we can thus define $t^{\Sigma_i({\bf k},w_x)}\in \mathcal{R}_{L,K}$. In particular we deduce from (\ref{dexat}) (and the properties of the $t_\tau$):
\begin{eqnarray}\label{wedgedelta}
\mathcal{R}_{k(x),K}(\delta_{1}\cdots \delta_{i})\cong t^{\Sigma_i({\bf k},w_x)}\wedge^i_{\mathcal{R}_{k(x),K}}D_{\rig}(r)^{\leq i}\hookrightarrow \wedge^i_{\mathcal{R}_{k(x),K}}D_{\rig}(r).
\end{eqnarray}

We consider for $1\leq i\leq n$ the cartesian square (which defines $V_i$):
$$\xymatrix{{\rm Ext}^1_{(\varphi,\Gamma_K)}\big(D_{\rig}(r),D_{\rig}(r)\big)\ar[r]&{\rm Ext}^1_{(\varphi,\Gamma_K)}\big(t^{\Sigma_i({\bf k},w_x)}D_{\rig}(r)^{\leq i},D_{\rig}(r)\big)\\ 
V_i \ar[r]\ar@{^{(}->}[u]& 
{\rm Ext}^1_{(\varphi,\Gamma_K)}\big(t^{\Sigma_i({\bf k},w_x)}D_{\rig}(r)^{\leq i},D_{\rig}(r)^{\leq i}\big)\ar@{^{(}->}[u]}$$
where the first horizontal map is the restriction map and where the injection on the right easily follows from the very regularity assumption (Definition \ref{veryreg}). Equivalently we have:
\begin{multline}\label{viker}
V_i\cong \ker\!\Big({\rm Ext}^1_{(\varphi,\Gamma_K)}\big(D_{\rig}(r),D_{\rig}(r)\big)\longrightarrow
{\rm Ext}^1_{(\varphi,\Gamma_K)}\big(t^{\Sigma_i({\bf k},w_x)}D_{\rig}(r)^{\leq i},D_{\rig}(r)/D_{\rig}(r)^{\leq i}\big)\Big)
\end{multline}
where the map is defined by pushforward along $D_{\rig}(r)\twoheadrightarrow D_{\rig}(r)/D_{\rig}(r)^{\leq i}$ and pullback along $t^{\Sigma_{i}({\bf k},w_x)}D_{\rig}(r)^{\leq i}\hookrightarrow D_{\rig}(r)$.

\begin{coro}\label{inclv}
The image of any $\vec{v}\in T_{X,x}$ in ${\rm Ext}^1_{\mathcal{G}_K}(r,r)\cong {\rm Ext}^1_{(\varphi,\Gamma_K)}(D_{\rig}(r),D_{\rig}(r))$ is in $V_1\cap \cdots \cap V_{n-1}$ (where the intersection is within ${\rm Ext}^1_{(\varphi,\Gamma_K)}(D_{\rig}(r),D_{\rig}(r))$).
\end{coro}
\begin{proof}
Note that $V_1\cap V_2\cap \cdots \cap V_{n}=V_1\cap V_2\cap \cdots \cap V_{n-1}$. Let $\vec{v}\in T_{X,x}$, $r_{\vec{v}}$ the associated $k(x)[\varepsilon]/(\varepsilon^2)$-deformation and see $D_{\rig}(r_{\vec{v}})$ as an element of ${\rm Ext}^1_{(\varphi,\Gamma_K)}(D_{\rig}(r),D_{\rig}(r))$. We have to prove that the image of $D_{\rig}(r_{\vec{v}})$ \ in ${\rm Ext}^1_{(\varphi,\Gamma_K)}(t^{\Sigma_i({\bf k},w_x)}D_{\rig}(r)^{\leq i},D_{\rig}(r)/D_{\rig}(r)^{\leq i})$ is zero for any $1\leq i\leq n$ (see (\ref{viker})). The proof is by induction on $i\geq 1$. The case $i=1$ follows immediately from Theorem \ref{bergsplit} and (\ref{wedgedelta}) (together with Definition \ref{veryreg}). We prove that the statement for $i-1$ implies the statement for $i$. 

So, assume $i\geq 2$ and that the image of $D_{\rig}(r_{\vec{v}})$ in:
$${\rm Ext}^1_{(\varphi,\Gamma_K)}\big(t^{\Sigma_{i-1}({\bf k},w_x)}D_{\rig}(r)^{\leq i-1},D_{\rig}(r)/D_{\rig}(r)^{\leq i-1}\big)$$
is zero. Then by Corollary \ref{zero2} the image of $D_{\rig}(r_{\vec{v}})$ in:
$${\rm Ext}^1_{(\varphi,\Gamma_K)}\big(t^{\Sigma_{i}({\bf k},w_x)}D_{\rig}(r)^{\leq i-1},D_{\rig}(r)/D_{\rig}(r)^{\leq i}\big)$$
is also zero. From the exact sequence:
\begin{multline*}
0\rightarrow {\rm Ext}^1_{(\varphi,\Gamma_K)}\big(t^{\Sigma_{i}({\bf k},w_x)}{\rm gr}_iD_{\rig}(r),D_{\rig}(r)/D_{\rig}(r)^{\leq i}\big)\rightarrow \\
{\rm Ext}^1_{(\varphi,\Gamma_K)}\big(t^{\Sigma_{i}({\bf k},w_x)}D_{\rig}(r)^{\leq i},D_{\rig}(r)/D_{\rig}(r)^{\leq i}\big)\rightarrow \\
{\rm Ext}^1_{(\varphi,\Gamma_K)}\big(t^{\Sigma_{i}({\bf k},w_x)}D_{\rig}(r)^{\leq i-1},D_{\rig}(r)/D_{\rig}(r)^{\leq i}\big)
\end{multline*}
(where the injectivity on the left follows from Definition \ref{veryreg}), we see that the image of $D_{\rig}(r_{\vec{v}})$ in ${\rm Ext}^1_{(\varphi,\Gamma_K)}\big(t^{\Sigma_{i}({\bf k},w_x)}D_{\rig}(r)^{\leq i},D_{\rig}(r)/D_{\rig}(r)^{\leq i}\big)$ comes from a unique extension:
$$D_{\rig}(r_{\vec{v}})^{(i)}\in {\rm Ext}^1_{(\varphi,\Gamma_K)}\big(t^{\Sigma_{i}({\bf k},w_x)}{\rm gr}_iD_{\rig}(r),D_{\rig}(r)/D_{\rig}(r)^{\leq i}\big).$$
We thus have to prove that $D_{\rig}(r_{\vec{v}})^{(i)}=0$. 

The twist by the rank one $(\varphi,\Gamma_K)$-module $\wedge_{\mathcal{R}_{k(x),K}}^{i-1}D_{\rig}(r)^{\leq i-1}$ is easily seen (by elementary linear algebra) to induce an isomorphism:
\begin{multline}\label{wedgetwist}
{\rm Ext}^1_{(\varphi,\Gamma_K)}\big(t^{\Sigma_{i}({\bf k},w_x)}{\rm gr}_iD_{\rig}(r),D_{\rig}(r)/D_{\rig}(r)^{\leq i}\big)\buildrel\sim\over\longrightarrow \\
{\rm Ext}^1_{(\varphi,\Gamma_K)}\big(t^{\Sigma_{i}({\bf k},w_x)}\wedge^iD_{\rig}(r)^{\leq i},(D_{\rig}(r)/D_{\rig}(r)^{\leq i})\wedge (\wedge^{i-1}D_{\rig}(r)^{\leq i-1})\big)
\end{multline}
where we write $\wedge^\cdot D_{\rig}(r)$ for $\wedge_{\mathcal{R}_{k(x),K}}^\cdot D_{\rig}(r)$ and where $(D_{\rig}(r)/D_{\rig}(r)^{\leq i})\wedge (\wedge^{i-1}D_{\rig}(r)^{\leq i-1})$ stands for the quotient:
\begin{multline*}
\big(D_{\rig}(r)\wedge (\wedge^{i-1}D_{\rig}(r)^{\leq i-1})\big)/\big(D_{\rig}(r)^{\leq i}\wedge (\wedge^{i-1}D_{\rig}(r)^{\leq i-1})\big)\cong\\
\big(D_{\rig}(r)\wedge (\wedge^{i-1}D_{\rig}(r)^{\leq i-1})\big)/\wedge^{i}D_{\rig}(r)^{\leq i}.
\end{multline*}
(here, $D_{\rig}(r)\wedge (\wedge^{i-1}D_{\rig}(r)^{\leq i-1})$ and $D_{\rig}(r)^{\leq i}\wedge (\wedge^{i-1}D_{\rig}(r)^{\leq i-1})$ are seen inside $\wedge^{i}D_{\rig}(r)$). Moreover the injective map $\wedge^{i-1}D_{\rig}(r)^{\leq i-1}\hookrightarrow \wedge^{i-1}D_{\rig}(r)^{\leq i}$ still induces an injection (using Definition \ref{veryreg}):
\begin{multline*}
{\rm Ext}^1_{(\varphi,\Gamma_K)}\big(t^{\Sigma_{i}({\bf k},w_x)}\wedge^iD_{\rig}(r)^{\leq i},(D_{\rig}(r)/D_{\rig}(r)^{\leq i})\wedge (\wedge^{i-1}D_{\rig}(r)^{\leq i-1})\big)\hookrightarrow\\
{\rm Ext}^1_{(\varphi,\Gamma_K)}\big(t^{\Sigma_{i}({\bf k},w_x)}\wedge^iD_{\rig}(r)^{\leq i},(D_{\rig}(r)/D_{\rig}(r)^{\leq i})\wedge (\wedge^{i-1}D_{\rig}(r)^{\leq i})\big).
\end{multline*}
Denote by:
\begin{equation}\label{rtilde}
\widetilde D_{\rig}(r_{\vec{v}})^{(i)}\ \in \ {\rm Ext}^1_{(\varphi,\Gamma_K)}\big(t^{\Sigma_{i}({\bf k},w_x)}\wedge^iD_{\rig}(r)^{\leq i},(D_{\rig}(r)/D_{\rig}(r)^{\leq i})\wedge (\wedge^{i-1}D_{\rig}(r)^{\leq i})\big)
\end{equation}
the image of $D_{\rig}(r_{\vec{v}})^{(i)}$ (using the isomorphism (\ref{wedgetwist})). It is thus equivalent to prove that $\widetilde D_{\rig}(r_{\vec{v}})^{(i)}=0$. Note that:
\begin{multline}\label{wedgequot}
(D_{\rig}(r)/D_{\rig}(r)^{\leq i})\wedge (\wedge^{i-1}D_{\rig}(r)^{\leq i})\cong \big(D_{\rig}(r)\wedge (\wedge^{i-1}D_{\rig}(r)^{\leq i})\big)/\wedge^{i}D_{\rig}(r)^{\leq i}.
\end{multline}

For $1\leq i\leq n$, we have a $k(x)$-linear map ${\rm Ext}^1_{\mathcal{G}_K}(r,r)\rightarrow {\rm Ext}^1_{\mathcal{G}_K}(\wedge_{k(x)}^ir,\wedge_{k(x)}^ir)$ defined by mapping a $k(x)[\varepsilon]/(\varepsilon^2)$-valued representation of $\mathcal{G}_K$ to its $i$-th exterior power over $k(x)[\varepsilon]/(\varepsilon^2)$. This induces an $\mathcal{R}_{k(x)[\varepsilon]/(\varepsilon^2),K}$-linear map:
$${\rm Ext}^1_{(\varphi,\Gamma_K)}(D_{\rig}(r),D_{\rig}(r))\longrightarrow {\rm Ext}^1_{(\varphi,\Gamma_K)}\big(\wedge^i\!D_{\rig}(r),\wedge^iD_{\rig}(r)\big).$$
Let $D_{\rig}(\wedge^ir_{\vec{v}})\in  {\rm Ext}^1_{(\varphi,\Gamma_K)}(\wedge^i\!D_{\rig}(r),\wedge^iD_{\rig}(r))$ be the image of $D_{\rig}(r_{\vec{v}})$. The pull-back along $\wedge^iD_{\rig}(r)^{\leq i}\hookrightarrow \wedge^iD_{\rig}(r)$ sends $D_{\rig}(\wedge^ir_{\vec{v}})$ to an element in ${\rm Ext}^1_{(\varphi,\Gamma_K)}(\wedge^iD_{\rig}(r)^{\leq i},\wedge^{i}D_{\rig}(r))$. Elementary linear algebra (recall $\varepsilon^2=0$!) shows this element in fact belongs to:
$${\rm Ext}^1_{(\varphi,\Gamma_K)}\big(\wedge^iD_{\rig}(r)^{\leq i},D_{\rig}(r)\wedge (\wedge^{i-1}D_{\rig}(r)^{\leq i})\big)$$
(which embeds into ${\rm Ext}^1_{(\varphi,\Gamma_K)}(\wedge^iD_{\rig}(r)^{\leq i},\wedge^{i}D_{\rig}(r))$ again by Definition \ref{veryreg}). The pushforward along:
$$D_{\rig}(r)\wedge (\wedge^{i-1}D_{\rig}(r)^{\leq i})\twoheadrightarrow (D_{\rig}(r)\wedge (\wedge^{i-1}D_{\rig}(r)^{\leq i}))/\wedge^{i}D_{\rig}(r)^{\leq i}$$
now gives by (\ref{wedgequot}) an element in:
$${\rm Ext}^1_{(\varphi,\Gamma_K)}\big(\wedge^iD_{\rig}(r)^{\leq i},(D_{\rig}(r)/D_{\rig}(r)^{\leq i})\wedge (\wedge^{i-1}D_{\rig}(r)^{\leq i})\big)$$
and further pull-back along $t^{\Sigma_{i}({\bf k},w_x)}\wedge^iD_{\rig}(r)^{\leq i}\hookrightarrow \wedge^iD_{\rig}(r)^{\leq i}$ finally gives an element:
\begin{equation}\label{wedgetilde}
\widetilde D_{\rig}(\wedge^ir_{\vec{v}})\ \in \ {\rm Ext}^1_{(\varphi,\Gamma_K)}\big(t^{\Sigma_{i}({\bf k},w_x)}\wedge^iD_{\rig}(r)^{\leq i},(D_{\rig}(r)/D_{\rig}(r)^{\leq i})\wedge (\wedge^{i-1}D_{\rig}(r)^{\leq i})\big).
\end{equation}
Now, again manipulations of elementary linear algebra show we recover the element $\widetilde D_{\rm rig}(r_{\vec{v}})^{(i)}$ of (\ref{rtilde}), that is, we have $\widetilde D_{\rig}(\wedge^ir_{\vec{v}})=\widetilde D_{\rm rig}(r_{\vec{v}})^{(i)}$.

But we know from Theorem \ref{bergsplit} (using (\ref{wedgedelta}) and Definition \ref{veryreg}) that the image of $D_{\rig}(\wedge^ir_{\vec{v}})$ (by pullback) in ${\rm Ext}^1_{(\varphi,\Gamma_K)}\big(t^{\Sigma_{i}({\bf k},w_x)}\wedge^iD_{\rig}(r)^{\leq i},\wedge^iD_{\rig}(r)\big)$ actually sits in:
$${\rm Ext}^1_{(\varphi,\Gamma_K)}\big(t^{\Sigma_{i}({\bf k},w_x)}\wedge^iD_{\rig}(r)^{\leq i},\wedge^iD_{\rig}(r)^{\leq i}\big)$$
(in fact even in the image of ${\rm Ext}^1_{(\varphi,\Gamma_K)}(t^{\Sigma_{i}({\bf k},w_x)}\wedge^iD_{\rig}(r)^{\leq i},t^{\Sigma_{i}({\bf k},w_x)}\wedge^iD_{\rig}(r)^{\leq i})$). In particular its image $\widetilde D_{\rig}(\wedge^ir_{\vec{v}})$ in:
$${\rm Ext}^1_{(\varphi,\Gamma_K)}\big(t^{\Sigma_{i}({\bf k},w_x)}\wedge^iD_{\rig}(r)^{\leq i},(D_{\rig}(r)\wedge (\wedge^{i-1}D_{\rig}(r)^{\leq i}))/\wedge^{i}D_{\rig}(r)^{\leq i}\big)$$
must be zero. Since $\widetilde D_{\rig}(\wedge^ir_{\vec{v}})=\widetilde D_{\rm rig}(r_{\vec{v}})^{(i)}$, we obtain $\widetilde D_{\rm rig}(r_{\vec{v}})^{(i)}=0$.
\end{proof}

\subsection{Proof of the main local theorem}\label{endofproof}

We compute various dimensions and finish the proof of Theorem \ref{upperbound}.

Seeing an element of ${\rm Ext}^1_{\mathcal{G}_K}(r,r)$ as a $k(x)[\varepsilon]/(\varepsilon^2)$-valued representation of $\mathcal{G}_K$, we can write its Sen weights as $(k_{\tau,i}+\varepsilon d_{\tau,i})_{1\leq i\leq n,\tau:\, K\hookrightarrow L}$ for some $d_{\tau,i}\in k(x)$. We let $V$ be the $k(x)$-subvector space of ${\rm Ext}^1_{\mathcal{G}_K}(r,r)$ (or of ${\rm Ext}^1_{(\varphi,\Gamma_K)}(D_{\rig}(r),D_{\rig}(r))$) of extensions such that $d_{\tau,i}=d_{\tau,w_{x,\tau}^{-1}(i)}$ for $1\leq i\leq n$ and $\tau:\, K\hookrightarrow L$.

\begin{prop}\label{condweight}
We have $\dim_{k(x)}V =\dim_{k(x)}{\rm Ext}^1_{\mathcal{G}_K}(r,r)-d_x$.
\end{prop}
\begin{proof}
The Sen map ${\rm Ext}^1_{\mathcal{G}_K}(r,r)\cong {\rm Ext}^1_{(\varphi,\Gamma_K)}(D_{\rig}(r),D_{\rig}(r))\longrightarrow k(x)^{[K:\Q_p]n}$ sending an extension to $(d_{\tau,i})_{1\leq i\leq n,\tau:\, K\hookrightarrow L}$ is easily checked to be surjective (by a d\'evissage argument using D\'efinition \ref{veryreg}, we are reduced to the rank one case where it is obvious). The \ $k(x)$-subvector \ space \ of \ $k(x)^{[K:\Q_p]n}$ \ of \ tuples $(d_{\tau,i})_{1\leq i\leq n,\tau:\, K\hookrightarrow L}$ such that $d_{\tau,i}=d_{\tau,w_{x,\tau}^{-1}(i)}$ for $1\leq i\leq n$ and $\tau:\, K\hookrightarrow L$ has dimension $[K:\Q_p]n-d_x$ (argue as in the beginning of the proof of Lemma \ref{coxeter}). The result follows.
\end{proof}

Recall that a $(\varphi,\Gamma_K)$-module $D$ over $\mathcal{R}_{k(x),K}$ is called {\it crystalline} if $D[1/\prod_{\tau:\, K\hookrightarrow L}t_\tau]^{\Gamma_K}$ is free over $K_0\otimes_{\Q_p}k(x)$ of the same rank as $D$. If $D,D'$ are two crystalline $(\varphi,\Gamma_K)$-module over $\mathcal{R}_{k(x),K}$, one can define the $k(x)$-subvector space of {\it crystalline extensions} ${\rm Ext}^1_{\rm cris}(D,D')\subseteq {\rm Ext}^1_{(\varphi,\Gamma_K)}(D,D')$. Note that ${\rm Ext}^1_{\rm cris}(\cdot,\cdot)$ respects surjectivities on the right entry (resp. sends injectivities to surjectivities on the left entry) as there is no ${\rm Ext}^2_{\rm cris}$, see \cite[Cor.1.4.6]{Benois}.

\begin{lemm}\label{dim1}
For $1\leq i\leq \ell\leq n$ we have:
\begin{multline*}
\dim_{k(x)}{\rm Ext}^1_{(\varphi,\Gamma_K)}\big({\rm gr}_iD_{\rig}(r)/(t^{\Sigma_{\ell}({\bf k},w_x)}),D_{\rig}(r)/D_{\rig}(r)^{\leq \ell}\big)=\\
\sum_{\tau:\, K\hookrightarrow L}\big\vert\{i_1\in \{\ell+1,\dots,n\}, w_{x,\tau}^{-1}(i_1)<w_{x,\tau}^{-1}(i)\}\big\vert.
\end{multline*}
\end{lemm}
\begin{proof}
It follows from Proposition \ref{cris} below (applied with $(i,\ell)=(i,i)$ and $(i,\ell)=(i-1,i)$) together with the two exact sequences:
\begin{multline*}
0\rightarrow {\rm Ext}^1_{(\varphi,\Gamma_K)}\big({\rm gr}_iD_{\rig}(r)/(t^{\Sigma_{i}({\bf k},w_x)}),D_{\rig}(r)/D_{\rig}(r)^{\leq i}\big)\rightarrow \\
{\rm Ext}^1_{(\varphi,\Gamma_K)}\big(D_{\rig}^{\leq i}(r)/(t^{\Sigma_{i}({\bf k},w_x)}),D_{\rig}(r)/D_{\rig}(r)^{\leq i}\big)\rightarrow \\
{\rm Ext}^1_{(\varphi,\Gamma_K)}\big(D_{\rig}(r)^{\leq i-1}/(t^{\Sigma_{i}({\bf k},w_x)}),D_{\rig}(r)/D_{\rig}(r)^{\leq i}\big),
\end{multline*}
\begin{multline*}
0\rightarrow {\rm Ext}^1_{\rm cris}\big({\rm gr}_iD_{\rig}(r),D_{\rig}(r)/D_{\rig}(r)^{\leq i}\big)\rightarrow {\rm Ext}^1_{\rm cris}\big(D_{\rig}^{\leq i},D_{\rig}(r)/D_{\rig}(r)^{\leq i}\big)
\rightarrow \\
{\rm Ext}^1_{\rm cris}\big(D_{\rig}(r)^{\leq i-1},D_{\rig}(r)/D_{\rig}(r)^{\leq i}\big)
\end{multline*}
(injectivity on the left following again from Definition \ref{veryreg}), that we have:
\begin{multline}\label{dimcrisi}
{\rm Ext}^1_{(\varphi,\Gamma_K)}\big({\rm gr}_iD_{\rig}(r)/(t^{\Sigma_{\ell}({\bf k},w_x)}),D_{\rig}(r)/D_{\rig}(r)^{\leq \ell}\big)\cong \\
{\rm Ext}^1_{\rm cris}\big({\rm gr}_iD_{\rig}(r),D_{\rig}(r)/D_{\rig}(r)^{\leq \ell}\big).
\end{multline}
By d\'evissage on $D_{\rig}(r)/D_{\rig}(r)^{\leq \ell}$ using that ${\rm Ext}^1_{\rm cris}$ respects here short exact sequences (by Definition \ref{veryreg} and the discussion above), we have:
$$\dim_{k(x)}{\rm Ext}^1_{\rm cris}\big({\rm gr}_iD_{\rig}(r),D_{\rig}(r)/D_{\rig}(r)^{\leq \ell}\big)=
\sum_{i_1=\ell+1}^{n}\dim_{k(x)}{\rm Ext}^1_{\rm cris}\big({\rm gr}_iD_{\rig}(r),{\rm gr}_{i_1}D_{\rig}(r)\big).$$
The result follows from (\ref{dimcris}) below.
\end{proof}

\begin{prop}\label{condsplit}
We have:
$$\dim_{k(x)}\big( V_1\cap V_2\cap \cdots \cap V_{n-1}\big) = \dim_{k(x)}{\rm Ext}^1_{\mathcal{G}_K}(r,r)-\big([K:\Q_p]\frac{n(n-1)}{2}-\lg(w_x)\big).$$
\end{prop}
\begin{proof}
To lighten notation in this proof, we write $D_{\rig}$ instead of $D_{\rig}(r)$ and drop the subscript $(\varphi,\Gamma_K)$. We first prove that, for $1\leq i\leq n$, we have an isomorphism of $k(x)$-vector spaces:
\begin{multline}\label{isoi}
V_1\cap\cdots \cap V_{i-1}/V_1\cap\cdots \cap V_{i} \buildrel\sim\over\longrightarrow\\
{\rm Ext}^1\big({\rm gr}_iD_{\rig},D_{\rig}/D_{\rig}^{\leq i}\big)\big/{\rm Ext}^1\big({\rm gr}_iD_{\rig}/(t^{\Sigma_{i}({\bf k},w_x)}),D_{\rig}/D_{\rig}^{\leq i}\big)
\end{multline}
where $V_1\cap\cdots \cap V_{i-1}:={\rm Ext}^1(D_{\rig},D_{\rig})$ if $i=1$. We first define the map. We have the following commutative diagram:
$$\begin{matrix}
&&\scriptstyle 0 &&\scriptstyle 0  &&\scriptstyle 0  &&\\
&&\downarrow &&\downarrow &&\downarrow &&\\
\!\!\!\scriptstyle{0}&\!\!\!\scriptstyle{\rightarrow} &\!\!\!\scriptstyle{{\rm Ext}^1({\rm gr}_iD_{\rig}/(t^{\Sigma_{i}({\bf k},w_x)}),D_{\rig}/D_{\rig}^{\leq i})}&\!\!\!\scriptstyle\rightarrow &\!\!\!\scriptstyle{\rm Ext}^1(D_{\rig}^{\leq i}/(t^{\Sigma_{i}({\bf k},w_x)}),D_{\rig}/D_{\rig}^{\leq i})&\!\!\!\scriptstyle\rightarrow &\!\!\!\scriptstyle{\rm Ext}^1(D_{\rig}^{\leq i-1}/(t^{\Sigma_{i}({\bf k},w_x)}),D_{\rig}/D_{\rig}^{\leq i})&\!\!\!\scriptstyle\rightarrow &\!\!\!\scriptstyle 0\\
&&\downarrow &&\downarrow &&\downarrow &&\\
\!\!\!\scriptstyle 0&\!\!\!\scriptstyle\rightarrow &\!\!\!\scriptstyle{\rm Ext}^1({\rm gr}_iD_{\rig},D_{\rig}/D_{\rig}^{\leq i})&\!\!\!\scriptstyle\rightarrow &\!\!\!\scriptstyle{\rm Ext}^1(D_{\rig}^{\leq i},D_{\rig}/D_{\rig}^{\leq i})&\!\!\!\scriptstyle\rightarrow &\!\!\!\scriptstyle{\rm Ext}^1(D_{\rig}^{\leq i-1},D_{\rig}/D_{\rig}^{\leq i})&\!\!\!\scriptstyle\rightarrow &\!\!\!\scriptstyle 0\\
&&\downarrow &&\downarrow &&\downarrow &&\\
\!\!\!\scriptstyle 0&\!\!\!\scriptstyle\rightarrow &\!\!\!\scriptstyle{\rm Ext}^1(t^{\Sigma_{i}({\bf k},w_x)}{\rm gr}_iD_{\rig},D_{\rig}/D_{\rig}^{\leq i})&\!\!\!\scriptstyle\rightarrow &\!\!\!\scriptstyle{\rm Ext}^1(t^{\Sigma_{i}({\bf k},w_x)}D_{\rig}^{\leq i},D_{\rig}/D_{\rig}^{\leq i})&\!\!\!\scriptstyle\rightarrow &\!\!\!\scriptstyle{\rm Ext}^1(t^{\Sigma_{i}({\bf k},w_x)}D_{\rig}^{\leq i-1},D_{\rig}/D_{\rig}^{\leq i})&\!\!\!\scriptstyle\rightarrow &\!\!\!\scriptstyle 0
\end{matrix}$$
where the injections on top and left and the surjections on the two bottom lines all follow from Definition \ref{veryreg}, and where the surjection on the top right corner follows from Corollary \ref{zero3} below. Denote by $E_i$ the inverse image of ${\rm Ext}^1(D_{\rig}^{\leq i-1}/(t^{\Sigma_{i}({\bf k},w_x)}),D_{\rig}/D_{\rig}^{\leq i})$ in ${\rm Ext}^1(D_{\rig}^{\leq i},D_{\rig}/D_{\rig}^{\leq i})$, then we have an isomorphism:
\begin{multline}\label{coim}
{\rm Ext}^1\big({\rm gr}_iD_{\rig},D_{\rig}/D_{\rig}^{\leq i}\big)\big/{\rm Ext}^1\big({\rm gr}_iD_{\rig}/(t^{\Sigma_{i}({\bf k},w_x)}),D_{\rig}/D_{\rig}^{\leq i}\big)\buildrel\sim\over\rightarrow \\
E_i\big/{\rm Ext}^1\big(D_{\rig}^{\leq i}/(t^{\Sigma_{i}({\bf k},w_x)}),D_{\rig}/D_{\rig}^{\leq i}\big).
\end{multline}
We consider the composition:
\begin{multline}\label{mapi}
V_1\cap\cdots \cap V_{i-1}\hookrightarrow {\rm Ext}^1\big(D_{\rig},D_{\rig}\big)\twoheadrightarrow {\rm Ext}^1\big(D_{\rig}^{\leq i},D_{\rig}/D_{\rig}^{\leq i}\big)\twoheadrightarrow \\
{\rm Ext}^1\big(D_{\rig}^{\leq i},D_{\rig}/D_{\rig}^{\leq i}\big)\big/{\rm Ext}^1\big(D_{\rig}^{\leq i}/(t^{\Sigma_{i}({\bf k},w_x)}),D_{\rig}/D_{\rig}^{\leq i}\big)
\end{multline}
and note that the image of $V_1\cap\cdots \cap V_{i-1}$ falls in $E_i\big/{\rm Ext}^1(D_{\rig}^{\leq i}/(t^{\Sigma_{i}({\bf k},w_x)}),D_{\rig}/D_{\rig}^{\leq i})$ by Corollary \ref{zero2} below. If $v\in V_1\cap\cdots \cap V_{i-1}$ is also in $V_i$, then its image in ${\rm Ext}^1(D_{\rig}^{\leq i},D_{\rig}/D_{\rig}^{\leq i})$ maps to $0$ in ${\rm Ext}^1(t^{\Sigma_{i}({\bf k},w_x)}D_{\rig}^{\leq i},D_{\rig}/D_{\rig}^{\leq i})$, hence belongs to ${\rm Ext}^1(D_{\rig}^{\leq i}/(t^{\Sigma_{i}({\bf k},w_x)}),D_{\rig}/D_{\rig}^{\leq i})$. By (\ref{coim}), we thus have a canonical induced map:
\begin{multline}\label{imap}
V_1\cap\cdots \cap V_{i-1}/V_1\cap\cdots \cap V_{i} \rightarrow \\
{\rm Ext}^1\big({\rm gr}_iD_{\rig},D_{\rig}/D_{\rig}^{\leq i}\big)\big/{\rm Ext}^1\big({\rm gr}_iD_{\rig}/(t^{\Sigma_{i}({\bf k},w_x)}),D_{\rig}/D_{\rig}^{\leq i}\big).
\end{multline}
Let us prove that (\ref{imap}) is surjective. One easily checks that ${\rm Ext}^1(D_{\rig}/D_{\rig}^{\leq i-1},D_{\rig})\subseteq V_1\cap\cdots \cap V_{i-1}$ and that the natural map ${\rm Ext}^1(D_{\rig}/D_{\rig}^{\leq i-1},D_{\rig})\rightarrow {\rm Ext}^1\big({\rm gr}_iD_{\rig},D_{\rig}/D_{\rig}^{\leq i}\big)$ is surjective (again by Definition \ref{veryreg}). This implies that {\it a fortiori} (\ref{imap}) must also be surjective. Let us prove that (\ref{imap}) is injective. If $v\in V_1\cap\cdots \cap V_{i-1}$ maps to zero, then the image of $v$ in ${\rm Ext}^1(D_{\rig}^{\leq i},D_{\rig}/D_{\rig}^{\leq i})$ belongs to ${\rm Ext}^1(D_{\rig}^{\leq i}/(t^{\Sigma_{i}({\bf k},w_x)}),D_{\rig}/D_{\rig}^{\leq i})$ by (\ref{coim}), i.e. maps to zero in ${\rm Ext}^1(t^{\Sigma_{i}({\bf k},w_x)}D_{\rig}^{\leq i},D_{\rig}/D_{\rig}^{\leq i})$, i.e. $v\in V_i$ by (\ref{viker}), hence $v\in V_1\cap\cdots \cap V_{i}$.

We now prove the statement of the proposition. From (\ref{isoi}) and Lemma \ref{dim1}, we obtain for $1\leq i\leq n$:
\begin{multline*}
\dim_{k(x)}\big(V_1\cap\cdots \cap V_{i-1}/V_1\cap\cdots \cap V_{i}\big)=\\
[K:\Q_p](n-i)-\sum_{\tau:\, K\hookrightarrow L}\big\vert\{j\in \{i+1,\dots,n\}, w_{x,\tau}^{-1}(j)<w_{x,\tau}^{-1}(i)\}\big\vert=\\
\sum_{\tau:\, K\hookrightarrow L}\big\vert\{j\in \{i+1,\dots,n\}, w_{x,\tau}^{-1}(i)<w_{x,\tau}^{-1}(j)\}\big\vert.
\end{multline*}
Summing up $\dim_{k(x)}(V_1\cap\cdots \cap V_{i-1}/V_1\cap\cdots \cap V_{i})$ for $i=1$ to $n-1$ thus yields:
\begin{multline*}
\dim_{k(x)}{\rm Ext}^1\big(D_{\rig},D_{\rig}\big)-\dim_{k(x)}(V_1\cap\cdots \cap V_{n-1})= \\
\sum_{\tau:\, K\hookrightarrow L}\big\vert\{1\leq i_1<i_2\leq n,\ w_{x,\tau}^{-1}(i_1)<w_{x,\tau}^{-1}(i_2)\}\big\vert.
\end{multline*}
But $\vert\{1\leq i_1<i_2\leq n,\ w_{x,\tau}^{-1}(i_1)<w_{x,\tau}^{-1}(i_2)\}\vert= \frac{n(n-1)}{2}-\lg(w_{x,\tau})$ (see e.g. \cite[\S0.3]{HumBGG}), and thus we get:
\begin{eqnarray*}
\dim_{k(x)}(V_1\cap\cdots \cap V_{n-1})&= &\dim_{k(x)}{\rm Ext}^1\big(D_{\rig},D_{\rig}\big)-\sum_{\tau:\, K\hookrightarrow L}\big(\frac{n(n-1)}{2}-\lg(w_{x,\tau})\big)\\
&= &\dim_{k(x)}{\rm Ext}^1\big(D_{\rig},D_{\rig}\big)-\big([K:\Q_p]\frac{n(n-1)}{2}-\lg(w_x)\big)
\end{eqnarray*}
which finishes the proof.
\end{proof}

\begin{prop}\label{inter}
We have:
$$\dim_{k(x)}\!\big(V \cap (V_1\cap \cdots \cap V_{n-1})\big)\!= \dim_{k(x)}{\rm Ext}^1_{\mathcal{G}_K}(r,r)-d_x-\big([K:\Q_p]\frac{n(n-1)}{2}-\lg(w_x)\big).$$
\end{prop}
\begin{proof}
Consider the following cartesian diagram which defines $W_i$:
$$\xymatrix{{\rm Ext}^1_{(\varphi,\Gamma_K)}\big(D_{\rig}(r),D_{\rig}(r)\big)\ar[r]&{\rm Ext}^1_{(\varphi,\Gamma_K)}\big(D_{\rig}(r)^{\leq i},D_{\rig}(r)\big)\\ 
W_i \ar[r]\ar@{^{(}->}[u]& 
{\rm Ext}^1_{(\varphi,\Gamma_K)}\big(D_{\rig}(r)^{\leq i},D_{\rig}(r)^{\leq i}\big)\ar@{^{(}->}[u]},$$
then $W_i\subseteq V_i$, hence $W_1\cap \cdots \cap W_{n-1}\subseteq V_1\cap \cdots \cap V_{n-1}$. In fact, $W_1\cap \cdots \cap W_{n-1}$ is the $k(x)$-subvector space of ${\rm Ext}^1_{(\varphi,\Gamma_K)}(D_{\rig}(r),D_{\rig}(r))$ of extensions which respect the triangulation $(D_{\rig}(r)^{\leq i})_{1\leq i\leq n}$ on $D_{\rig}(r)$. A d\'evissage argument (using Definition \ref{veryreg}) that we leave to the reader then shows that the composition:
$$W_1\cap \cdots \cap W_{n-1}\hookrightarrow {\rm Ext}^1_{(\varphi,\Gamma_K)}\big(D_{\rig}(r),D_{\rig}(r)\big)\longrightarrow k(x)^{[K:\Q_p]n}$$
(where the second map is the Sen map in the proof of Proposition \ref{condweight}) remains surjective. {\it A fortiori}, $V_1\cap \cdots \cap V_{n-1}\hookrightarrow {\rm Ext}^1_{(\varphi,\Gamma_K)}(D_{\rig}(r),D_{\rig}(r))\rightarrow k(x)^{[K:\Q_p]n}$ is also surjective. By the same proof as that of Proposition \ref{condweight} we get:
$$\dim_{k(x)}\!\big(V \cap (V_1\cap \cdots \cap V_{n-1})\big)\!= \dim_{k(x)}(V_1\cap \cdots \cap V_{n-1})-d_x$$
and the result follows from Proposition \ref{condsplit}.
\end{proof}

\begin{coro}\label{ouf}
The image of $T_{X,x}$ in ${\rm Ext}^1_{\mathcal{G}_K}(r,r)$ has dimension $\leq \dim_{k(x)}{\rm Ext}^1_{\mathcal{G}_K}(r,r)-d_x-\big([K:\Q_p]\frac{n(n-1)}{2}-\lg(w_x)\big)$.
\end{coro}
\begin{proof}
It follows from Theorem \ref{bergweight} that the image of any $\vec{v}\in T_{X,x}$ in ${\rm Ext}^1_{\mathcal{G}_K}(r,r)$ is in $V$. It follows from Corollary \ref{inclv} that the image of any $\vec{v}\in T_{X,x}$ in ${\rm Ext}^1_{\mathcal{G}_K}(r,r)$ is also in $V_1\cap \cdots \cap V_{n-1}$. One concludes with Proposition \ref{inter}.
\end{proof}

By Lemma \ref{boundtx} this finishes the proof of Theorem \ref{upperbound}.

\begin{rema}
{\rm The collection of $(\varphi,\Gamma_K)$-submodules $(t^{\Sigma_i({\bf k},w_x)}D_{\rig}(r)^{\leq i})_{1\leq i\leq n}$ of $D_{\rig}(r)$ plays an important role in the proof of Corollary \ref{ouf}. One can wonder if they ``globalize'' in a neighbourhood of the point $x$ in $X$? Note that those for which $\Sigma_i({\bf k},w_x)=0$ do by \cite[Th.A]{Bergdall}.}
\end{rema}

\subsection{Calculation of some Ext groups}

We prove several technical but crucial results of Galois cohomology that were used above. 

For a continuous character $\delta:K^\times\rightarrow L^\times$ and $\tau:\, K\hookrightarrow L$, we define its (Sen) weight ${\rm wt}_\tau(\delta)\in L$ in the direction $\tau$ by taking the {\it opposite} of the weight defined in \cite[\S2.3]{Bergdall}. For instance ${\rm wt}_\tau(\tau(z)^{k_\tau})=k_\tau$ ($k_\tau\in \Z$).

\begin{lemm}\label{null}
Let $\tau:\, K\hookrightarrow L$ and $k_\tau\in \Z_{>0}$.\\
(i) For $j\in \{0,1\}$ we have ${\rm Ext}^j_{(\varphi,\Gamma_K)}\big(\mathcal{R}_{L,K},\mathcal{R}_{L,K}(\delta)/(t_\tau^{k_\tau})\big)\ne 0$ if and only if ${\rm wt}_\tau(\delta)\in \{-(k_\tau-1),\dots,0\}$ and we have ${\rm Ext}^2_{(\varphi,\Gamma_K)}\big(\mathcal{R}_{L,K},\mathcal{R}_{L,K}(\delta)/(t_\tau^{k_\tau})\big)=0$ for all $\delta$.\\
(ii) For $j\in \{1,2\}$ we have ${\rm Ext}^j_{(\varphi,\Gamma_K)}\big(\mathcal{R}_{L,K}(\delta)/(t_\tau^{k_\tau}),\mathcal{R}_{L,K}\big)\ne 0$ if and only if ${\rm wt}_\tau(\delta)\in \{-k_\tau,\dots,-1\}$ and we have ${\rm Ext}^0_{(\varphi,\Gamma_K)}\big(\mathcal{R}_{L,K}(\delta)/(t_\tau^{k_\tau}),\mathcal{R}_{L,K}\big)=0$ for all $\delta$.\\
(iii) When either of these spaces is nonzero, it has dimension $1$ over $L$.
\end{lemm}
\begin{proof}
The first part of (i) is in \cite[Prop.2.7]{Bergdall} (and initially in \cite[Prop.2.18]{Co} for $K=\Q_p$) and the second part in \cite[Th.3.7(2)]{Liuduality}. The second part of (ii) is obvious, let us prove the first. We have an exact sequence:
\begin{equation}\label{exactR}
0\longrightarrow \mathcal{R}_{L,K}(\delta^{-1})\longrightarrow \mathcal{R}_{L,K}(\tau(z)^{-k_\tau}\delta^{-1})\longrightarrow \mathcal{R}_{L,K}(\tau(z)^{-k_\tau}\delta^{-1})/(t_\tau^{k_\tau})\longrightarrow 0.
\end{equation}
The cup product with (\ref{exactR}) yields canonical morphisms of $L$-vector spaces:
\begin{eqnarray*}\label{mapL}
{\rm Ext}^0_{(\varphi,\Gamma_K)}\big(\mathcal{R}_{L,K}/(t_\tau^{k_\tau}),\mathcal{R}_{L,K}(\tau(z)^{-k_\tau}\delta^{-1})/(t_\tau^{k_\tau})\big)&\rightarrow & {\rm Ext}^1_{(\varphi,\Gamma_K)}\big(\mathcal{R}_{L,K}/(t_\tau^{k_\tau}),\mathcal{R}_{L,K}(\delta^{-1})\big)\\
\nonumber {\rm Ext}^1_{(\varphi,\Gamma_K)}\big(\mathcal{R}_{L,K}/(t_\tau^{k_\tau}),\mathcal{R}_{L,K}(\tau(z)^{-k_\tau}\delta^{-1})/(t_\tau^{k_\tau})\big)&\rightarrow & {\rm Ext}^2_{(\varphi,\Gamma_K)}\big(\mathcal{R}_{L,K}/(t_\tau^{k_\tau}),\mathcal{R}_{L,K}(\delta^{-1})\big).
\end{eqnarray*}
There is an obvious isomorphism of $L$-vector spaces:
\begin{eqnarray*}
{\rm Ext}^0_{(\varphi,\Gamma_K)}\big(\mathcal{R}_{L,K},\mathcal{R}_{L,K}(\tau(z)^{-k_\tau}\delta^{-1})/(t_\tau^{k_\tau})\big)&\!\!\cong \!\!& {\rm Ext}^0_{(\varphi,\Gamma_K)}\big(\mathcal{R}_{L,K}/(t_\tau^{k_\tau}),\mathcal{R}_{L,K}(\tau(z)^{-k_\tau}\delta^{-1})/(t_\tau^{k_\tau})\big)
\end{eqnarray*}
and an analysis of the cokernel of the multiplication by $t_\tau^{k_\tau}$ map on a short exact sequence $0\rightarrow  \mathcal{R}_{L,K}(\tau(z)^{-k_\tau}\delta^{-1})/(t_\tau^{k_\tau})\rightarrow {\mathcal E}\rightarrow \mathcal{R}_{L,K}\rightarrow 0$ of $(\varphi,\Gamma_K)$-module over $\mathcal{R}_{L,K}$ yields a canonical morphism of $L$-vector spaces:
\begin{eqnarray*}
{\rm Ext}^1_{(\varphi,\Gamma_K)}\big(\mathcal{R}_{L,K},\mathcal{R}_{L,K}(\tau(z)^{-k_\tau}\delta^{-1})/(t_\tau^{k_\tau})\big)&\!\!\!\rightarrow \!\!\!& {\rm Ext}^1_{(\varphi,\Gamma_K)}\big(\mathcal{R}_{L,K}/(t_\tau^{k_\tau}),\mathcal{R}_{L,K}(\tau(z)^{-k_\tau}\delta^{-1})/(t_\tau^{k_\tau})\big).
\end{eqnarray*}
Thus we have canonical morphisms of $L$-vector spaces:
\begin{eqnarray}\label{mapL2}
{\rm Ext}^0_{(\varphi,\Gamma_K)}\big(\mathcal{R}_{L,K},\mathcal{R}_{L,K}(\tau(z)^{-k_\tau}\delta^{-1})/(t_\tau^{k_\tau})\big)&\!\!\!\rightarrow \!\!\!& {\rm Ext}^1_{(\varphi,\Gamma_K)}\big(\mathcal{R}_{L,K}/(t_\tau^{k_\tau}),\mathcal{R}_{L,K}(\delta^{-1})\big)\\
\nonumber {\rm Ext}^1_{(\varphi,\Gamma_K)}\big(\mathcal{R}_{L,K},\mathcal{R}_{L,K}(\tau(z)^{-k_\tau}\delta^{-1})/(t_\tau^{k_\tau})\big)&\!\!\!\rightarrow \!\!\!& {\rm Ext}^2_{(\varphi,\Gamma_K)}\big(\mathcal{R}_{L,K}/(t_\tau^{k_\tau}),\mathcal{R}_{L,K}(\delta^{-1})\big).
\end{eqnarray}
It is then a simple exercise of linear algebra to check that the morphisms in (\ref{mapL}) fit into a natural morphism of complexes of $L$-vector spaces from the long exact sequence of ${\rm Ext}^j_{(\varphi,\Gamma_K)}(\mathcal{R}_{L,K},\cdot)$ applied to the short exact sequence (\ref{exactR}) to the long exact sequence of ${\rm Ext}^j_{(\varphi,\Gamma_K)}(\cdot,\mathcal{R}_{L,K}(\delta^{-1}))$ applied to the short exact sequence $0\rightarrow t_\tau^{k_\tau}\mathcal{R}_{L,K}\rightarrow \mathcal{R}_{L,K}\rightarrow \mathcal{R}_{L,K}/(t_\tau^{k_\tau})\rightarrow 0$ (note that there is a shift in this map of complexes). Since all the morphisms are obviously isomorphisms except possibly the morphisms (\ref{mapL2}), we deduce that the latter are also isomorphisms. Twisting by $\mathcal{R}_{L,K}(\delta)$ on the right hand side of (\ref{mapL2}) and using (i) applied to the left hand side, the first part of (ii) easily follows. Finally (iii) follows from \cite[Prop.2.7]{Bergdall} and from the previous isomorphisms (\ref{mapL2}).
\end{proof}

Recall that for $i,\ell\in \{1,\dots,n\}$ we have an exact sequence:
\begin{multline}\label{ses}
0\rightarrow {\rm Ext}^1_{(\varphi,\Gamma_K)}\big(D_{\rig}(r)^{\leq i}/t^{\Sigma_{\ell}({\bf k},w_x)}D_{\rig}(r)^{\leq i},D_{\rig}(r)/D_{\rig}(r)^{\leq \ell}\big)\rightarrow \\
{\rm Ext}^1_{(\varphi,\Gamma_K)}\big(D_{\rig}(r)^{\leq i},D_{\rig}(r)/D_{\rig}(r)^{\leq \ell}\big)\rightarrow \\
{\rm Ext}^1_{(\varphi,\Gamma_K)}\big(t^{\Sigma_{\ell}({\bf k},w_x)}D_{\rig}(r)^{\leq i},D_{\rig}(r)/D_{\rig}(r)^{\leq \ell}\big)
\end{multline}
where the injection on the left follows as usual from Definition \ref{veryreg}.

\begin{prop}\label{cris}
For \ $1\leq i\leq \ell\leq n$, \ we \ have \ an \ isomorphism \ of \ subspaces \ of ${\rm Ext}^1_{(\varphi,\Gamma_K)}(D_{\rig}(r)^{\leq i},D_{\rig}(r)/D_{\rig}(r)^{\leq \ell})$:
\begin{multline*}
{\rm Ext}^1_{(\varphi,\Gamma_K)}\big(D_{\rig}(r)^{\leq i}/t^{\Sigma_{\ell}({\bf k},w_x)}D_{\rig}(r)^{\leq i},D_{\rig}(r)/D_{\rig}(r)^{\leq \ell}\big)\cong \\
{\rm Ext}^1_{\rm cris}\big(D_{\rig}(r)^{\leq i},D_{\rig}(r)/D_{\rig}(r)^{\leq \ell}\big).
\end{multline*}
\end{prop}
\begin{proof}
To lighten notation, we write $D_{\rig}$ instead of $D_{\rig}(r)$ and drop the subscript $(\varphi,\Gamma_K)$. By the exact sequence (\ref{ses}) and a d\'evissage on $D_{\rig}^{\leq i}$ and $D_{\rig}/D_{\rig}^{\leq \ell}$ (recall from Definition \ref{veryreg} and the discussion preceding Lemma \ref{dim1} that ${\rm Ext}^1_{\rm cris}$ respects short exact sequences here), it is enough to prove (i) that the composition:
$${\rm Ext}^1_{\rm cris}\big(D_{\rig}^{\leq i},D_{\rig}/D_{\rig}^{\leq \ell}\big)\subseteq {\rm Ext}^1\big(D_{\rig}^{\leq i},D_{\rig}/D_{\rig}^{\leq \ell}\big)\longrightarrow  {\rm Ext}^1\big(t^{\Sigma_{\ell}({\bf k},w_x)}D_{\rig}^{\leq i},D_{\rig}/D_{\rig}^{\leq \ell}\big)$$
is zero and (ii) that:
$${\rm Ext}^1\big({\rm gr}_{\ell'}D_{\rig}/(t^{\Sigma_{\ell}({\bf k},w_x)}),{\rm gr}_{\ell''}D_{\rig}\big)\cong {\rm Ext}^1_{\rm cris}\big({\rm gr}_{\ell'}D_{\rig},{\rm gr}_{\ell''}D_{\rig}\big)$$
(inside ${\rm Ext}^1({\rm gr}_{\ell'}D_{\rig},{\rm gr}_{\ell''}D_{\rig})$) for all $\ell',\ell''$ such that $\ell'\leq \ell$ and $\ell''\geq \ell+1$. 

We prove (i). The map clearly factors through:
$${\rm Ext}^1_{\rm cris}\big(t^{\Sigma_{\ell}({\bf k},w_x)}D_{\rig}^{\leq i},D_{\rig}/D_{\rig}^{\leq \ell}\big),$$
let us prove that the latter vector space is zero. By d\'evissage again, it is enough to prove that:
$${\rm Ext}^1_{\rm cris}\big(t^{\Sigma_{\ell}({\bf k},w_x)}{\rm gr}_{\ell'}D_{\rig},{\rm gr}_{\ell''}D_{\rig}\big)=0$$
for $\ell',\ell''$ such that $\ell'\leq \ell$ and $\ell''\geq \ell+1$. It is enough to prove that, for all $\tau:\, K\hookrightarrow L$, we have ${\rm wt}_\tau(t^{\Sigma_{\ell}({\bf k},w_x)}{\rm gr}_{\ell'}D_{\rig})\geq {\rm wt}_\tau({\rm gr}_{\ell''}D_{\rig})$ (using Definition \ref{veryreg} when these two weights are equal). This is equivalent to:
\begin{eqnarray}\label{bound}
\sum_{j=1}^\ell(k_{\tau,j} - k_{\tau,w_{x,\tau}^{-1}(j)})+k_{\tau,w_{x,\tau}^{-1}(\ell')}\geq k_{\tau,w_{x,\tau}^{-1}(\ell'')}
\end{eqnarray}
which indeed holds for $\ell'$, $\ell''$ as above because $k_{\tau,1}>k_{\tau,2}>\cdots>k_{\tau,n}$. 

We prove (ii). From (i) we have in particular an inclusion:
\begin{eqnarray}\label{incl}
{\rm Ext}^1_{\rm cris}\big({\rm gr}_{\ell'}D_{\rig},{\rm gr}_{\ell''}D_{\rig}\big)\subseteq {\rm Ext}^1\big({\rm gr}_{\ell'}D_{\rig}/(t^{\Sigma_{\ell}({\bf k},w_x)}),{\rm gr}_{\ell''}D_{\rig}\big).
\end{eqnarray}
It is an easy (and well-known) exercise that we leave to the reader to check that:
\begin{eqnarray}\label{dimcris}
\dim_{k(x)}{\rm Ext}^1_{\rm cris}\big({\rm gr}_{\ell'}D_{\rig},{\rm gr}_{\ell''}D_{\rig}\big)=\big\vert\{\tau:\, K\hookrightarrow L, w_{x,\tau}^{-1}(\ell'')<w_{x,\tau}^{-1}(\ell')\}\big\vert.
\end{eqnarray}
On the other hand, from (ii) and (iii) of Lemma \ref{null}, using (\ref{bound}) and $\mathcal{R}_{L,K}(\delta)/(t_\tau^{k_\tau}t_\sigma^{k_\sigma})\cong \mathcal{R}_{L,K}(\delta)/(t_\tau^{k_\tau})\times \mathcal{R}_{L,K}(\delta)/(t_\sigma^{k_\sigma})$ if $\tau\ne\sigma$, we deduce:
\begin{multline}\label{dimrig}
\dim_{k(x)}{\rm Ext}^1\big({\rm gr}_{\ell'}D_{\rig}/(t^{\Sigma_{\ell}({\bf k},w_x)}),{\rm gr}_{\ell''}D_{\rig}\big)=\big\vert\{\tau:\, K\hookrightarrow L, w_{x,\tau}^{-1}(\ell'')<w_{x,\tau}^{-1}(\ell')\}\big\vert.
\end{multline}
(\ref{incl}), (\ref{dimcris}) and (\ref{dimrig}) imply ${\rm Ext}^1_{\rm cris}({\rm gr}_{\ell'}D_{\rig},{\rm gr}_{\ell''}D_{\rig})\cong {\rm Ext}^1({\rm gr}_{\ell'}D_{\rig}/(t^{\Sigma_{\ell}({\bf k},w_x)}),{\rm gr}_{\ell''}D_{\rig})$ which finishes the proof.
\end{proof}

\begin{coro}\label{zero2}
Let $i\in \{1,\dots,n\}$, ${\mathcal E}\in {\rm Ext}^1_{(\varphi,\Gamma_K)}(D_{\rig}(r),D_{\rig}(r))$ and assume that the image of $\mathcal E$ (by pullback and pushforward) in:
$${\rm Ext}^1_{(\varphi,\Gamma_K)}\big(t^{\Sigma_{i-1}({\bf k},w_x)}D_{\rig}(r)^{\leq i-1},D_{\rig}(r)/D_{\rig}(r)^{\leq i-1}\big)$$
is zero. Then the image of $\mathcal E$ in:
$${\rm Ext}^1_{(\varphi,\Gamma_K)}\big(t^{\Sigma_{i}({\bf k},w_x)}D_{\rig}(r)^{\leq i-1},D_{\rig}(r)/D_{\rig}(r)^{\leq i}\big)$$
is also zero.
\end{coro}
\begin{proof}
By \ Proposition \ \ref{cris} \ applied \ with \ $(i,\ell)=(i-1,i-1)$, \ the \ image \ of \ $\mathcal E$ \ in ${\rm Ext}^1_{(\varphi,\Gamma_K)}(D_{\rig}(r)^{\leq i-1},D_{\rig}(r)/D_{\rig}(r)^{\leq i-1})$ sits in:
$${\rm Ext}^1_{\rm cris}\big(D_{\rig}(r)^{\leq i-1},D_{\rig}(r)/D_{\rig}(r)^{\leq i-1}\big).$$
Hence its image in ${\rm Ext}^1_{(\varphi,\Gamma_K)}(D_{\rig}(r)^{\leq i-1},D_{\rig}(r)/D_{\rig}(r)^{\leq i})$ sits in:
$${\rm Ext}^1_{\rm cris}\big(D_{\rig}(r)^{\leq i-1},D_{\rig}(r)/D_{\rig}(r)^{\leq i}\big).$$
It \ follows \ from \ Proposition \ \ref{cris} \ again \ applied \ with \ $(i,\ell)=(i-1,i)$ \ that \ it \ maps \ to zero in ${\rm Ext}^1_{(\varphi,\Gamma_K)}(t^{\Sigma_{i}({\bf k},w_x)}D_{\rig}(r)^{\leq i-1},D_{\rig}(r)/D_{\rig}(r)^{\leq i})$.
\end{proof}

\begin{coro}\label{zero3}
For \ $2\leq i\leq n$ we have a surjection:
\begin{multline*}
{\rm Ext}^1_{(\varphi,\Gamma_K)}\big(D_{\rig}(r)^{\leq i}/t^{\Sigma_{i}({\bf k},w_x)}D_{\rig}(r)^{\leq i},D_{\rig}(r)/D_{\rig}(r)^{\leq i}\big)\twoheadrightarrow \\
{\rm Ext}^1_{(\varphi,\Gamma_K)}\big(D_{\rig}(r)^{\leq i-1}/t^{\Sigma_{i}({\bf k},w_x)}D_{\rig}(r)^{\leq i-1},D_{\rig}(r)/D_{\rig}(r)^{\leq i}\big)
\end{multline*}
where the map is the pullback along:
$$D_{\rig}(r)^{\leq i-1}/t^{\Sigma_{i}({\bf k},w_x)}D_{\rig}(r)^{\leq i-1}\hookrightarrow D_{\rig}(r)^{\leq i}/t^{\Sigma_{i}({\bf k},w_x)}D_{\rig}(r)^{\leq i}.$$
\end{coro}
\begin{proof}
This follows from Proposition \ref{cris} (applied with $(i,\ell)=(i,i)$ and $(i,\ell)=(i-1,i)$) and the fact that the map:
$${\rm Ext}^1_{\rm cris}\big(D_{\rig}(r)^{\leq i},D_{\rig}(r)/D_{\rig}(r)^{\leq i}\big)\longrightarrow
{\rm Ext}^1_{\rm cris}\big(D_{\rig}(r)^{\leq i-1},D_{\rig}(r)/D_{\rig}(r)^{\leq i}\big)$$
is surjective.
\end{proof}

\section{Modularity and local geometry of the trianguline variety}\label{modularity}

We prove that the main conjecture of \cite{BHS} (see \cite[Conj.3.22]{BHS}), and thus the classical modularity conjectures by \cite[Prop.3.26]{BHS}, imply Conjecture \ref{mainconj} when $\rbar$ ``globalizes'' and $x$ is very regular.

\subsection{A closed embedding}\label{closed}

Assuming the main conjecture of \cite{BHS} and using Theorem \ref{intercompanion} below we construct a certain closed embedding in the trianguline variety (Proposition \ref{companion}).

We fix a continuous representation $\rbar:\mathcal{G}_K\rightarrow \GL_n(k_L)$ as in \S\ref{begin} and keep the local notation of \S\ref{localpart1} and \S\ref{phiGammacohomology}. We also assume that there exist number fields $F/F^+$, a unitary group $G/F^+$, a tame level $U^p$, a set of finite places $S$ and an irreducible representation $\rhobar$ as in \S\ref{classic} such that all the assumptions in \S\ref{classic} and \S\ref{firstclassical} are satisfied, and such that for each place $v\in S_p$ there is a place $\tilde v$ of $F$ dividing $v$ satisfying $F_{\tilde v}\cong K$ and $\rhobar_{\tilde v}\cong \rbar$. Note that this implies in particular $(2n,p)=1$ (as $p>2$ and as $(n,p)=1$ by the proof of \cite[Th.9]{GHTT}). Assuming $(2n,p)=1$, it follows from \cite[Lem.2.2]{CEGGPS} and \cite[\S2.3]{CEGGPS} that such $(F/F^+,G,U^p,S,\rhobar)$ always exist if $n=2$ or if $\rbar$ is (absolutely) semi-simple (increasing $L$ if necessary).

We recall the statement of \cite[Conj.3.22]{BHS} (see \S\ref{weyl} for $\widetilde X_{\rm tri}^\square(\rbar)$).

\begin{conj}\label{BHS}
The rigid subvariety $X_{\rm tri}^{\Xfrak^p \rm-aut}(\rhobar_p)$ of $X_{\rm tri}^\square(\rhobar_p)$ doesn't depend on $\Xfrak^p$ and is isomorphic to $\widetilde X_{\rm tri}^\square(\rhobar_p):= \prod_{v\in S_p}\widetilde X_{\rm tri}^\square(\rhobar_{\tilde v})$.
\end{conj}

\begin{rema}\label{conjvariant}
{\rm (i) By (\ref{union}), Conjecture \ref{BHS} is thus equivalent to $X_p(\rhobar)\buildrel{\sim}\over \rightarrow  \Xfrak_{\rhobar^p}\times \widetilde X_{\rm tri}^\square(\rhobar_p)\times \Ubb^g$.\\
(ii) The authors do not know if $\widetilde X_{\rm tri}^\square(\rhobar_p)$ is really strictly smaller than $X_{\rm tri}^\square(\rhobar_p)$.\\
(iii) Finally, recall that Conjecture \ref{BHS} is {\it implied} by the classical modularity lifting conjectures for $\rhobar$ (in all weights with trivial inertial type), see \cite[Prop.3.26]{BHS}.}
\end{rema}

Let ${\bf k}:=({\bf k}_i)_{1\leq i\leq n}$ where ${\bf k}_i:=(k_{\tau,i})_{\tau:\, K\hookrightarrow L}\in\mathbb{Z}^{\Hom(K,L)}$ is such that $k_{\tau,i}> k_{\tau,i+1}$ for all $i$ and $\tau$. For $w=(w_{\tau})_{\tau:\, K\hookrightarrow L}\in W=\prod_{\tau: K\hookrightarrow L}\Scal_n$, denote by $\mathcal{W}^n_{w,{\bf k},L}\subset \mathcal{W}^n_L$ the Zariski-closed (reduced) subset of characters $(\eta_1,\dots,\eta_n)$ defined by the equations:
\begin{equation}\label{equations}
{\rm wt}_\tau(\eta_{w_\tau(i)}\eta_i^{-1})=k_{\tau,i}-k_{\tau,w_\tau^{-1}(i)},\ \ 1\leq i\leq n,\ \ \tau:\, K\hookrightarrow L.
\end{equation}
For instance one always has:
\begin{equation}\label{point}
(z^{{\bf k}_{w^{-1}(1)}}\chi_1,\dots,z^{{\bf k}_{w^{-1}(n)}}\chi_n)\in \mathcal{W}^n_{w,{\bf k},L}
\end{equation}
where $\chi_i\in \mathcal{W}_L$ are finite order characters. Note that $\mathcal{W}^n_{1,{\bf k},L}=\mathcal{W}^n_L$. We define an automorphism $\jmath_{w,{\bf k}}:\mathcal{T}^n_L\buildrel\sim\over\rightarrow \mathcal{T}^n_L$, $\eta=(\eta_1,\dots,\eta_n)\mapsto \jmath_{w,{\bf k}}(\eta)=\jmath_{w,{\bf k}}(\eta_1,\dots,\eta_n)$ by:
$$\jmath_{w,{\bf k}}(\eta_1,\dots,\eta_n):= (z^{{\bf k}_1-{\bf k}_{w^{-1}(1)}}\eta_1,\dots,z^{{\bf k}_n-{\bf k}_{w^{-1}(n)}}\eta_n)$$
which we extend to an automorphism $\jmath_{w,{\bf k}}:\mathfrak{X}_{\rbar}^\square\times \mathcal{T}^n_L\buildrel\sim\over\rightarrow\mathfrak{X}_{\rbar}^\square\times \mathcal{T}^n_L$, $(r,\eta)\mapsto (r,\jmath_{w,{\bf k}}(\eta))$. We will be particularly interested in applying $\jmath_{w,{\bf k}}$ to points whose image in $\mathcal{W}^n_L$ lies in $\mathcal{W}^n_{w,{\bf k},L}$.

\begin{ex}\label{example}
{\rm Consider the case $[K:\Q_p]=2$ (so $\Hom(K,L)=\{\tau,\tau'\}$), $n=3$ and $w=(w_\tau,w_{\tau'})$ with $w_{\tau}=s_1s_2s_1$, $w_{\tau'}=s_2s_1$ ($s_1,s_2$ being the simple reflections in ${\mathcal S}_3$). Then $\mathcal{W}^3_{w,{\bf k},L}$ is the set of characters of the form:
$$\eta=(\eta_1,\eta_2,\eta_3) =  \big(\tau(z)^{k_{\tau,3}}\tau'(z)^{k_{\tau',2}}\chi_1, \ \tau(z)^{k_{\tau,2}}\tau'(z)^{k_{\tau',3}}\chi_2, \ \tau(z)^{k_{\tau,1}}\tau'(z)^{k_{\tau',1}}\chi_3\big)$$
where ${\rm wt}_\tau(\chi_1)={\rm wt}_\tau(\chi_3)$ and ${\rm wt}_{\tau'}(\chi_1)={\rm wt}_{\tau'}(\chi_2)={\rm wt}_{\tau'}(\chi_3)$. Note that there is no condition on ${\rm wt}_\tau(\chi_2)$ (so one could as well rewrite the middle character as just $\tau'(z)^{k_{\tau',3}}\chi_2 $). One has (when the $\eta_i$, or equivalently the $\chi_i$, come from characters in $\mathcal{T}_L$):
$$\jmath_{w,{\bf k}}(\eta)= \big(\tau(z)^{k_{\tau,1}}\tau'(z)^{k_{\tau',1}}\chi_1,\ \tau(z)^{k_{\tau,2}}\tau'(z)^{k_{\tau',2}}\chi_2,\ \tau(z)^{k_{\tau,3}}\tau'(z)^{k_{\tau',3}}\chi_3\big).$$}
\end{ex}

Let $\widetilde U_{\rm tri}^\square(\rbar):=U_{\rm tri}^\square(\rbar)\cap \widetilde X_{\rm tri}^\square(\rbar)$ (a union of connected components of $U_{\rm tri}^\square(\rbar)$), then $\widetilde U_{\rm tri}^\square(\rbar)\times_{\mathcal{W}^n_L}\mathcal{W}^n_{w,{\bf k},L}$ is reduced (since smooth over $\mathcal{W}^n_{w,{\bf k},L}$) and Zariski-open (but not necessarily Zariski-dense) in $(\widetilde X_{\rm tri}^\square(\rbar)\times_{\mathcal{W}^n_L}\mathcal{W}^n_{w,{\bf k},L})^\red$ where $(-)^{\red}$ means the associated reduced closed analytic subvariety. We denote by $\overline{\widetilde U_{\rm tri}^\square(\rbar)\times_{\mathcal{W}^n_L}\mathcal{W}^n_{w,{\bf k},L}}$ its Zariski-closure, so that we have a chain of Zariski-closed embeddings:
\begin{multline*}
\overline{\widetilde U_{\rm tri}^\square(\rbar)\times_{\mathcal{W}^n_L}\mathcal{W}^n_{w,{\bf k},L}}\subseteq (\widetilde X_{\rm tri}^\square(\rbar)\times_{\mathcal{W}^n_L}\mathcal{W}^n_{w,{\bf k},L})^\red\subseteq \widetilde X_{\rm tri}^\square(\rbar)\times_{\mathcal{W}^n_L}\mathcal{W}^n_{w,{\bf k},L}\subseteq \widetilde X_{\rm tri}^\square(\rbar)\\
\subseteq X_{\rm tri}^\square(\rbar)\subseteq \mathfrak{X}_{\rbar}^\square\times \mathcal{T}^n_L.
\end{multline*}

\begin{prop}\label{companion}
Assume Conjecture \ref{BHS}, then for $w\in W$ the automorphism $\jmath_{w,{\bf k}}:\mathfrak{X}_{\rbar}^\square\times \mathcal{T}^n_L\buildrel\sim\over\rightarrow\mathfrak{X}_{\rbar}^\square\times \mathcal{T}^n_L$ induces a closed embedding of reduced rigid analytic spaces over $L$:
$$\jmath_{w,{\bf k}}: \overline{\widetilde U_{\rm tri}^\square(\rbar)\times_{\mathcal{W}^n_L}\mathcal{W}^n_{w,{\bf k},L}}\hookrightarrow \widetilde X_{\rm tri}^\square(\rbar)\subseteq X_{\rm tri}^\square(\rbar).$$
\end{prop}
\begin{proof}
Since $\widetilde U_{\rm tri}^\square(\rbar)\times_{\mathcal{W}^n_L}\mathcal{W}^n_{w,{\bf k},L}$ is Zariski-dense in $\overline{\widetilde U_{\rm tri}^\square(\rbar)\times_{\mathcal{W}^n_L}\mathcal{W}^n_{w,{\bf k},L}}$, it is enough to prove $\jmath_{w,{\bf k}}(\widetilde U_{\rm tri}^\square(\rbar)\times_{\mathcal{W}^n_L}\mathcal{W}^n_{w,{\bf k},L})\subseteq \widetilde X_{\rm tri}^\square(\rbar)$, i.e. that any point $x'=(r',\delta')$ in $\widetilde U_{\rm tri}^\square(\rbar)$ with $\omega(x')\in \mathcal{W}^n_{w,{\bf k},L}$ is such that $\jmath_{w,{\bf k}}(x')$ is still in $\widetilde X_{\rm tri}^\square(\rbar)$.

Recall that by assumption:
\begin{equation}\label{rappelconj}
X_p(\rhobar) \buildrel{\substack{(\ref{union})\\ \sim}}\over \longrightarrow  \Xfrak_{\rhobar^p}\times \widetilde X_{\rm tri}^\square(\rhobar_p)\times \Ubb^g\subseteq \Xfrak_{\rhobar^p}\times (\Xfrak_{\rhobar_p}\times \widehat T_{p,L})\times \Ubb^g.
\end{equation}
Let $y'\in X_p(\rhobar)$ be any point such that its image in $\widetilde X_{\rm tri}^\square(\rhobar_p)$ by (\ref{rappelconj}) is $(x')_{v\in S_p}$. Write again $w$ for the element $(w)_{v\in S_p}\in \prod_{v\in S_p}(\prod_{F_{\tilde v}\hookrightarrow L}\Scal_n)$ (that is, for each $v$ we have the same element $w=(w_{\tau})_{\tau:\, K\hookrightarrow L}\in \prod_{\tau: K\hookrightarrow L}\Scal_n$), ${\bf k}$ for $({\bf k})_{v\in S_p}\in \prod_{v\in S_p}\Z^{\Hom(K,L)}$ (ibid.), $\jmath_{w,{\bf k}}$ for the automorphism $(\jmath_{w,{\bf k}})_{v\in S_p}$ of $\widehat T_{p,L}\cong \prod_{v\in S_p}\widehat T_{v,L}\cong \prod_{v\in S_p}\mathcal{T}^n_L$ and (again) $\jmath_{w,{\bf k}}$ for the automorphism $\id \times (\id \times \jmath_{w,{\bf k}})\times \id$ of $\Xfrak_{\rhobar^p}\times (\Xfrak_{\rhobar_p}\times \widehat T_{p,L})\times \Ubb^g$. Then it is enough to prove that $\jmath_{w,{\bf k}}(y')\in X_p(\rhobar)$ (via (\ref{rappelconj})). Writing $y'=(\mathfrak{m}',\epsilon')\in X_p(\rhobar)\subseteq \Xfrak_\infty\times \widehat T_{p,L}$ where $\mathfrak{m}'\subset R_\infty[1/p]$ is the maximal ideal corresponding to the projection of $y'$ in $\Xfrak_\infty$ and $\epsilon'=(\imath_v(\delta'))_{v\in S_p}$, we have $\Hom_{T_p}(\epsilon',J_{B_p}(\Pi_\infty^{R_\infty-\rm an}[\mathfrak{m}']\otimes_{k(\mathfrak{m}')}k(y')))\ne 0$ (see (\ref{pointpatching})) and we have to prove (note that $\jmath_{w,{\bf k}}\circ\imath_v^{-1}=\imath_v^{-1}\circ\jmath_{w,{\bf k}}$ on $\widehat T_{v,L}$ and that $k(y')=k(\jmath_{w,{\bf k}}(y'))$):
\begin{equation}\label{nonzero}
\Hom_{T_p}\big(\jmath_{w,{\bf k}}(\epsilon'),J_{B_p}(\Pi_\infty^{R_\infty-\rm an}[\mathfrak{m}']\otimes_{k(\mathfrak{m}')}k(y'))\big)\ne 0.
\end{equation}

From Theorem \ref{intercompanion} below, it is enough to prove $\epsilon' \uparrow \jmath_{w,{\bf k}}(\epsilon')$ in the sense of Definition \ref{stronglink} below. Since $\epsilon'\jmath_{w,{\bf k}}(\epsilon')^{-1}$ is clearly an algebraic character of $T_p$ by definition of $\jmath_{w,{\bf k}}$, it is enough to prove $\imath_v(\delta') \uparrow_{{\mathfrak t}_v} \jmath_{w,{\bf k}}(\imath_v(\delta'))$ (see \S\ref{jw} for the notation) for one, or equivalently all here, $v\in S_p$. From (\ref{equations}), we see that we can write:
$$\delta'=(z^{{\bf k}_{w^{-1}(1)}}\chi_1,\dots,z^{{\bf k}_{w^{-1}(n)}}\chi_n)\ \ {\rm and}\ \ \jmath_{w,{\bf k}}(\imath_v(\delta'))=(z^{{\bf k}_1}\chi_1,\dots,z^{{\bf k}_n}\chi_n)$$
where ${\rm wt}_\tau(\chi_i)={\rm wt}_\tau(\chi_{w_\tau(i)})$ for $1\leq i\leq n$ and $\tau:\, K=F_{\tilde v}\hookrightarrow L$ (compare Example \ref{example}). As we only care about the ${\mathfrak t}_{v,L}$-action, setting $s_{\tau,i}:={\rm wt}_\tau(\chi_i)\in L$ and using usual additive notation, we can write $\imath_v(\delta')\vert_{{\mathfrak t}_{v,L}}=(\imath_v(\delta')_\tau)_{\tau:\, F_{\tilde v}\hookrightarrow L}$ and $\jmath_{w,{\bf k}}(\imath_v(\delta'))\vert_{{\mathfrak t}_{v,L}}=(\jmath_{w,{\bf k}}(\imath_v(\delta'))_\tau)_{\tau:\, F_{\tilde v}\hookrightarrow L}$ with:
\begin{eqnarray*}
\imath_v(\delta')_\tau&=&(k_{\tau,w_{\tau}^{-1}(1)}+s_{\tau,1},k_{\tau,w_{\tau}^{-1}(2)}+s_{\tau,2}+1,\dots,k_{\tau,w_{\tau}^{-1}(n)}+s_{\tau,n}+n-1)\\
\jmath_{w,{\bf k}}(\imath_v(\delta'))_\tau&=&(k_{\tau,1}+s_{\tau,1},k_{\tau,2}+s_{\tau,2}+1,\dots,k_{\tau,n}+s_{\tau,n}+n-1)
\end{eqnarray*}
(see the beginning of \S\ref{jw}). Since $s_{\tau,i}=s_{\tau,w_\tau^{-1}(i)}$ for all $i,\tau$, we can rewrite:
$$\imath_v(\delta')_\tau=(k_{\tau,w_{\tau}^{-1}(1)}+s_{\tau,w_{\tau}^{-1}(1)},k_{\tau,w_{\tau}^{-1}(2)}+s_{\tau,w_{\tau}^{-1}(2)}+1,\dots,k_{\tau,w_{\tau}^{-1}(n)}+s_{\tau,w_{\tau}^{-1}(n)}+n-1)$$
hence we have $\imath_v(\delta')_\tau=w_\tau\cdot \jmath_{w,{\bf k}}(\imath_v(\delta'))_\tau$ for the ``dot action'' $\cdot$ with respect to the upper triangular matrices in $\GL_{n,F_{\tilde v}}\times_{F_{\tilde v},\tau}L$ (see \cite[\S1.8]{HumBGG}). Let us write the permutation $w_\tau$ on $\{1,\dots,n\}$ as a product of commuting cycles $c_1\circ\cdots \circ c_m$ with pairwise disjoint support ${\rm supp}(c_i)\subseteq \{1,\dots,n\}$. Let us denote by ${\mathcal S}_{n,i}\subseteq {\mathcal S}_{n}$ the subgroup of permutations which fixes the elements in $\{1,\dots,n\}$ {\it not} in ${\rm supp}(c_i)$ and set ${\mathcal S}_{n,w_\tau}:= \prod_{i=1}^m{\mathcal S}_{n,i}\subseteq {\mathcal S}_n$. Then, arguing in each ${\rm supp}(c_i)$, it is not difficult to see that one can write $w_\tau$ as a product:
$$w_\tau = s_{\alpha_d}s_{\alpha_{d-1}}\cdots s_{\alpha_1}$$
where the $\alpha_i$ are (not necessarily simple) roots of the upper triangular matrices in $\GL_{n,F_{\tilde v}}\times_{F_{\tilde v},\tau}L$, the associated reflections $s_{\alpha_i}$ are in ${\mathcal S}_{n,w_\tau}$ and where $s_{\alpha_i+1}s_{\alpha_{i}}\cdots s_{\alpha_1}> s_{\alpha_i}\cdots s_{\alpha_1}$ for the Bruhat order in ${\mathcal S}_n$ ($1\leq i\leq n-1$). By an argument analogous {\it mutatis mutandis} to the one in \cite[\S5.2]{HumBGG}, it then follows from the above assumptions (in particular $s_{\tau,i}=s_{\tau,w_\tau^{-1}(i)}$ for all $i$) that we have for $1\leq i\leq n-1$ with obvious notation:
$$(s_{\alpha_i+1}\cdots s_{\alpha_1})\cdot \jmath_{w,{\bf k}}(\imath_v(\delta'))_\tau \leq (s_{\alpha_i}\cdots s_{\alpha_1})\cdot \jmath_{w,{\bf k}}(\imath_v(\delta'))_\tau.$$
By definition this implies that $w_\tau\cdot\jmath_{w,{\bf k}}(\imath_v(\delta'))_\tau$ is {\it strongly linked} to $\jmath_{w,{\bf k}}(\imath_v(\delta'))_\tau$ (\cite[\S5.1]{HumBGG}). As this holds for all $\tau$, we have $\imath_v(\delta') \uparrow_{{\mathfrak t}_v} \jmath_{w,{\bf k}}(\imath_v(\delta'))$.
\end{proof}

\begin{rema}
{\rm It would be very interesting to find a purely local proof of the local statement of Proposition \ref{companion} without assuming Conjecture \ref{BHS}.}
\end{rema}

\subsection{Companion points on the patched eigenvariety}\label{jw}

We prove that the existence of certain points on the patched eigenvariety $X_p(\rhobar)$ implies the existence of others (Theorem \ref{intercompanion}). This result is crucially used in the proof of Proposition \ref{companion} above.

We use the notation of \S\ref{globalpart}. We denote by ${\mathfrak g}$ (resp. ${\mathfrak b}$, resp. ${\mathfrak t}$)  the $\Q_p$-Lie algebra of $G_p$ (resp. $B_p$, resp. $T_p$). We also denote by ${\mathfrak n}$ (resp. $\overline{\mathfrak n}$) the $\Q_p$-Lie algebra of the inverse image $N_p$ in $B_p$ (resp. $\overline N_p$ in $\overline B_p$) of the subgroup of upper (resp. lower) unipotent matrices of $\prod_{v\in S_p}\GL_n(F_{\tilde v})$. We add an index $L$ for the $L$-Lie algebras obtained by scalar extension $\cdot \otimes_{\Q_p}L$ (e.g. ${\mathfrak g}_L$, etc.) and we denote by $U(\cdot)$ the corresponding enveloping algebras.

For $v\in S_p$ we denote by ${\mathfrak t}_v$ the $\Q_p$-Lie algebra of the torus $T_v$, so that ${\mathfrak t}=\prod_{v\in S_p}{\mathfrak t}_v$. Recall that ${\mathfrak t}_v$ is an $F_{\tilde v}$-vector space, and thus ${\mathfrak t}_{v,L}={\mathfrak t}_v\otimes_{\Q_p}L\cong \prod_{\tau:\, F_{\tilde v}\hookrightarrow L}{\mathfrak t}_v\otimes_{F_{\tilde v},\tau}L$. We can see any $\eta=(\eta_v)_{v\in S_p}=(\eta_{v,1},\dots,\eta_{v,n})_{v\in S_p}\in \widehat T_{p,L}$ as an $L$-valued additive character of ${\mathfrak t}$, and thus of ${\mathfrak t}_{L}$ by $L$-linearity, via the usual derivative action $({\mathfrak z}_{v,1},\dots,{\mathfrak z}_{v,n})_{v\in S_p}\mapsto \sum_{v\in S_p}\sum_{i=1}^n\frac{d}{dt}\eta_{v,i}(\exp(t{\mathfrak z}_{v,i}))\vert_{t=0}$. Recall that the character ${\mathfrak z}_{v,i}\in F_{\tilde v}\mapsto \frac{d}{dt}\eta_{v,i}(\exp(t{\mathfrak z}_{v,i}))\vert_{t=0}$ is nothing else than $\sum_{\tau:\, F_{\tilde v}\hookrightarrow L}\tau({\mathfrak z}_{v,i}){\rm wt}_\tau(\eta_{v,i})\in L$.

In what follows we use notation and definitions from \cite{STdist} concerning $L$-Banach representations of $p$-adic Lie groups and their locally $\Q_p$-analytic vectors. If $\Pi$ is an admissible continuous representation of $G_p$ on a $L$-Banach space we denote by $\Pi^{\rm an}\subseteq \Pi$ its invariant subspace of locally $\Q_p$-analytic vectors.

\begin{lemm}\label{injective}
Let $\Pi$ be an admissible continuous representation of $G_p$ on a $L$-Banach space and assume that the continuous dual $\Pi'$ is a finite projective $\mathcal{O}_L\dbl K_p\dbr[1/p]$-module. Let $\lambda$, $\mu$ be $L$-valued characters of ${\mathfrak t}_L$ that we see as $L$-valued characters of ${\mathfrak b}_L$ by sending ${\mathfrak n}_L$ to $0$. If $U(\mathfrak{g}_{L})\otimes_{U(\mathfrak{b}_L)}\mu\hookrightarrow U(\mathfrak{g}_{L})\otimes_{U(\mathfrak{b}_L)}\lambda$ is an injection of $U(\mathfrak{g}_{L})$-modules, then the map:
$$\Hom_{U(\mathfrak{g}_{L})}\big(U(\mathfrak{g}_{L})\otimes_{U(\mathfrak{b}_L)}\lambda,\Pi^{\rm an}\big)\longrightarrow\Hom_{U(\mathfrak{g}_{L})}\big(U(\mathfrak{g}_{L})\otimes_{U(\mathfrak{b}_L)}\mu,\Pi^{\rm an}\big)$$
induced by functoriality is surjective.
\end{lemm}
\begin{proof}
We have as in \cite[Prop.6.5]{STdist} a $K_p$-equivariant isomorphism:
\begin{equation}\label{limr}
\Pi^{\rm an}\cong {\lim_{\stackrel{r\rightarrow 1}{r<1}}}\ \Pi_r
\end{equation}
where each $\Pi_r\subseteq \Pi^{\rm an}$ is a Banach space over $L$ endowed with an admissible locally $\Q_p$-analytic action of $K_p$. In particular each $\Pi_r$ is stable under $U(\mathfrak{g}_L)$ in $\Pi^{\rm an}$. If
$f:U(\mathfrak{g}_{L})\otimes_{U(\mathfrak{b}_L)}\lambda\rightarrow\Pi^{\rm an}$ is a $U(\mathfrak{g}_L)$-equivariant morphism, the source, being of finite type over $U(\mathfrak{g}_{L})$, factors through some $\Pi_r$ by (\ref{limr}). Moreover the action of $U(\mathfrak{g}_{L})$ on $\Pi_r$ extends to an action of the $L$-Banach algebra $U_r(\mathfrak{g}_{L})$ which is the topological closure of $U(\mathfrak{g}_{L})$ in the completed distribution algebra $D_r(K_p,L)$ (see \cite[\S5]{STdist}). Consequently $f$ extends to a $U_r(\mathfrak{g}_{L})$-equivariant morphism:
$$f_r:\, U_r(\mathfrak{g}_{L})\otimes_{U(\mathfrak{g}_{L})}(U(\mathfrak{g}_{L})\otimes_{U(\mathfrak{b}_L)}\lambda)\cong
U_r(\mathfrak{g}_{L})\otimes_{U_r(\mathfrak{b}_L)} \lambda\longrightarrow\Pi_r$$
where $U_r(\mathfrak{b}_L)$ is the closure of $U(\mathfrak{b}_L)$ in $D_r(K_p,L)$ and the first isomorphism follows from
$\lambda\buildrel\sim\over\rightarrow U_r(\mathfrak{b}_L)\otimes_{U(\mathfrak{b}_L)}\lambda$ (as $\lambda$ is both finite dimensional with dense image). We deduce from \cite[Prop.3.4.8]{OrlikStrauch} (applied with $w=1$) that the injection $U(\mathfrak{g}_{L})\otimes_{U(\mathfrak{b}_L)}\lambda\hookrightarrow U(\mathfrak{g}_{L})\otimes_{U(\mathfrak{b}_L)}\mu$ extends to an injection
of $U_r(\mathfrak{g}_{L})$-modules $U_r(\mathfrak{g}_{L})\otimes_{U_r(\mathfrak{b}_L)}\lambda\hookrightarrow U_r(\mathfrak{g}_{L})\otimes_{U_r(\mathfrak{b}_L)}\mu$. Moreover, as $U_r(\mathfrak{g}_{L})\otimes_{U_r(\mathfrak{b}_L)}\lambda$ and $U_r(\mathfrak{g}_{L})\otimes_{U_r(\mathfrak{b}_L)}\mu$ are $U_r(\mathfrak{g}_{L})$-modules of finite type, they have a unique topology of Banach module over $U_r(\mathfrak{g}_{L})$ and every $U_r(\mathfrak{g}_{L})$-linear map of one of them into $\Pi_r$ is automatically continuous (see \cite[Prop.2.1]{STdist}). We deduce from all this isomorphisms:
\begin{eqnarray*}
\Hom_{U(\mathfrak{g}_{L})}\big(U(\mathfrak{g}_{L})\otimes_{U(\mathfrak{b}_L)}\lambda,\Pi^{\rm an}\big)&\buildrel\sim\over\longrightarrow &\varinjlim_r\Hom_{U(\mathfrak{g}_{L})}\big(U(\mathfrak{g}_{L})\otimes_{U(\mathfrak{b}_L)}\lambda,\Pi_r\big)\\
&\buildrel\sim\over\longrightarrow &\varinjlim_r\Hom_{U_r(\mathfrak{g}_{L})-{\rm cont}}\big(U_r(\mathfrak{g}_{L})\otimes_{U_r(\mathfrak{b}_L)}\lambda,\Pi_r\big)
\end{eqnarray*}
where $\Hom_{U_r(\mathfrak{g}_{L})-{\rm cont}}$ means continuous homomorphisms of $U_r(\mathfrak{g}_{L})$-Banach modules, and likewise with $\mu$ instead of $\lambda$. By exactitude of $\varinjlim_r$, we see that it is enough to prove that $\Pi_r$ is an injective object (with respect to injections which have closed image) in the category of $U_r(\mathfrak{g}_{L})$-Banach modules with continuous maps.

By assumption the dual $\Pi'$ is a projective module of finite type over $\mathcal{O}_L\dbl K_p\dbr[1/p]$, hence a direct summand of $\mathcal{O}_L\dbl K_p\dbr[1/p]^{\oplus s}$ for some $s>0$. From the proof of \cite[Prop.6.5]{STdist} together with \cite[Th.7.1(iii)]{STdist}, we also know that $\Pi_r$ is the continuous dual of the $D_r(K_p,L)$-Banach module:
$$\Pi_r':=D_r(K_p,L)\otimes_{\mathcal{O}_L\dbl K_p\dbr[1/p]}\Pi'.$$
We get that the $D_r(K_p,L)$-module $\Pi_r'$ is a direct summand of $D_r(K_p,L)^{\oplus s}$. Now it easily follows from the results in \cite[\S1.4]{Ko} that $D_r(K_p,L)$ is itself a free $U_r(\mathfrak{g}_{L})$-module of finite rank. Dualizing, we finally obtain that there is a finite dimensional $L$-vector space $W$ such that the left $U_r(\mathfrak{g}_{L})$-Banach module $\Pi_r$ is a direct factor of the left $U_r(\mathfrak{g}_{L})$-Banach module ${\rm Hom}_{\rm cont}(U_r(\mathfrak{g}_{L})\otimes_L W,L)$ (which is seen as a left
$U_r(\mathfrak{g}_{L})$-module via the automorphism on $U_r(\mathfrak{g}_{L})$ extending the multiplication by $-1$ on $\mathfrak{g}_{L}$). Since direct summands and finite sums of injective modules are still injective, it is enough to prove the injectivity of ${\rm Hom}_{\rm cont}(U_r(\mathfrak{g}_{L}),L)$ in the category of $U_r(\mathfrak{g}_{L})$-Banach modules with continuous maps.

If $V$ is any $U_r(\mathfrak{g}_L)$-Banach module, it is not difficult to see that there is a canonical isomorphism of Banach spaces over $L$:
\begin{equation}\label{dualdual}
{\Hom}_{U_r(\mathfrak{g}_{L})-{\rm cont}}\big(V,{\Hom}_{\rm cont}(U_r(\mathfrak{g}_L),L)\big)\buildrel\sim\over\longrightarrow {\Hom}_{\rm cont}(V,L)
\end{equation}
so that the required injectivity property is a consequence of the Hahn-Banach Theorem (see for example \cite[Prop.9.2]{Schneidernfa}).
\end{proof}

We go on with two technical lemmas which require more notation. Fix a compact open uniform normal pro-$p$ subgroup $H_p$ of $K_p$ such that $H_p=(\overline N_p\cap H_p)(T_p\cap H_p)(N_p\cap H_p)$. For example $H_p$ can be chosen of the form $\prod_{v\in S_p}H_{v}$ where $H_v$ is the inverse image in $K_v$ of matrices of $\GL_n(\mathcal{O}_{F_{\tilde{v}}})$ congruent to $1$ mod $p^m$ for $m$ big enough. Let $N_0:=N_p\cap H_p$, $T_{p,0}:=T_p\cap H_p$, $\overline{N}_0:=\overline{N}_p\cap H_p$ (which are still uniform pro-p-groups) and $T_p^+:=\{ t\in T_p\ {\rm such\ that}\ tN_0t^{-1}\subseteq N_0\}$ (which is a multiplicative monoid in $T_p$). We also fix $z\in T_p^+$ such that $zN_0z^{-1}\subseteq N_0^p$ and we assume moreover $z^{-1}H_pz\subseteq K_p$ so that the elements of $z^{-1}H_pz$ normalize $H_p$ (as $H_p$ is normal in $K_p$). Note that such a $z$ always exists, for instance take $z$ such that $zN_0z^{-1}\subseteq N_0^p$, choose $r$ such that $H_p^{p^r}\subseteq zK_pz^{-1}$ and replace $H_p$ by $H_p^{p^r}$: with this new choice we still have $zN_0z^{-1}\subseteq N_0^p$.

For any uniform pro-$p$-group $H$ we denote by $\mathcal{C}(H,L)$ the Banach space of continuous $L$-valued functions on $H$ and, if $h\geq 1$, by $\mathcal{C}^{(h)}(H,L)$ the Banach space of $h$-analytic $L$-valued functions on $H$ defined in \cite[\S0.3]{CD}. We have $\mathcal{C}(H_p,L)\cong \mathcal{C}(\overline{N}_0,L)\widehat{\otimes}_L\mathcal{C}(T_{p,0},L)\widehat{\otimes}_L\mathcal{C}(N_0,L)$ and likewise with $\mathcal{C}^{(h)}(\cdot,L)$.

\begin{lemm}\label{restriction}
Let $f\in\mathcal{C}(H_p,L)$ such that for each left coset $(zH_pz^{-1}\cap H_p)n\subset H_p$, there exists $f_n\in\mathcal{C}^{(h)}(H_p,L)$ such that $f(gn)=f_n(z^{-1}gz)$ for $g\in zH_pz^{-1}\cap H_p$. Then we have:
$$f\in\mathcal{C}^{(h-1)}(\overline{N}_0,L)\widehat{\otimes}_L\mathcal{C}^{(h)}(T_{p,0},L)\widehat{\otimes}_L\mathcal{C}(N_0,L).$$
\end{lemm}
\begin{proof}
Representatives of the quotient $(zH_pz^{-1}\cap H_p)\backslash H_p$ can be chosen in $N_0$, whence the above notation $n$ (do not confuse with the $n$ of $\GL_n$!). Restricting $f$ to the left coset $(zH_pz^{-1}\cap H_p)n$ for some $n\in N_0$ and translating on the right by $n$ we can assume that the support of $f$ is contained in $zH_pz^{-1}\cap H_p$. Then if $g\in zH_pz^{-1}\cap H_p$, we have by assumption $f(g)=F(z^{-1}gz)$ for some $F\in\mathcal{C}^{(h)}(H_p,L)$. Consequently $f|_{zH_pz^{-1}\cap H_p}$ can be extended to an $h$-analytic function on $zH_pz^{-1}$ and $f$ can be extended (by $0$) on $zH_pz^{-1}N_0=z\overline N_0z^{-1} T_{p,0} N_0$ as an element of:
$$\mathcal{C}^{(h)}(z\overline{N}_0z^{-1},L)\widehat{\otimes}_L\mathcal{C}^{(h)}(T_{p,0},L)\widehat{\otimes}_L\mathcal{C}(N_0,L).$$
We deduce that $f$ is in the image of the restriction map (note that $zN_0z^{-1}\subseteq N_0$ implies $\overline{N}_0\subseteq z\overline{N_0}z^{-1}$):
$$\mathcal{C}^{(h)}(z\overline{N}_0z^{-1},L)\widehat{\otimes}_L\mathcal{C}^{(h)}(T_{p,0},L)\widehat{\otimes}_L\mathcal{C}(N_0,L)\longrightarrow\mathcal{C}^{(h)}(\overline{N}_0,L)\widehat{\otimes}_L\mathcal{C}^{(h)}(T_{p,0},L)\widehat{\otimes}_L\mathcal{C}(N_0,L).$$
Now the stronger condition $zN_0z^{-1}\subseteq N_0^p$ implies $\overline{N}_0\subseteq z\overline{N_0}^pz^{-1}=(z\overline{N}_0z^{-1})^p$. But by \cite[Rem.IV.12]{CD} the restriction to $(z\overline{N}_0z^{-1})^p$ (and {\it a fortiori} to $\overline N_0$) of an $h$-analytic function on $z\overline{N}_0z^{-1}$ is $(h-1)$-analytic and we can conclude.
\end{proof}

If $\Pi$ is an admissible continuous representation of $G_p$ on a $L$-Banach space and if $h\geq1$, we denote by $\Pi_{H_p}^{(h)}$ the $H_p$-invariant Banach subspace of $\Pi^{\rm an}$ defined in \cite[\S0.3]{CD}. If $V$ is any (left) $U(\mathfrak{t}_L)$-module over $L$ and $\lambda:\mathfrak{t}_L\rightarrow L$ is a character, we let $V_\lambda$ be the $L$-subvector space of $V$ on which $\mathfrak{t}_L$ acts via the multiplication by $\lambda$. Recall that if $V$ is any $L[G_p]$-module, the monoid $T_p^+$ acts on $V^{N_0}$ via $v\longmapsto t\cdot v:=\delta_{B_p}(t)\sum_{n_0\in N_0/tN_0t^{-1}}n_0tv$ ($v\in V^{N_0}$, $t\in T_p^+$, see \S\ref{classic} for $\delta_{B_p}$). This $T_p^+$-action respects the subspace $(\Pi_\lambda^{\rm an})^{N_0}$ of $(\Pi^{\rm an})^{N_0}$ (use that $tN_0t^{-1}=N_0$ for $t\in T_p^0$).

We don't claim any originality on the following lemma which is a variant of classical results (see e.g. \cite{EmertonJacquetI}), however we couldn't find its exact statement in the literature.

\begin{lemm}\label{compact}
Let $\Pi$ be an admissible continuous representation of $G_p$ on a $L$-Banach space, $\lambda$ an $L$-valued character of $\mathfrak{t}_L$ and $h\geq1$. Then the action of $z$ on $(\Pi^{\rm an})^{N_0}$ preserves the subspace $(\Pi^{(h)}_{H_p})^{N_0}_\lambda=(\Pi^{(h)}_{H_p})^{N_0}\cap \Pi^{\rm an}_\lambda$ and is a {\rm compact} operator on this subspace.
\end{lemm}
\begin{proof}
Let $l_1,\dots,l_s$ be a system of generators of the continous dual $\Pi'$ as a module over the algebra $\mathcal{O}_L\dbl H_p\dbr [1/p]$. Define a closed embedding of $\Pi$ into $\mathcal{C}(H_p,L)^{\oplus s}$ via the map $v\mapsto (g\mapsto l_i(gv))_{1\leq i\leq s}$. This embedding is $H_p$-equivariant for the left action of $H_p$ on $\mathcal{C}(H_p,L)$ by right translation on functions. By \cite[Prop.IV.5]{CD}, we have $\Pi_{H_p}^{(h)}=\Pi\cap \mathcal{C}^{(h)}(H_p,L)^{\oplus s}$. If $v\in\Pi_{H_p}^{(h)}$, $n\in N_0$ and $g\in H_p\cap z H_pz^{-1}$, we have $l_i(gnzv)=l_i(z(z^{-1}gz)(z^{-1}nz)v)$. Let $v\in\Pi_{H_p}^{(h)}$ and $n\in N_0$. As $z^{-1}N_0z$ normalizes $H_p$ (by the choice of $z$) we have $w:=(z^{-1}nz)v\in\Pi_{H_p}^{(h)}$ (see \cite[Prop.IV.16]{CD}). As $l_i(z\cdot)$ is a continuous linear form on $\Pi$, using \cite[Thm.IV.6(i)]{CD} the function $f_n:H_p\rightarrow L$, $g\mapsto f_n(g):=l_i(zgw)$ is in $\mathcal{C}^{(h)}(H_p,L)$ and $l_i(gnzv)=f_n(z^{-1}gz)$ for $g\in zH_pz^{-1}\cap H_p$. We deduce from Lemma \ref{restriction} applied to the functions $f:H_p\rightarrow L$, $g\mapsto l_i(gzv)$ for $1\leq i\leq s$ that:
\begin{equation}\label{incl(h)}
z\Pi_{H_p}^{(h)}\subset \big(\mathcal{C}^{(h-1)}(\overline{N}_0,L)\widehat{\otimes}_L\mathcal{C}^{(h)}(T_{p,0},L)\widehat{\otimes}_L\mathcal{C}(N_0,L)\big)^{\oplus s}.
\end{equation}
Let $v\in(\Pi_{H_p}^{(h)})^{N_0}$, the space on the right hand side of (\ref{incl(h)}) being stable under $N_0$ (acting by right translation on functions), it still contains $z\cdot v=\sum_{n_0\in N_0/tN_0t^{-1}}n_0zv$. Since $z\cdot v$ is fixed under $N_0$, we deduce:
$$z\cdot v\ \in\ \big(\mathcal{C}^{(h-1)}(\overline{N}_0,L)\widehat{\otimes}_L\mathcal{C}^{(h)}(T_{p,0},L)\big)^{\oplus s}\subseteq \ \mathcal{C}^{(h)}(H_p,L)^{\oplus s}.$$
In particular $z\cdot(\Pi_{H_p}^{(h)})^{N_0}_\lambda \subseteq (\Pi^{\rm an}_\lambda)^{N_0}\cap \mathcal{C}^{(h)}(H_p,L)^{\oplus s}=(\Pi_{H_p}^{(h)})^{N_0}_\lambda$ which shows the first statement. We also deduce:
$$z\cdot(\Pi_{H_p}^{(h)})^{N_0}_\lambda\subseteq\big(\mathcal{C}^{(h-1)}(\overline{N}_0,L)\widehat{\otimes}_L\mathcal{C}^{(h)}(T_{p,0},L)\big)_{\!\lambda}^{\oplus s}\cong \mathcal{C}^{(h-1)}(\overline{N}_0,L)\widehat{\otimes}_L\big(\mathcal{C}^{(h)}(T_{p,0},L)_\lambda^{\oplus s}\big).$$
But by \cite[Prop.~IV.13.(i)]{CD}, we have $\mathcal{C}^{(h)}(T_{p,0},L)'\simeq D_r(T_{p,0},L)$ for $r=p^{-1/p^h}$ where $D_r(T_{p,0},L)$ is as in \cite[\S4]{STdist}. Let $U_r(\mathfrak{t}_L)$ be the closure of $U(\mathfrak{t}_L)$ in $D_r(T_{p,0},L)$, then (as in the proof of Lemma \ref{injective}) $D_r(T_{p,0},L)$ is a finite free $U_r(\mathfrak{t}_L)$-module (\cite[\S1.4]{Ko}). Using $\lambda\buildrel\sim\over\rightarrow U_r(\mathfrak{t}_L)\otimes_{U(\mathfrak{t}_L)}\lambda$, it follows that $(\mathcal{C}^{(h)}(T_{p,0},L)_\lambda)'$, and hence $\mathcal{C}^{(h)}(T_{p,0},L)_\lambda^{\oplus s}$, are finite dimensional $L$-vector spaces. We denote the latter by $W_\lambda$.

We thus have $z\cdot(\Pi_{H_p}^{(h)})^{N_0}_\lambda\subseteq \mathcal{C}^{(h-1)}(\overline{N}_0,L)\otimes_L W_\lambda$: the endomorphism induced by $z$ on $(\Pi_{H_p}^{(h)})^{N_0}_\lambda$ factors through the subspace $(\Pi_{H_p}^{(h)})^{N_0}_\lambda\cap\big(\mathcal{C}^{(h-1)}(\overline{N}_0,L)\otimes_L W_\lambda\big)$. As the inclusion of $\mathcal{C}^{(h-1)}(\overline{N}_0,L)$ into $\mathcal{C}^{(h)}(\overline{N}_0,L)$ is compact and $W_\lambda$ is finite dimensional over $L$, the inclusion of $(\Pi_{H_p}^{(h)})^{N_0}_\lambda\cap\big(\mathcal{C}^{(h-1)}(\overline{N}_0,L)\otimes_L W\big)$ into $(\Pi_{H_p}^{(h)})^{N_0}_\lambda$ is compact, which proves the result.
\end{proof}

If $\delta_v,\epsilon_v\in \widehat T_{v,L}$, we write $\epsilon_v \uparrow_{{\mathfrak t}_v} \delta_v$ if, seeing $\delta_v,\epsilon_v$ as $U({\mathfrak t}_{v,L})$-modules, we have $\epsilon_v \uparrow \delta_v$ in the sense of \cite[\S5.1]{HumBGG} with respect to the roots of the upper triangular matrices in $(\Res_{F_{\tilde v}/\Q_p}\GL_{n,F_{\tilde v}})_L$. Likewise if $\delta,\epsilon\in \widehat T_{p,L}$, we write $\epsilon \uparrow_{{\mathfrak t}} \delta$ if, seeing $\delta,\epsilon$ as $U({\mathfrak t}_{L})$-modules, we have $\epsilon \uparrow \delta$ in the sense of \cite[\S5.1]{HumBGG} with respect to the roots of the upper triangular matrices in $\prod_{v\in S_p}(\Res_{F_{\tilde v}/\Q_p}\GL_{n,F_{\tilde v}})_L$. Thus writing $\delta=(\delta_v)_{v\in S_p},\epsilon=(\epsilon_v)_{v\in S_p}$, we have $\epsilon \uparrow_{{\mathfrak t}} \delta$ if and only if $\epsilon_v \uparrow_{{\mathfrak t}_v} \delta_v$ for all $v\in S_p$.

\begin{defi}\label{stronglink}
Let $\delta,\epsilon\in \widehat T_{p,L}$, we write $\epsilon \uparrow \delta$ if $\epsilon \uparrow_{{\mathfrak t}} \delta$ and if $\epsilon\delta^{-1}$ is an algebraic character of $T_p$, i.e. $\epsilon\delta^{-1}=\delta_{\lambda}$ for some ${\lambda}=({\lambda}_v)_{v\in S_p}\in\prod_{v\in S_p}(\mathbb{Z}^n)^{\Hom(F_{\tilde v},L)}$.
\end{defi}

We can now prove the main theorem of this section.

\begin{theo}\label{intercompanion}
Let ${\mathfrak m}\subseteq R_\infty[1/p]$ be a maximal ideal, $\delta,\epsilon\in \widehat T_{p,L}$ such that $\epsilon \uparrow \delta$ and $L'$ a finite extension of $L$ containing the residue fields $k(\delta)=k(\epsilon)$ and $k({\mathfrak m})$. Then we have:
$$\Hom_{T_p}\big(\epsilon,J_{B_p}(\Pi_\infty^{R_\infty-\rm an}[\mathfrak{m}]\otimes_{k({\mathfrak m})}L')\big)\ne 0\Longrightarrow \Hom_{T_p}\big(\delta,J_{B_p}(\Pi_\infty^{R_\infty-\rm an}[\mathfrak{m}]\otimes_{k({\mathfrak m})}L')\big)\ne 0.$$
\end{theo}
\begin{proof}
We assume first $k(\delta)=k(\epsilon)=L$ and $L'=k({\mathfrak m})$, so that we can forget about $L'$.
Let $\Pi$ be a locally $\Q_p$-analytic representation of $B_p$ over $L$. The subspace $\Pi^{\mathfrak{n}_L}$ of vectors killed by $\mathfrak{n}_L$ is a smooth representation of the group $N_0$ and we denote by $\pi_{N_0}:\Pi^{\mathfrak{n}_L}\twoheadrightarrow \Pi^{N_0}\subseteq \Pi^{\mathfrak{n}_L}$ the unique $N_0$-equivariant projection on its subspace $\Pi^{N_0}$.  It is preserved by the action of $T_p$ inside $\Pi$, hence also by the action of $\mathfrak{t}_L$ and one easily checks that:
\begin{equation}\label{commutt}
\pi_{N_0}\circ \mathfrak{x}=\mathfrak{x}\circ \pi_{N_0}\ \ \ \ (\mathfrak{x}\in \mathfrak{t}_L)
\end{equation}
(use $tN_0t^{-1}=N_0$ for $t\in T_p^0$). The subspace $\Pi_{\lambda}^{\mathfrak{n}_L}:=\Pi_\lambda\cap \Pi^{\mathfrak{n}_L}\subseteq \Pi^{\mathfrak{n}_L}$ is still preserved by $T_p$ and by (\ref{commutt}) the projection $\pi_{N_0}$ sends $\Pi_{\lambda}^{\mathfrak{n}_L}$ onto $\Pi_{\lambda}^{N_0}:=\Pi^{N_0}\cap \Pi_{\lambda}^{\mathfrak{n}_L}\subseteq \Pi_{\lambda}^{\mathfrak{n}_L}$. We have $t\cdot v=\pi_{N_0}(tv)$ for $t\in T_p^+$, $v\in \Pi^{N_0}_\lambda$ and in the rest of the proof we view $\Pi^{N_0}_\lambda$ as an $L[T_p^+]$-module via this monoid action.

The locally $\Q_p$-analytic character $\delta:T_p\rightarrow L^\times$ determines a surjection of $L$-algebras $L[T_p]\twoheadrightarrow L$ and we denote its kernel by $\mathfrak{m}_\delta$ (a maximal ideal of the $L$-algebra $L[T_p]$). We still write $\mathfrak{m}_\delta$ for its intersection with $L[T_p^+]$, which is then a maximal ideal of $L[T_p^+]$. Let $\lambda:\mathfrak{t}_L\rightarrow L$ be the derivative of $\delta$, arguing as in \cite[Prop.3.2.12]{EmertonJacquetI} we get for $s\geq 1$:
\begin{equation}\label{jm}
J_{B_p}(\Pi)[\mathfrak{m}_\delta^s]\cong \Pi^{N_0}[\mathfrak{m}_{\delta}^s]\cong \Pi_\lambda^{N_0}[\mathfrak{m}_{\delta}^s],
\end{equation}
(in particular $\Hom_{T_p}(\delta,J_{B_p}(\Pi))\cong \Pi^{N_0}[\mathfrak{m}_{\delta}]\cong \Pi_\lambda^{N_0}[\mathfrak{m}_{\delta}]$). Likewise we have $J_{B_p}(\Pi)[\mathfrak{m}_{\epsilon}^s]\cong \Pi^{N_0}[\mathfrak{m}_{\epsilon}^s]\cong \Pi_\mu^{N_0}[\mathfrak{m}_{\epsilon}^s]$ if $\mu:\mathfrak{t}_L\rightarrow L$ is the derivative of $\epsilon$.

Let $\mathfrak{I}\subset S_\infty[1/p]$ be an ideal such that $\dim_L(S_\infty[1/p]/\mathfrak{I})<\infty$ and define $\Pi_{\mathfrak{I}}:=\Pi_\infty[\mathfrak{I}]$. As the continuous dual $\Pi_\infty'$ is a finite projective $S_\infty\dbl K_p\dbr [1/p]$-module (property (ii) in \S\ref{firstclassical}), the continuous dual $\Pi_\infty[\mathfrak{I}]'$ of the $G_p$-representation $\Pi_\infty[\mathfrak{I}]$, which is isomorphic to $\Pi_\infty'/\mathfrak{I}$ by the discussion at the end of \cite[\S3.1]{BHS}, is a finite projective $S_\infty\dbl K_p\dbr [1/p]/\mathfrak{I}S_\infty\dbl K_p\dbr [1/p]\cong \mathcal{O}_L\dbl K_p\dbr [1/p]$-module (in particular it is an admissible continuous representation of $G_p$ over $L$). Moreover it is immediate to check that $\Pi_\infty^{R_\infty-\rm an}[\mathfrak{I}]$ is isomorphic to the subspace $\Pi_{\mathfrak{I}}^{\rm an}$ of locally $\Q_p$-analytic vectors of $\Pi_{\mathfrak{I}}$.

Taking the image of a vector in (the underlying $L$-vector space of) $\lambda$ or $\mu$ gives natural isomorphisms:
$$\Hom_{U(\mathfrak{g}_L)}\big(U(\mathfrak{g}_L)\otimes_{U(\mathfrak{b}_L)}\lambda,\Pi_{\mathfrak{I}}^{\rm an}\big)\buildrel\sim\over\longrightarrow (\Pi_{\mathfrak{I}}^{\rm an})_\lambda^{\mathfrak{n}_L}\ {\rm and}\ \Hom_{U(\mathfrak{g}_L)}\big(U(\mathfrak{g}_L)\otimes_{U(\mathfrak{b}_L)}\mu,\Pi_{\mathfrak{I}}^{\rm an}\big)\buildrel\sim\over\longrightarrow (\Pi_{\mathfrak{I}}^{\rm an})_\mu^{\mathfrak{n}_L}.$$
As $\mu$ is strongly linked to $\lambda$ by assumption, \cite[Th.5.1]{HumBGG} implies the existence of a unique (up to $L$-homothety) $U(\mathfrak{g}_L)$-equivariant injection:
\begin{equation}\label{mulambdau}
\iota_{\mu, \lambda}:\,U(\mathfrak{g}_L)\otimes_{U(\mathfrak{b}_L)}\mu\hookrightarrow U(\mathfrak{g}_L)\otimes_{U(\mathfrak{b}_L)}\lambda
\end{equation}
which induces an $L$-linear map:
\begin{equation}\label{imulambda}
\iota^*_{\mu, \lambda}:\,(\Pi_{\mathfrak{I}}^{\rm an})_\lambda^{\mathfrak{n}_L}\longrightarrow (\Pi_{\mathfrak{I}}^{\rm an})_\mu^{\mathfrak{n}_L}.
\end{equation}

We claim that (\ref{imulambda}) maps the subspace $(\Pi_{\mathfrak{I}}^{\rm an})_\lambda^{N_0}$ to the subspace $(\Pi_{\mathfrak{I}}^{\rm an})_\mu^{N_0}$. It is enough to prove:
\begin{equation}\label{mulambda}
\iota^*_{\mu, \lambda}\circ\pi_{N_0}=\pi_{N_0}\circ \iota^*_{\mu, \lambda}.
\end{equation}
Let $\mathfrak{v}$ be the image by (\ref{mulambdau}) of a nonzero vector $v$ in the underlying $L$-vector space of $\mu$. Writing $U(\mathfrak{g}_L)\otimes_{U(\mathfrak{b}_L)}\lambda\cong U(\mathfrak{n}_L^-)$ we see that $\mathfrak{v}\in U(\mathfrak{n}_L^-)_{\mu-\lambda}$ (with obvious notation), that $\mathfrak{v}$ is killed by $\mathfrak{n}_L$ in $U(\mathfrak{g}_L)\otimes_{U(\mathfrak{b}_L)}\lambda$ and that the morphism $\iota^*_{\mu, \lambda}$ is given by the action (on the left) by $\mathfrak{v}$. To get (\ref{mulambda}) it is enough to prove $\pi_{N_0}\circ\mathfrak{v}=\mathfrak{v}\circ\pi_{N_0}$ on $(\Pi_\mathfrak{I}^{\rm an})^{\mathfrak{n}_L}_{\lambda}$, which itself follows from $n\circ \mathfrak{v}=\mathfrak{v}\circ n$ for $n\in N_0$ ($n$ acting via the underlying $G_p$-action on $\Pi_{\mathfrak{I}}^{\rm an}$), or equivalently $\mathrm{Ad}(n)(\mathfrak{v})=\mathfrak{v}$ on $(\Pi_\mathfrak{I}^{\rm an})^{\mathfrak{n}_L}_{\lambda}$. Writing $n=\exp(\mathfrak{m})$ with $\mathfrak{m}\in\mathfrak{n}_L$ (do not confuse here with the maximal ideal $\mathfrak{m}$!), recall we have $\mathrm{Ad}(n)(\mathfrak{v})=\exp(\mathrm{ad}(\mathfrak{m}))(\mathfrak{v})$ (using standard notation). Since $\mathfrak{v}$ is killed by left multiplication by $\mathfrak{n}_L$ in $U(\mathfrak{g}_L)\otimes_{U(\mathfrak{b}_L)}\lambda$, we have:
$$\exp(\mathrm{ad}(\mathfrak{m}))(\mathfrak{v})\in\mathfrak{v}+U(\mathfrak{g}_L)(\mathfrak{n}_L+\ker(\lambda))$$
where $\ker(\lambda):=\ker(U(\mathfrak{g}_L)\rightarrow U(\mathfrak{g}_L)\otimes_{U(\mathfrak{b}_L)}\lambda, \ \mathfrak{x}\mapsto \mathfrak{x}\otimes v)$. The action of $\mathrm{Ad}(n)(\mathfrak{v})$ on $(\Pi_\mathfrak{I}^{\rm an})^{\mathfrak{n}_L}_{\lambda}$ is thus the same as that of $\mathfrak{v}$. 

We still write:
\begin{equation}\label{mulambda0}
\iota^*_{\mu, \lambda}:\,(\Pi_{\mathfrak{I}}^{\rm an})_\lambda^{N_0}\longrightarrow (\Pi_{\mathfrak{I}}^{\rm an})_\mu^{N_0}
\end{equation}
for the map induced by (\ref{imulambda}). Using $\mathfrak{v}\in U(\mathfrak{n}_L^-)_{\mu-\lambda}$ together with (\ref{mulambda}), it is easy to check that $\iota^*_{\mu, \lambda}\circ t=(\delta\epsilon^{-1})(t)(t\circ \iota^*_{\mu, \lambda})$ for $t\in T_p^+$ (for the previous $L[T_p^+]$-module structure). Moreover, it follows from Lemma \ref{injective} that (\ref{imulambda}) is surjective, hence the top horizontal map and the two vertical maps are surjective in the commutative diagram:
$$\begin{xy}\xymatrix{(\Pi_{\mathfrak{I}}^{\rm an})_\lambda^{\mathfrak{n}_L}\ar[d]_{\pi_{N_0}}\ar[r]^{(\ref{imulambda})} & (\Pi_{\mathfrak{I}}^{\rm an})_\mu^{\mathfrak{n}_L}\ar[d]^{\pi_{N_0}}\\(\Pi_{\mathfrak{I}}^{\rm an})_\lambda^{N_0}\ar[r]^{(\ref{mulambda0})} & (\Pi_{\mathfrak{I}}^{\rm an})_\mu^{N_0}}
\end{xy}$$
which implies that (\ref{mulambda0}) is also surjective. Note also that both (\ref{imulambda}) and (\ref{mulambda0}) trivially commute with the action of $R_\infty$ (which factors through $R_\infty/\mathfrak{I}R_\infty$).

From \cite[\S0.3]{CD} we have $\Pi_{\mathfrak{I}}^{\rm an}\cong \lim_{h\rightarrow +\infty}\Pi_{\mathfrak{I},H_p}^{(h)}$ and thus:
$$(\Pi_{\mathfrak{I}}^{\rm an})_\lambda^{N_0}\cong \lim_{h\rightarrow +\infty}(\Pi_{\mathfrak{I},H_p}^{(h)})_\lambda^{N_0}\ \ {\rm and}\ \ (\Pi_{\mathfrak{I}}^{\rm an})_\mu^{N_0}\cong \lim_{h\rightarrow +\infty}(\Pi_{\mathfrak{I},H_p}^{(h)})_\mu^{N_0}.$$
By Lemma \ref{compact} there is $z\in T_p^+$ which acts compactly on $(\Pi_{\mathfrak{I},H_p}^{(h)})_\lambda^{N_0}$ and $(\Pi_{\mathfrak{I},H_p}^{(h)})_\mu^{N_0}$. We deduce from this fact together with \cite[Prop.9]{SerreCC} and \cite[Prop.12]{SerreCC} that the map $\iota^*_{\mu,\lambda}$ in (\ref{mulambda0}) remains surjective at the level of {\it generalized eigenspaces} for the action of $T_p^+$ (twisting this action by the character $\delta\epsilon^{-1}$ on the right hand side). Consequently $\iota^*_{\mu,\lambda}$ induces a surjective map:
$$\bigcup_{s\geq 1}(\Pi_{\mathfrak{I}}^{\rm an})_{\lambda}^{N_0}[\mathfrak{m}_{\delta}^s]\twoheadrightarrow \bigcup_{s\geq 1}(\Pi_{\mathfrak{I}}^{\rm an})^{N_0}_\mu[\mathfrak{m}_{\epsilon}^s].$$
As both the source and target of this map are unions of finite dimensional $L$-vector spaces (as follows from the admissibility of $\Pi_{\mathfrak{I}}^{\rm an}$, \cite[Th.4.3.2]{EmertonJacquetI} and (\ref{jm})) which are stable under $R_\infty$ and as $\iota^*_{\mu,\lambda}$ is $R_\infty$-equivariant, the following map induced by $\iota^*_{\mu,\lambda}$ remains surjective:
\begin{equation}\label{cupcup}
\bigcup_{s,t\geq 1}(\Pi_\mathfrak{I}^{\rm an})_{\lambda}^{N_0}[\mathfrak{m}_{\delta}^s, \mathfrak{m}^t]\twoheadrightarrow \bigcup_{s,t\geq 1}(\Pi_\mathfrak{I}^{\rm an})^{N_0}_\mu[\mathfrak{m}_{\epsilon}^s,\mathfrak{m}^t].
\end{equation}
Since $\mathfrak{m}^t$ is an ideal of cofinite dimension in $R_\infty[1/p]$, the inverse image $\mathfrak{I}$ of $\mathfrak{m}^t$ in $S_\infty[1/p]$ is {\it a fortiori} of cofinite dimension in $S_\infty[1/p]$ and we can apply (\ref{cupcup}) with such an $\mathfrak I$. But we have for this $\mathfrak{I}$:
\begin{equation*}
(\Pi_{\mathfrak{I}}^{\rm an})^{N_0}_\lambda[\mathfrak{m}_\delta^s,\mathfrak{m}^t]=(\Pi_\infty^{R_\infty-\rm an})^{N_0}_\lambda[\mathfrak{m}_\delta^s,\mathfrak{m}^t,\mathfrak{I}]=(\Pi_\infty^{R_\infty-\rm an})^{N_0}_\lambda[\mathfrak{m}_{\delta}^s,\mathfrak{m}^t]
\end{equation*}
and likewise with $\mathfrak{m}_{\epsilon}$, so that (\ref{cupcup}) is a surjection:
$$\bigcup_{s,t\geq 1}(\Pi_\infty^{R_\infty-\rm an})_{\lambda}^{N_0}[\mathfrak{m}_{\delta}^s, \mathfrak{m}^t]\twoheadrightarrow \bigcup_{s,t\geq 1}(\Pi_\infty^{R_\infty-\rm an})_{\mu}^{N_0}[\mathfrak{m}_{\epsilon}^s, \mathfrak{m}^t].$$
Looking at the eigenspaces on both sides, we obtain $\Hom_{T_p}(\delta,J_{B_p}(\Pi_\infty^{R_\infty-\rm an}[\mathfrak{m}]))\ne 0$ if $\Hom_{T_p}(\epsilon,J_{B_p}(\Pi_\infty^{R_\infty-\rm an}[\mathfrak{m}]))\ne 0$.\\
Finally, when $k(\delta)=k(\epsilon)$ is larger than $L$, we replace $\Pi_\infty$ by $\Pi_\infty':=\Pi_\infty\otimes_LL'$, $S_\infty[1/p]$ by $S_\infty[1/p]\otimes_LL'$, ${\mathfrak m}$ by ${\mathfrak m}':=\ker(R_\infty[1/p]\otimes_LL'\twoheadrightarrow k({\mathfrak m})\otimes_LL'\twoheadrightarrow L')$ (the last surjection coming from the inclusion $k({\mathfrak m})\subseteq L'$) and the reader can check that all the arguments of the previous proof go through {\it mutatis mutandis}.
\end{proof}

\subsection{Tangent spaces on the trianguline variety}\label{geometry}

We prove that Conjecture \ref{BHS} implies Conjecture \ref{mainconj} (when $\rbar$ ``globalizes'' and $x$ is very regular) and give one (conjectural) application.

We keep the notation and assumptions of \S\ref{closed}. We fix $x=(r,\delta)\in \widetilde X_{\rm tri}^\square(\rbar)\subseteq X_{\rm tri}^\square(\rbar)$ which is crystalline strictly dominant very regular. Recall from Lemma \ref{paramofcrystpt} that $\delta=(\delta_1\dots,\delta_n)$ where $\delta_i=z^{{\bf k}_i}{\rm unr}(\varphi_i)$ with ${\bf k}_i=(k_{\tau,i})_{\tau:\, K\hookrightarrow L}\in\mathbb{Z}^{\Hom(K,L)}$ and $\varphi_i\in k(x)^\times$. The following result immediately follows from Proposition \ref{acconZaut} and Theorem \ref{upperbound} applied to $X=\widetilde X_{\rm tri}^\square(\overline r)$.

\begin{coro}\label{easybound}
Assume Conjecture \ref{BHS}, then we have:
$$\dim_{k(x)}T_{\widetilde X_{\rm tri}^\square(\overline r),x}\leq \lg(w_x)-d_x+\dim X_{\rm tri}^\square(\rbar)=\lg(w_x)-d_x+n^2+[K:\Q_p]\frac{n(n+1)}{2}.$$
\end{coro}

The rest of this section is devoted to the proof of the converse inequality (still assuming Conjecture \ref{BHS}).

As in the proof of Proposition \ref{inter}, we consider for $1\leq i\leq n$ the cartesian diagram which defines $W_i$ (with the notation of \S\ref{wedgesection}):
$$\xymatrix{{\rm Ext}^1_{(\varphi,\Gamma_K)}\big(D_{\rig}(r),D_{\rig}(r)\big)\ar[r]&{\rm Ext}^1_{(\varphi,\Gamma_K)}\big(D_{\rig}(r)^{\leq i},D_{\rig}(r)\big)\\ 
W_i \ar[r]\ar@{^{(}->}[u]& 
{\rm Ext}^1_{(\varphi,\Gamma_K)}\big(D_{\rig}(r)^{\leq i},D_{\rig}(r)^{\leq i}\big).\ar@{^{(}->}[u]}$$
We define $W_{{\rm cris},i}\subseteq {\rm Ext}^1_{\rm cris}(D_{\rig}(r),D_{\rig}(r))$ as $W_i$ but replacing everywhere ${\rm Ext}^1_{(\varphi,\Gamma_K)}$ by its subspace ${\rm Ext}^1_{\rm cris}$. Note that $W_{{\rm cris},i}\subseteq W_i$ for $1\leq i\leq n$.

\begin{prop}\label{gradcris}
For $1\leq i\leq n$, we have isomorphisms of $k(x)$-vector spaces:
\begin{eqnarray}
W_1\cap\cdots \cap W_{i-1}/W_1\cap\cdots \cap W_{i} &\buildrel\sim\over\longrightarrow &
{\rm Ext}_{(\varphi,\Gamma_K)}^1\big({\rm gr}_iD_{\rig}(r),D_{\rig}(r)/D_{\rig}(r)^{\leq i}\big)\\
\nonumber W_{{\rm cris},1}\cap\cdots \cap W_{{\rm cris},i-1}/W_{{\rm cris},1}\cap\cdots \cap W_{{\rm cris},i} &\buildrel\sim\over\longrightarrow &
{\rm Ext}_{\rm cris}^1\big({\rm gr}_iD_{\rig}(r),D_{\rig}(r)/D_{\rig}(r)^{\leq i}\big)
\end{eqnarray}
where \ $W_1\cap\cdots \cap W_{i-1}:={\rm Ext}_{(\varphi,\Gamma_K)}^1(D_{\rig}(r),D_{\rig}(r))$ \ (resp. \ $W_{{\rm cris},1}\cap\cdots \cap W_{{\rm cris},i-1}:={\rm Ext}_{\rm cris}^1(D_{\rig}(r),D_{\rig}(r))$) if $i=1$.
\end{prop}
\begin{proof}
We write $D_{\rig}$ instead of $D_{\rig}(r)$ and drop the subscript $(\varphi,\Gamma_K)$ in this proof. We start with the first isomorphism, the proof of which is analogous to (though simpler than) the proof of (\ref{isoi}) in \S\ref{endofproof}. We have the exact sequence (using Definition \ref{veryreg}):
\begin{equation}\label{exactwcris}
\!\!0\rightarrow {\rm Ext}^1\big({\rm gr}_iD_{\rig},D_{\rig}/D_{\rig}^{\leq i}\big)\rightarrow {\rm Ext}^1\big(D_{\rig}^{\leq i},D_{\rig}/D_{\rig}^{\leq i}\big)\rightarrow {\rm Ext}^1\big(D_{\rig}^{\leq i-1},D_{\rig}/D_{\rig}^{\leq i}\big)\rightarrow  0.
\end{equation}
The composition:
\begin{equation*}
W_1\cap\cdots \cap W_{i-1}\hookrightarrow {\rm Ext}^1\big(D_{\rig},D_{\rig}\big)\twoheadrightarrow {\rm Ext}^1\big(D_{\rig}^{\leq i},D_{\rig}/D_{\rig}^{\leq i}\big)
\end{equation*}
lands in ${\rm Ext}^1({\rm gr}_iD_{\rig},D_{\rig}/D_{\rig}^{\leq i})$ by (\ref{exactwcris}). If $v\in W_1\cap\cdots \cap W_{i-1}$ is also in $W_i$, then its image in ${\rm Ext}^1(D_{\rig}^{\leq i},D_{\rig}/D_{\rig}^{\leq i})$ is $0$. We thus deduce a canonical induced map:
\begin{equation}\label{wimap}
W_1\cap\cdots \cap W_{i-1}/W_1\cap\cdots \cap W_{i} \rightarrow {\rm Ext}^1\big({\rm gr}_iD_{\rig},D_{\rig}/D_{\rig}^{\leq i}\big).
\end{equation}
Let us prove that (\ref{wimap}) is surjective. One easily checks that ${\rm Ext}^1(D_{\rig}/D_{\rig}^{\leq i-1},D_{\rig})\subseteq W_1\cap\cdots \cap W_{i-1}$ and that the natural map ${\rm Ext}^1(D_{\rig}/D_{\rig}^{\leq i-1},D_{\rig})\rightarrow {\rm Ext}^1({\rm gr}_iD_{\rig},D_{\rig}/D_{\rig}^{\leq i})$ is surjective (again by Definition \ref{veryreg}). This implies that {\it a fortiori} (\ref{wimap}) must also be surjective. Let us prove that (\ref{wimap}) is injective. If $w\in W_1\cap\cdots \cap W_{i-1}$ maps to zero, then the image of $w$ in ${\rm Ext}^1(D_{\rig}^{\leq i},D_{\rig}/D_{\rig}^{\leq i})$ is also zero, i.e. $w\in W_i$ hence $w\in W_1\cap\cdots \cap W_{i}$. 

The proof for the second isomorphism is exactly the same replacing everywhere $W_j$ by $W_{{\rm cris},j}$ and ${\rm Ext}^1_{(\varphi,\Gamma_K)}$ by ${\rm Ext}^1_{\rm cris}$.
\end{proof}

\begin{coro}\label{condcris}
We have:
\begin{eqnarray*}
\dim_{k(x)}\big(W_{1}\cap\cdots \cap W_{n-1}\big) &=& \dim_{k(x)}{\rm Ext}_{(\varphi,\Gamma_K)}^1\big(D_{\rig}(r),D_{\rig}(r)\big) - [K:\Q_p]\frac{n(n-1)}{2}\\
\dim_{k(x)}\big(W_{{\rm cris},1}\cap\cdots \cap W_{{\rm cris},n-1}\big) &=& \dim_{k(x)}{\rm Ext}^1_{\rm cris}\big(D_{\rig}(r),D_{\rig}(r)\big)-\lg(w_x).
\end{eqnarray*}
\end{coro}
\begin{proof}
This follows from Proposition \ref{gradcris} together with (\ref{dimcrisi}) and Lemma \ref{dim1} (both for $\ell=i$) by the same argument as at the end of the proof of Proposition \ref{condsplit}.
\end{proof}

\begin{rema}\label{intercris}
{\rm Note that $W_1\cap\cdots \cap W_{n-1} \cap {\rm Ext}^1_{\rm cris}(D_{\rig}(r),D_{\rig}(r))=W_{{\rm cris},1}\cap\cdots \cap W_{{\rm cris},n-1}$.}
\end{rema}

Now consider $x':=(r,\delta')=(r,\delta'_1,\dots,\delta'_n)$ with $\delta'_i:=z^{{\bf k}_{w_x^{-1}(i)}}{\rm unr}(\varphi_i)$, then $x'\in \widetilde U_{\rm tri}^\square(\rbar)$ by (\ref{triang}). We also have $\omega(x')\in \mathcal{W}^n_{w_x,{\bf k},L}$ by (\ref{point}), thus $x'\in \widetilde U_{\rm tri}^\square(\rbar)\times_{\mathcal{W}^n_L}\mathcal{W}^n_{w_x,{\bf k},L}\subseteq \widetilde U_{\rm tri}^\square(\rbar)$ and $\jmath_{w,{\bf k}}(x')=x$. Recall from \S\ref{wedgesection} and the smoothness of $U_{\rm tri}^\square(\rbar)$ over $\mathcal{W}^n_{L}$ that the weight map $\omega$ induces a $k(x)$-linear surjection on tangent spaces (note that $k(x')=k(x)$):
\begin{equation}\label{sen}
d\omega : T_{\widetilde X_{\rm tri}^\square(\rbar),x'}\cong T_{X_{\rm tri}^\square(\rbar),x'}\twoheadrightarrow T_{{\mathcal W}^n_L,\omega(x')}\cong k(x)^{[K:\Q_p]n}, \ \ \vec{v} \longmapsto (d_{\tau,i,\vec{v}})_{1\leq i\leq n,\tau:\, K\hookrightarrow L}.
\end{equation}

\begin{prop}\label{tgtx'}
We have an isomorphism of $k(x)$-subvector spaces of $T_{\widetilde X_{\rm tri}^\square(\rbar),x'}$:
\begin{multline*}
T_{(\widetilde X_{\rm tri}^\square(\rbar)\times_{\mathcal{W}^n_L}\mathcal{W}^n_{w_x,{\bf k},L})^\red,x'}\buildrel\sim\over\longrightarrow \\
\left\{\vec{v}\in T_{\widetilde X_{\rm tri}^\square(\rbar),x'}\ {\rm such\ that}\ d_{\tau,i,\vec{v}}=d_{\tau,w_{x,\tau}^{-1}(i),\vec{v}},\ 1\leq i\leq n,\ \tau:\, K\hookrightarrow L\right\}.
\end{multline*}
In particular $\dim_{k(x)}T_{(\widetilde X_{\rm tri}^\square(\rbar)\times_{\mathcal{W}^n_L}\mathcal{W}^n_{w_x,{\bf k},L})^\red,x'}=\dim X_{\rm tri}^\square(\rbar)-d_x$.
\end{prop}
\begin{proof}
We write $\Hom$ instead of $\Hom_{k(x)-{\rm alg}}$ in this proof. Let $\widetilde U_{{\rm tri},w_x,{\bf k}}^\square(\rbar):=\widetilde U_{\rm tri}^\square(\rbar)\times_{\mathcal{W}^n_L}\mathcal{W}^n_{w_x,{\bf k},L}$, we have:
\begin{equation}\label{tensorlocal}
{\mathcal O}_{\widetilde U_{{\rm tri},w_x,{\bf k}}^\square(\rbar),x'}\cong {\mathcal O}_{\widetilde U_{\rm tri}^\square(\rbar),x'}\otimes_{{\mathcal O}_{\mathcal{W}^n_L,\omega(x')}}{\mathcal O}_{\mathcal{W}^n_{w_x,{\bf k},L},\omega(x')}
\end{equation}
and note that $T_{(\widetilde X_{\rm tri}^\square(\rbar)\times_{\mathcal{W}^n_L}\mathcal{W}^n_{w_x,{\bf k},L})^\red,x'}=T_{\widetilde U_{{\rm tri},w,{\bf k}}^\square(\rbar),x'}$. Recall that, if $A,B,C,D$ are commutative $k(x)$-algebras with $B,C$ being $A$-algebras, we have:
\begin{equation}\label{abcd}
\Hom(B\otimes_AC,D)\buildrel\sim\over\longrightarrow \Hom(B,D)\times_{\Hom(A,D)}\Hom(C,D).
\end{equation}
From (\ref{tensorlocal}) and (\ref{abcd}) we deduce:
\begin{multline}
T_{\widetilde U_{{\rm tri},w,{\bf k}}^\square(\rbar),x'}=\Hom\big({\mathcal O}_{\widetilde U_{{\rm tri},w_x,{\bf k}}^\square(\rbar),x'},k(x)[\varepsilon]/(\varepsilon^2)\big)\cong T_{\widetilde X_{\rm tri}^\square(\rbar),x'}\times_{T_{{\mathcal W}^n_L,\omega(x')}}T_{{\mathcal W}^n_{w_x,{\bf k},L},\omega(x')}.
\end{multline}
But from (\ref{equations}) we have:
$$T_{{\mathcal W}^n_{w_x,{\bf k},L},\omega(x')}=\left\{(d_{\tau,i})_{1\leq i\leq n,\tau:\, K\hookrightarrow L}\in T_{{\mathcal W}^n_{L},\omega(x')}\ {\rm such\ that}\ d_{\tau,i}=d_{\tau,w_{x,\tau}^{-1}(i)},\ \forall\ i,\ \forall\ \tau\right\}$$
whence the first statement. The last statement comes from $\dim_{k(x)}T_{\widetilde X_{\rm tri}^\square(\rbar),x'}=\dim \widetilde X_{\rm tri}^\square(\rbar)=\dim X_{\rm tri}^\square(\rbar)$ (since $x'$ is smooth on $X_{\rm tri}^\square(\rbar)$ as $x'\in U_{\rm tri}^\square(\rbar)$), the surjectivity of $T_{\widetilde X_{\rm tri}^\square(\rbar),x'}\rightarrow T_{{\mathcal W}^n_{L},\omega(x')}$ (since the morphism $\widetilde U_{\rm tri}^\square(\rbar)\rightarrow {\mathcal W}^n_{L}$ is smooth by \cite[Th.2.6(iii)]{BHS}) and the same argument as in the proof of Proposition \ref{condweight}. 
\end{proof}

Recall from the discussion just before Conjecture \ref{mainconj} that we have a closed embedding $\iota_{\bf k}:\widetilde\Xfrak_{\overline r}^{\square,{\bf k}\rm -cr}\hookrightarrow \widetilde X_{\rm tri}^\square(\overline r)$ with $x\in \iota_{\bf k}(\widetilde\Xfrak_{\overline r}^{\square,{\bf k}\rm -cr})$. We deduce an injection of $k(x)$-vector spaces:
$$T_{\iota_{\bf k}(\widetilde \Xfrak_{\overline r}^{\square,{\bf k}\rm -cr}),x}\hookrightarrow T_{\widetilde X_{\rm tri}^\square(\overline r),x}.$$
Likewise we deduce from Proposition \ref{companion} (assuming Conjecture \ref{BHS}) another injection of $k(x)$-vector spaces:
$$T_{\jmath_{w_x,{\bf k}}\big(\overline{\widetilde U_{\rm tri}^\square(\rbar)\times_{\mathcal{W}^n_L}\mathcal{W}^n_{w,{\bf k},L}}\big),x}\hookrightarrow T_{\widetilde X_{\rm tri}^\square(\overline r),x}.$$
Taking the sum in $T_{\widetilde X_{\rm tri}^\square(\rbar),x}$ of these two subspaces of $T_{\widetilde X_{\rm tri}^\square(\rbar),x}$, we have an injection of $k(x)$-vector spaces:
\begin{equation}\label{injtang}
T_{\jmath_{w_x,{\bf k}}\big(\overline{\widetilde U_{\rm tri}^\square(\rbar)\times_{\mathcal{W}^n_L}\mathcal{W}^n_{w,{\bf k},L}}\big),x}+T_{\iota_{\bf k}(\widetilde\Xfrak_{\overline r}^{\square,{\bf k}\rm -cr}),x}\hookrightarrow T_{\widetilde X_{\rm tri}^\square(\rbar),x}.
\end{equation}

\begin{prop}\label{dimsum}
Assume Conjecture \ref{BHS}, then we have:
$$\dim_{k(x)}\big(T_{\jmath_{w_x,{\bf k}}\big(\overline{\widetilde U_{\rm tri}^\square(\rbar)\times_{\mathcal{W}^n_L}\mathcal{W}^n_{w,{\bf k},L}}\big),x}+T_{\iota_{\bf k}(\widetilde\Xfrak_{\overline r}^{\square,{\bf k}\rm -cr}),x}\big)=\lg(w_x)-d_{x}+\dim X_{\rm tri}^\square(\rbar).$$
\end{prop}
\begin{proof}
The composition $\widetilde X_{\rm tri}^\square(\rbar)\hookrightarrow \mathfrak{X}_{\rbar}^\square\times \mathcal{T}^n_L\twoheadrightarrow \mathfrak{X}_{\rbar}^\square$ induces a $k(x)$-linear morphism $T_{\widetilde X_{\rm tri}^\square(\rbar),x'}\!\rightarrow T_{\mathfrak{X}_{\rbar}^\square,r}$. Since $x'\in \widetilde U_{\rm tri}^\square(\rbar)\subseteq U_{\rm tri}^\square(\rbar)$, it follows from \cite[Th.Cor.6.3.10]{KPX} (arguing e.g. as in the proof of \cite[Lem.2.11]{BHS}) that the triangulation $(D_{\rig}^{\leq i})_{1\leq i\leq n}$ ``globalizes'' in a small neighbourhood of $x'$ in $U_{\rm tri}^\square(\rbar)$, or equivalenty in $\widetilde U_{\rm tri}^\square(\rbar)$. In particular, for any $\vec{v}\in T_{\widetilde U_{\rm tri}^\square(\rbar),x'}\cong T_{\widetilde X_{\rm tri}^\square(\rbar),x'}$ we have a triangulation of $D_{\rig}(r_{\vec{v}})$ by free $(\varphi,\Gamma_K)$-submodules over $\mathcal{R}_{k(x)[\varepsilon]/(\varepsilon^2),K}$ such that the associated parameter is $(\delta_{1,\vec{v}},\dots,\delta_{n,\vec{v}})$ (see \S\ref{wedgesection} for the notation). This has two consequences: (1) the proof of Lemma \ref{inj} goes through and the above map $T_{\widetilde X_{\rm tri}^\square(\rbar),x'}\rightarrow T_{\mathfrak{X}_{\rbar}^\square,r}$ is an injection of $k(x)$-vector spaces and (2) the image of the composition $T_{\widetilde X_{\rm tri}^\square(\rbar),x'}\hookrightarrow T_{\mathfrak{X}_{\rbar}^\square,r}\twoheadrightarrow {\rm Ext}_{(\varphi,\Gamma_K)}^1(D_{\rig}(r),D_{\rig}(r))$ (see Lemma \ref{ext1}) lies in $W_{1}\cap\cdots \cap W_{n-1}\!\subseteq \!{\rm Ext}_{(\varphi,\Gamma_K)}^1(D_{\rig}(r),D_{\rig}(r))$. From Lemma \ref{ext1} we thus obtain an exact sequence:
$$0\rightarrow K(r)\cap T_{\widetilde X_{\rm tri}^\square(\rbar),x'}\rightarrow T_{\widetilde X_{\rm tri}^\square(\rbar),x'}\rightarrow W_{1}\cap\cdots \cap W_{n-1}.$$
But $\dim_{k(x)}T_{\widetilde X_{\rm tri}^\square(\rbar),x'}=\dim X_{\rm tri}^\square(\rbar)$ since $\widetilde X_{\rm tri}^\square(\rbar)$ is smooth at $x'$, and from Lemma \ref{ext1}, Lemma \ref{dim} and Corollary \ref{condcris}, we have:
$$\dim_{k(x)}K(r) + \dim_{k(x)}\big(W_{1}\cap\cdots \cap W_{n-1}\big) = n^2 + [K:\Q_p]\frac{n(n+1)}{2}= \dim X_{\rm tri}^\square(\rbar)$$
which forces a short exact sequence $0\rightarrow K(r)\rightarrow T_{\widetilde X_{\rm tri}^\square(\rbar),x'}\rightarrow W_{1}\cap\cdots \cap W_{n-1}\rightarrow 0$. It then follows from Proposition \ref{tgtx'} that we have a short exact sequence of $k(x)$-vector spaces:
\begin{equation}\label{se1}
0\rightarrow K(r)\rightarrow T_{(\widetilde X_{\rm tri}^\square(\rbar)\times_{\mathcal{W}^n_L}\mathcal{W}^n_{w_x,{\bf k},L})^\red,x'}\rightarrow W_{1}\cap\cdots \cap W_{n-1}\cap V\rightarrow 0
\end{equation}
where $V\subseteq {\rm Ext}^1_{(\varphi,\Gamma_K)}(D_{\rig}(r),D_{\rig}(r))$ is as in \S\ref{endofproof} (the intersection on the right hand side being in ${\rm Ext}^1_{(\varphi,\Gamma_K)}(D_{\rig}(r),D_{\rig}(r))$). 

Arguing as in \cite[\S2.3.5]{KisinModularity} we also have a short exact sequence (see \cite[(3.3.5)]{Kisindef}):
\begin{equation}\label{se2}
0\rightarrow K(r)\rightarrow T_{\Xfrak_{\overline r}^{\square,{\bf k}\rm -cr},r}\rightarrow {\rm Ext}^1_{\rm cris}(D_{\rig}(r),D_{\rig}(r))\rightarrow 0.
\end{equation}
Using $T_{\iota_{\bf k}(\widetilde\Xfrak_{\overline r}^{\square,{\bf k}\rm -cr}),x}\cong T_{\Xfrak_{\overline r}^{\square,{\bf k}\rm -cr},r}$ (which easily follows from the fact that the Frobenius eigenvalues $(\varphi_1,\dots,\varphi_n)$ are pairwise distinct) and:
$$T_{(\widetilde X_{\rm tri}^\square(\rbar)\times_{\mathcal{W}^n_L}\mathcal{W}^n_{w,{\bf k},L})^\red,x'}\cong T_{\widetilde U_{\rm tri}^\square(\rbar)\times_{\mathcal{W}^n_L}\mathcal{W}^n_{w,{\bf k},L},x'}\cong T_{\overline{\widetilde U_{\rm tri}^\square(\rbar)\times_{\mathcal{W}^n_L}\mathcal{W}^n_{w,{\bf k},L}},x'}\buildrel\sim\over\rightarrow T_{\jmath_{w_x,{\bf k}}(\overline{\widetilde U_{\rm tri}^\square(\rbar)\times_{\mathcal{W}^n_L}\mathcal{W}^n_{w,{\bf k},L}}),x},$$
we deduce from (\ref{se1}) and (\ref{se2}) a short exact sequence of $k(x)$-vector spaces:
\begin{multline*}
0\rightarrow K(r)\rightarrow T_{\jmath_{w_x,{\bf k}}\big(\overline{\widetilde U_{\rm tri}^\square(\rbar)\times_{\mathcal{W}^n_L}\mathcal{W}^n_{w,{\bf k},L}}\big),x}\cap T_{\iota_{\bf k}(\widetilde\Xfrak_{\overline r}^{\square,{\bf k}\rm -cr}),x}\rightarrow \\
W_{1}\cap\cdots \cap W_{n-1}\cap V\cap {\rm Ext}^1_{\rm cris}(D_{\rig}(r),D_{\rig}(r))\rightarrow 0,
\end{multline*}
the intersection in the middle being in $T_{\mathfrak{X}_{\rbar}^\square,r}$. But we have:
\begin{multline*}
W_{1}\cap\cdots \cap W_{n-1}\cap V\cap {\rm Ext}^1_{\rm cris}(D_{\rig}(r),D_{\rig}(r))\buildrel\sim\over\rightarrow W_{{\rm cris},1}\cap\cdots \cap W_{{\rm cris},n-1}\cap V\\
\buildrel\sim\over\rightarrow W_{{\rm cris},1}\cap\cdots \cap W_{{\rm cris},n-1}
\end{multline*}
where \ \ the \ \ first \ \ isomorphism \ \ is \ \ Remark \ \ \ref{intercris} \ \ and \ \ the \ \ second \ \ follows \ \ from \ ${\rm Ext}^1_{\rm cris}(D_{\rig}(r),D_{\rig}(r))\subseteq V$ (since the Hodge-Tate weights don't vary at all in ${\rm Ext}^1_{\rm cris}(D_{\rig}(r),D_{\rig}(r))$). From Corollary \ref{condcris} we thus get:
\begin{multline}\label{dimker}
\dim_{k(x)}\Big(T_{\jmath_{w_x,{\bf k}}\big(\overline{\widetilde U_{\rm tri}^\square(\rbar)\times_{\mathcal{W}^n_L}\mathcal{W}^n_{w,{\bf k},L}}\big),x}\cap T_{\iota_{\bf k}(\widetilde\Xfrak_{\overline r}^{\square,{\bf k}\rm -cr}),x}\big)=\\
\dim_{k(x)}K(r)+\dim_{k(x)}{\rm Ext}^1_{\rm cris}\big(D_{\rig}(r),D_{\rig}(r)\big)-\lg(w_x).
\end{multline}
We now compute using Proposition \ref{tgtx'}, (\ref{se2}) and (\ref{dimker}):
\begin{multline*}
\dim_{k(x)}\big(T_{\jmath_{w_x,{\bf k}}\big(\overline{\widetilde U_{\rm tri}^\square(\rbar)\times_{\mathcal{W}^n_L}\mathcal{W}^n_{w,{\bf k},L}}\big),x}+ T_{\iota_{\bf k}(\widetilde\Xfrak_{\overline r}^{\square,{\bf k}\rm -cr}),x}\big)=\big(\dim X_{\rm tri}^\square(\rbar)-d_x\big)+\\
\big(\dim_{k(x)}K(r)+\dim_{k(x)}{\rm Ext}^1_{\rm cris}\big(D_{\rig}(r),D_{\rig}(r)\big)\big)-\\
\big(\dim_{k(x)}K(r)+\dim_{k(x)}{\rm Ext}^1_{\rm cris}\big(D_{\rig}(r),D_{\rig}(r)\big)-\lg(w_x)\big)=\\
\dim X_{\rm tri}^\square(\rbar)-d_x+\lg(w_x).
\end{multline*}
\end{proof}

\begin{coro}\label{bonnedim}
Conjecture \ref{BHS} implies Conjecture \ref{mainconj} for $\rbar=\rhobar_{\tilde v}$ ($v\in S_p$), i.e.:
$$\dim_{k(x)}T_{\widetilde X_{\rm tri}^\square(\rbar),x}=\lg(w_x)-d_{x}+\dim X_{\rm tri}^\square(\rbar).$$
In particular $x$ is smooth on $\widetilde X_{\rm tri}^\square(\rbar)$ if and only if $w_x$ is a product of distinct simple reflections.
\end{coro}
\begin{proof}
It follows from (\ref{injtang}) and Proposition \ref{dimsum} that we have $\lg(w_x)-d_{x}+\dim X_{\rm tri}^\square(\rbar)\leq \dim_{k(x)}T_{\widetilde X_{\rm tri}^\square(\rbar),x}$. The equality follows from Corollary \ref{easybound} which gives the converse inequality. Note that we also deduce $T_{\jmath_{w_x,{\bf k}}\big(\overline{\widetilde U_{\rm tri}^\square(\rbar)\times_{\mathcal{W}^n_L}\mathcal{W}^n_{w,{\bf k},L}}\big),x}+T_{\iota_{\bf k}(\widetilde \Xfrak_{\overline r}^{\square,{\bf k}\rm -cr}),x}\buildrel\sim\over\longrightarrow T_{\widetilde X_{\rm tri}^\square(\rbar),x}$. Finally, as we have already seen, the last statement follows from Lemma \ref{coxeter}.
\end{proof}

We end up with an application of Corollary \ref{bonnedim} (thus assuming Conjecture \ref{BHS}) to the classical eigenvariety $Y(U^p,\rhobar)$ of \S\ref{classic}. We keep the notation and assumptions of \S\ref{classic} and \S\ref{firstclassical} and we consider a point $x\in Y(U^p,\rhobar)$ which is crystalline strictly dominant very regular. In a recent on-going work (\cite{Bergdraft}), Bergdall, inspired by the upper bound in Theorem \ref{upperbound}, proved an analogous upper bound for $\dim_{k(x)}T_{Y(U^p,\rhobar),x}$, and obtained in particular that $Y(U^p,\rhobar)$ is smooth at $x$ when the Weyl group element $w_x$ in (\ref{wx}) is a product of distinct simple reflections and when some Selmer group (which is always conjectured to be zero) vanishes. As a consequence of Corollary \ref{bonnedim} we prove that this should {\it not} remain so when $w_x$ is {\it not} a product of distinct simple reflections.

\begin{coro}\label{singhecke}
Assume Conjecture \ref{BHS} and assume that $w_x$ is {\rm not} a product of distinct simple reflections. Then the eigenvariety $Y(U^p,\rhobar)$ is singular at $x$.
\end{coro}
\begin{proof}
For $v\in S_p$ denote by $x_v$ the image of $x$ in $X_{\rm tri}^\square(\rhobar_{\tilde v})$ via (\ref{eigenvartotrianguline}). Since $Y(U^p,\rhobar)\hookrightarrow X_p(\rhobar)$, we have $x_v\in \widetilde X_{\rm tri}^\square(\rhobar_{\tilde v})$. It follows from Corollary \ref{bonnedim} that it is enough to prove the following: if $Y(U^p,\rhobar)$ is smooth at $x$ then $\widetilde X_{\rm tri}^\square(\rhobar_{\tilde v})$ is smooth at $x_v$ for all $v\in S_p$, or equivalently $\widetilde X_{\rm tri}^\square(\rhobar_p)\cong \prod_{v\in S_p}\widetilde X_{\rm tri}^\square(\rhobar_{\tilde v})$ is smooth at $(x_v)_{v\in S_p}$. Recall from \S\ref{firstclassical} that we have:
\begin{equation}\label{basechange}
Y(U^p,\rhobar)\cong X_p(\rhobar)\times_{(\Spf\, S_\infty)^{\rig}}\Spm\, L
\end{equation}
where the map $S_\infty\twoheadrightarrow L$ is $S_\infty\twoheadrightarrow (S_\infty/{\mathfrak a})[1/p]$ and where $X_p(\rhobar)\rightarrow {\mathfrak X}_\infty\rightarrow (\Spf\, S_\infty)^{\rig}$ is induced by the morphism $S_\infty\rightarrow R_\infty$. Let $\omega_{\infty}(x)$ be the image of $x$ in $(\Spf\, S_\infty)^{\rig}$, by an argument similar to the one in the proof of Proposition \ref{tgtx'} we deduce from (\ref{basechange}):
\begin{equation*}
T_{Y(U^p,\rhobar),x}\cong \left\{\vec{v}\in T_{X_p(\rhobar),x}\ {\rm mapping\ to}\ 0\ {\rm in}\ T_{(\Spf\, S_\infty)^{\rig},\omega_\infty(x)}\otimes_{k(\omega_\infty(x))}k(x)\right\}.
\end{equation*}
This obviously implies:
\begin{equation}\label{tangentxp}
\dim_{k(x)}T_{Y(U^p,\rhobar),x}\geq \dim_{k(x)}T_{X_p(\rhobar),x} - \dim_{k(\omega_\infty(x))}T_{(\Spf\, S_\infty)^{\rig},\omega_\infty(x)}.
\end{equation}
But $\dim_{k(\omega_\infty(x))}T_{(\Spf\, S_\infty)^{\rig},\omega_\infty(x)}=g+[F^+:\Q]\tfrac{n(n-1)}{2}+|S|n^2$ (see beginning of \S\ref{firstclassical}) and $\dim_{k(x)}T_{Y(U^p,\rhobar),x}=\dim Y(U^p,\rhobar)=n[F^+:\Q]$ since $x$ is assumed to be smooth on $Y(U^p,\rhobar)$, hence we deduce from (\ref{tangentxp}):
$$\dim_{k(x)}T_{X_p(\rhobar),x}\leq g+[F^+:\Q]\tfrac{n(n+1)}{2}+|S|n^2=\dim X_p(\rhobar)$$
where the last equality follows from \cite[Cor.3.11]{BHS}. We thus have $\dim_{k(x)}T_{X_p(\rhobar),x}=\dim X_p(\rhobar)$ which implies that $x$ is smooth on $X_p(\rhobar)$, and thus by (i) of Remark \ref{conjvariant} that $(x_v)_{v\in S_p}$ is smooth on $\widetilde X_{\rm tri}^\square(\rhobar_p)$.
\end{proof}

\begin{rema}
{\rm Singular crystalline strictly dominant points on eigenvarieties are already known to exist by \cite[\S6]{Bel}. However, the singular points of {\it loc. cit.} are different from the points $x$ of Corollary \ref{singhecke} since they have reducible associated global Galois representations.}
\end{rema}

\def\cprime{$'$} \def\cprime{$'$} \def\cprime{$'$} \def\cprime{$'$}
\providecommand{\bysame}{\leavevmode ---\ }
\providecommand{\og}{``}
\providecommand{\fg}{''}
\providecommand{\smfandname}{\&}
\providecommand{\smfedsname}{\'eds.}
\providecommand{\smfedname}{\'ed.}
\providecommand{\smfmastersthesisname}{M\'emoire}
\providecommand{\smfphdthesisname}{Th\`ese}

\end{document}